\documentclass[a4paper]{article}
\usepackage{amsmath,amsthm,amsfonts,amssymb,amscd,bm,makeidx,indentfirst,tocloft}

\newcommand{\e}{\ensuremath{\epsilon}}

\newcommand{\kk}{\ensuremath{\tilde{k}}}
\newcommand{\rto}{\ensuremath{\rightarrow}}
\newcommand{\lem}{\ensuremath{\lesssim}}
\newcommand{\dt}{\ensuremath{\delta_{\theta}}}

\newtheorem{theorem}{Theorem}[section]

\newtheorem{lemma}[theorem]{Lemma}

\newtheorem{proposition}[theorem]{Proposition}
\newtheorem{remark}[theorem]{Remark}

\numberwithin{equation}{section}

\title{\Large Well-Posedness of the Nonlinear Unsteady Prandtl Equations with Robin Boundary Condition in Weighted Sobolev Spaces}
\author{\normalsize Fuzhou Wu\thanks{E-mail: michael8723@gmail.com; fuzhou.wu@yahoo.com; wufz12@mails.tsinghua.edu.cn} \\
\small\it  Yau Mathematical Sciences Center, Tsinghua University\\
\small\it  Beijing 100084, China
}
\date{}

\begin{document}
\maketitle
\setlength\parindent{2em}
\setlength\parskip{5pt}

\begin{abstract}
\normalsize{
In this paper, we study the well-posedness of classical solutions to the nonlinear unsteady Prandtl equations with Robin boundary condition in half space in weighted Sobolev spaces. We firstly investigate the monotonic shear flow with Robin boundary condition and the linearized Prandtl-type equations with Robin boundary condition in weighted Sobolev spaces. Due to the degeneracy of the Prandtl equations and the loss of regularity, we apply the Nash-Moser-H$\ddot{o}$rmander iteration scheme to prove the existence of classical solutions to the nonlinear Prandtl equations with Robin boundary condition when the initial data is a small perturbation of a monotonic shear flow satisfying Robin boundary condition. The uniqueness and stability are also proved in the weighted Sobolev spaces. The nonlinear Prandtl equations with Robin boundary arise in the inviscid limit of incompressible Navier-Stokes equations with Navier-slip boundary condition for which the slip length is square root of viscosity. Our results are also valid for the Dirichlet boundary case.
}
\end{abstract}

\tableofcontents

\section{Introduction}
In this paper, we establish the well-posedness of classical solutions to the nonlinear unsteady Prandtl equations with Robin boundary condition in half space:
\begin{equation}\label{Sect1_Prandtl_Equation}
\left\{\begin{array}{ll}
u_t + u u_x + v u_y + p_x = u_{yy},\quad (x,y)\in\mathbb{R}_{+}^2, \ t>0, \\[9pt]
u_x + v_y =0, \\[9pt]
(u_y - \beta u)|_{y=0} = 0,\quad v|_{y=0} =0,\\[9pt]
\lim\limits_{y\rto +\infty}u = U(t,x), \\[10pt]
u|_{t=0} = u_0(x,y),
\end{array}\right.
\end{equation}
where $u,v$ denote the tangential and normal velocities of the boundary layer, with $y$ being the scaled normal variable to the boundary, the parameter $\beta>0$. $U,p$ denote the values on the boundary of the tangential velocity and pressure of the Euler flow which satisfies the Bernoulli's law:
\begin{equation}\label{Sect1_Bernoulli_Law}
\begin{array}{ll}
U_t + U U_x + p_x = 0.
\end{array}
\end{equation}

The nonlinear Prandtl equations with Robin boundary condition are proposed in $\cite{Wang_Wang_Xin_2010}$, which studied the asymptotic behavior of the solutions to incompressible Navier-Stokes equations with Navier-slip boundary condition in which the slip length depends on the viscosity:
\begin{equation}\label{Sect1_NS_Equation}
\left\{\begin{array}{ll}
u_t + u u_x + v u_y + p_x = \e\triangle u, \\[7pt]
v_t + u v_x + v v_y + p_y = \e\triangle v, \\[7pt]
u_x + v_y =0, \\[7pt]
(\e^{\gamma}u_y - \beta u)|_{y=0} = 0,\quad v|_{y=0} =0, \\[7pt]
(u,v)|_{t=0} = (u_0,v_0)(x,y).
\end{array}\right.
\end{equation}
When $\gamma>\frac{1}{2}$ (super-critical), the leading boundary layer profile satisfies the Prandtl equations with Dirichlet boundary condition.
When $\gamma=\frac{1}{2}$ (critical), the leading boundary layer profile satisfies the Prandtl equations with Robin boundary condition $(\ref{Sect1_Prandtl_Equation})$.
When $\gamma<\frac{1}{2}$ (sub-critical), the leading boundary layer profile appears in the $O(\e^{1-2\gamma})$ order terms of the solutions and satisfies the linearized Prandtl equations.

For $(u_y - \beta u)|_{y=0} = 0$, $0<\beta<+\infty$ corresponds to Robin boundary condition.
$\beta=+\infty$ corresponds to Dirichlet boundary condition, since $(u-\frac{1}{\beta}u_y)|_{y=0} =0$.
While $\beta=0$ corresponds to Neumann boundary condition.
To our best knowledge, the Prandtl equations with Neumann boundary condition have no physical background, and their well-posedness is unknown in mathematical viewpoint. In this paper, the parameter $\beta\rto 0+$ is not allowed.

For the nonlinear Prandtl equations, the known results are mainly about the Dirichlet boundary case, where the solutions vanish on the boundary, namely $u|_{y=0}=v|_{y=0}=0$, we survey there some results:

After L. Prandtl (see \cite{Prandtl_1905}) proposed the Prandtl equations with Dirichlet boundary condition, their well-posedness theories attract much attention. Under Oleinik's monotonicity assumption $u_y>0$, the Prandtl equations with Dirichlet boundary condition can be reduced to a single quasilinear equation of $u_y$ via Crocco transformation, then O. A. Oleinik and V. N. Samokhin (see \cite{Oleinik_Samokhin_1999}) proved the local in time well-posedness. Under Oleinik's monotonicity assumption $u_y>0$ and favorable pressure condition $p_x\leq 0$, Xin and Zhang (see \cite{Xin_Zhang_2004}) proved the global existence of BV weak solutions via splitting viscosity method and Crocco transformation.
When the initial data is a small perturbation of a monotonic shear flow and Oleinik's monotonicity assumption is satisfied,
by using the energy method and Nash-Moser-H$\ddot{o}$rmander iteration, Alexander, Wang, Xu and Yang (see \cite{Alexandre_Yang_2014}) proved the well-posedness of the 2D Prandtl equations, Liu, Wang and Yang (see \cite{Liu_Wang_Yang_2014}) proved the well-posedness of the 3D Prandtl equations under constraints on its flow structure. This framework was also used to treat 2D compressible flow, see \cite{Wang_Xie_Yang_2014}.
Under Oleinik's monotonicity assumption, N. Masmoudi and T. K. Wong (see \cite{Masmoudi_Wong_2012}) proved local existence and uniqueness for 2D Prandtl equations in weighted Sobolev spaces via uniform regularity approach.

When Oleinik's monotonicity assumption is violated, E and Engquist (see \cite{E_Engquist_1997}) proved the unsteady Prandtl equations do not have global strong solutions, namely, local solutions either do not exist or blow up; Grenier (see \cite{Grenier_2000}), Hong and Hunter (see \cite{Hong_Hunter_2003}) proved the nonlinear instability of the unsteady Prandtl equations; \cite{Varet_Dormy_2010,Varet_Nguyen_2010,Guo_Nguyen_2011} proved ill-posedness of Prandtl equations in Sobolev spaces for some data or in some weak sense.

Additionally, as to the nonlinear steady Prandtl equations with Dirichlet boundary condition, O. A. Oleinik (see \cite{Oleinik_1963}) used von Mise transformation to prove strong solutions are global in space for favorable pressure $p_x\leq 0$. While for adverse pressure $p_x>0$, boundary layer separation may happen (see \cite{Caffarelli_E_1995}).

Without Oleinik's monotonicity assumption, the data and solutions are required to be in the analytic or Gevrey classes.
For the data that are analytic in both $x$ and $y$ variables, the abstract Cauchy-Kowalewski theorem (see \cite{Safonov_1995}) can be applied, then the local existence of analytic solutions is proved in \cite{Sammartino_Caflisch_1998,Lombardo_Cannone_Sammartino_2003} for the Dirichlet boundary case.
For the data that are analytic in $x$ variable and have Sobolev regularity in $y$ variable, the existence is proved in \cite{Kukavica_Vicol_2013,Zhang_Zhang_2014} by using the energy method. For the data that belong to the Gevrey class $\frac{7}{4}$ in $x$ variable,
D. G$\acute{e}$rard-Varet and N. Masmoudi (see \cite{Varet_Masmoudi_2013}) proved local well-posedness. As to the Gevrey class regularity, see \cite{Li_Wu_Xu_2015}.

For the Prandtl equations with Robin boundary condition $(\ref{Sect1_Prandtl_Equation})$, the only result in the present is in the analytic setting and for the analytic data, the local existence of analytic solutions is proved in \cite{Ding_Jiang_2014} by using the abstract Cauchy-Kowalewski theorem.

While the well-posedness of the Prandtl equations with Robin boundary condition $(\ref{Sect1_Prandtl_Equation})$ in Sobolev spaces has been widely open for some years. There are two main difficulties, one is the degeneracy in $x$ variable of the Prandtl equations, the other is the estimates of the solutions and their tangential derivatives, normal derivatives on the boundary. In this paper, we establish the well-posedness of $(\ref{Sect1_Prandtl_Equation})$ in weighted Sobolev spaces, our approach is to apply the weighted energy estimate method and the Nash-Moser-H$\ddot{o}$rmander iteration scheme, since Crocco transformation is useless for the Prandtl equations with Robin boundary condition. Our initial data is required to be a small perturbation of the monotonic shear flow $(u^s,0)$, where $u^s$ satisfies the following monotonic conditions:
\begin{equation}\label{Sect1_Data_Monotone}
\begin{array}{ll}
u^s>0,\quad \partial_y u^s>0,\quad \beta-\frac{\partial_{yy}u^s}{\partial_y u^s}\geq \delta_s >0,\quad \forall y\in [0,+\infty),\ \forall t\in [0,T].
\end{array}
\end{equation}
Without the smallness of the perturbation around the monotonic shear flow, $u_x$ or $v$ will blow up in the interior of the boundary layer in general.

Letting $w=u_y -\beta u$, the Prandtl system $(\ref{Sect1_Prandtl_Equation})$ can be transformed into a new degenerate system of $w$ with Dirichlet boundary condition $w|_{y=0}=0$. However, the attempt to establish the well-posedness of this new system of $w$ will fail, because the weighted energy estimates never close. Therefore, we will study the Prandtl system $(\ref{Sect1_Prandtl_Equation})$ without using this variable $w=u_y -\beta u$.

In order to prove the existence of classical solutions to the Prandtl system $(\ref{Sect1_Prandtl_Equation})$ and overcome the two main difficulties mentioned above, we use the Nash-Moser-H$\ddot{o}$rmander iteration scheme, while in its iteration process, we need the weighted a priori estimates and well-posedness of the linearized Prandtl-type equations with Robin boundary condition, zero data and nonzero force, see the equations $(\ref{Sect3_Linearized_Prandtl_ZeroData})$.

In order to prove the uniqueness and stability of classical solutions to the Prandtl system $(\ref{Sect1_Prandtl_Equation})$, we need the weighted a priori estimates and well-posedness of the linearized Prandtl-type equations with Robin boundary condition, nonzero data and zero force, see the equations $(\ref{Sect3_Linearized_Prandtl_ZeroForce})$.

Thus, the investigation of the well-posedness of the linearized Prandtl-type equations with Robin boundary condition is an important part of this paper. We need to transform the linearized Prandtl-type equations with Robin boundary condition into the appropriate equations with the appropriate boundary conditions, see the equations $(\ref{Sect3_VorticityEq}),(\ref{Sect3_VorticityEq_ZeroForce})$ respectively, such that the weighted energy estimates can proceed. By coupling the estimates in the interior and the estimates on the boundary, we are able to get the wanted a priori estimates.

Additionally, some mollified variables and mollified quantities arise in the Nash-Moser-H$\ddot{o}$rmander iteration scheme.
In order to make some mollified variables keep some properties, we introduce special extension operators $E^u,\ E^v$ and special smoothing operators $S_{\theta}^u,\ S_{\theta}^v$, except for the usual operators $E,\ S_{\theta}$. However, $S_{\theta}^u,\ S_{\theta}^v$ do not lose regularity.

In the Nash-Moser-H$\ddot{o}$rmander iteration process, the variables and mollified quantities are bounded by the powers of $\theta_n$.
Based on the estimates for these variables and mollified quantities, we can prove the convergence of the Nash-Moser-H$\ddot{o}$rmander iteration, which implies the existence of classical solutions to the nonlinear Prandtl equations with Robin boundary condition.

To start the Nash-Moser-H$\ddot{o}$rmander iteration process, it needs the zero-th order approximate solution which satisfies Robin boundary condition. By using time derivatives and the induction method, we are able to construct the zero-th order approximate solution from the initial data which satisfy Robin boundary condition.

For simplicity, we consider the uniform Euler flow $U=1$, which implies $p$ is a constant, by the Bernoulli's law.
However, the following IBVP keeps the two main difficulties mentioned above.
\begin{equation}\label{Sect1_PrandtlEq}
\left\{\begin{array}{ll}
u_t + u u_x + v u_y = u_{yy},\quad (x,y)\in\mathbb{R}_{+}^2,\ t>0, \\[9pt]
u_x + v_y =0, \\[9pt]
(u_y - \beta u)|_{y=0} = 0,\quad v|_{y=0} =0,\\[9pt]
\lim\limits_{y\rto +\infty}u = 1, \\[9pt]
u|_{t=0} = u_0(x,y).
\end{array}\right.
\end{equation}

\vspace{0.15cm}
In this paper, the time derivatives of initial data $u_0^s$ can be expressed in terms of the space derivatives of $u_0^s$ by solving the heat equation, the time derivatives of initial data $u_0$ can be expressed in terms of the space derivatives of $u_0,v_0$ by solving the Prandtl equations, the time derivatives of $\tilde{u}_0 =u_0 - u_0^s$ can be expressed in terms of the space derivatives of $\tilde{u}_0,u_0^s$ (see $(\ref{Sect4_Zeroth_Order_Sol_2})_1$).
Now we define the following functional spaces:
\begin{equation}\label{Sect1_Energy_Definition}
\begin{array}{ll}
\|u\|_{\mathcal{A}_{\ell}^k(\Omega_T)} =
\bigg(\sum\limits_{0\leq k_1 + [\frac{k_2+1}{2}]\leq k}
\|<y>^{\ell} \partial_{(t,x)}^{k_1}\partial_y^{k_2} u\|_{L^2([0,T]\times \mathbb{R}_{+}^2)}^2
\bigg)^{\frac{1}{2}}, \\[16pt]

\big\|u|_{t=t_1}\big\|_{\mathcal{A}_{\ell}^k(\Omega,t=t_1)} =
\bigg(\sum\limits_{0\leq k_1 + [\frac{k_2+1}{2}]\leq k}
\|<y>^{\ell} \partial_{(t,x)}^{k_1}\partial_y^{k_2} u|_{t=t_1}\|_{L^2(\mathbb{R}_{+}^2)}^2
\bigg)^{\frac{1}{2}}, \\[16pt]

\big\|u|_{y=0}\big\|_{A^k([0,T]\times \partial\Omega)} =
\bigg(\sum\limits_{0\leq m\leq k}
\|\partial_{(t,x)}^{m} u|_{y=0}\|_{L_{t,x}^2([0,T]\times \mathbb{R})}^2
\bigg)^{\frac{1}{2}}, \\[15pt]

\big\|u|_{y=0,t=t_1}\big\|_{A^k(\partial\Omega,t=t_1)} =
\bigg(\sum\limits_{0\leq m\leq k}
\|\partial_{(t,x)}^{m} u|_{y=0,t=t_1}\|_{L_{x}^2(\mathbb{R})}^2
\bigg)^{\frac{1}{2}}, \\[15pt]

\|u\|_{\mathcal{B}_{\lambda,\ell}^{k_1,k_2}(\Omega_T)} =
\bigg(\sum\limits_{0\leq m\leq k_1, 0\leq q\leq k_2}
\|e^{-\lambda t}<y>^{\ell} \partial_{(t,x)}^{m}\partial_y^{q} u\|_{L^2([0,T]\times \mathbb{R}_{+}^2)}^2
\bigg)^{\frac{1}{2}},
\end{array}
\end{equation}

\begin{equation*}
\begin{array}{ll}
\|u\|_{\tilde{\mathcal{B}}_{\lambda,\ell}^{k_1,k_2}(\Omega)} =
\bigg(\sum\limits_{0\leq m\leq k_1, 0\leq q\leq k_2}
\|e^{-\lambda t}<y>^{\ell} \partial_{(t,x)}^{m}\partial_y^{q} u\|_{L_t^{\infty}([0,T]; L_{x,y}^2(\mathbb{R}_{+}^2))}^2
\bigg)^{\frac{1}{2}}, \\[15pt]

\|u\|_{\mathcal{C}_{\ell}^k(\Omega_T)} =
\sum\limits_{0\leq k_1 + [\frac{k_2+1}{2}]\leq k}
\|<y>^{\ell} \partial_{(t,x)}^{k_1}\partial_y^{k_2} u\|_{L_y^2(L_{t,x}^{\infty})}, \\[15pt]

\|v\|_{\mathcal{D}_{\ell}^k(\Omega_T)} =
\sum\limits_{0\leq k_1 + [\frac{k_2+1}{2}]\leq k}
\|<y>^{\ell} \partial_{(t,x)}^{k_1}\partial_y^{k_2} v\|_{L_y^{\infty}(L_{t,x}^2)},
\end{array}
\end{equation*}
where $\Omega=\mathbb{R}_{+}^2,\Omega_T = [0,T]\times \mathbb{R}_{+}^2,\partial\Omega=\{(x,0)|x\in\mathbb{R}\}$, $<y>=\sqrt{1+|y|^2}$,
$0<T<+\infty$, $\ell>\frac{1}{2},\ \lambda >0$
and $k,k_1,k_2$ are non-negative integers. The homogeneous norms $\|\cdot\|_{\dot{\mathcal{A}}_{\ell}^k},\, \|\cdot\|_{\dot{A}^k},\, \|\cdot\|_{\dot{\mathcal{C}}_{\ell}^k},\, \|\cdot\|_{\dot{\mathcal{D}}_{\ell}^k}$
correspond to the summation $1\leq k_1 + [\frac{k_2 +1}{2}]\leq k$ in the definitions.
$\|f\|_{L_{\ell}^p(\mathbb{R}_{+}^2)} = \|<y>^{\ell} f\|_{L^p(\mathbb{R}_{+}^2)}$, where $1\leq p\leq +\infty$.

\vspace{0.3cm}
Noting the sense of the time derivatives of initial data for the Prandtl equations and the heat equation, we state the main results of this paper as follows:
\begin{theorem}\label{Sect1_Main_Thm}
Concerning the nonlinear unsteady Prandtl equations with Robin boundary condition $(\ref{Sect1_PrandtlEq})$,
giving any integer $k\geq 5$ and real number $\ell>\frac{1}{2}$,
we have the following existence, uniqueness and stability results.

$(1)$. For any $\delta_{\beta}>0$, $\delta_{\beta}\leq\beta<+\infty$, assume the initial data $u_0(x,y) = u_0^s(y)
+\tilde{u}_0(x,y)$ satisfies the following two conditions:

$(i)$ $u_0^s$ satisfies
\begin{equation}\label{Sect1_Condition1}
\left\{\begin{array}{ll}
u_0^s(y)>0,\quad \partial_y u_0^s(y)>0,\quad \beta -\frac{\partial_{yy}u_0^s(y)}{\partial_y u_0^s(y)}\geq \delta_{s,0}>0,\quad \forall y\in [0,+\infty), \\[13pt]
\partial_y^{2j}(\partial_y u_0^s(y) - \beta u_0^s(y))|_{y=0} =0,
\quad \forall 0\leq j\leq 2k+10, \\[12pt]

\lim\limits_{y\rto +\infty} u_0^s(y) =1, \\[12pt]

\|u_0^s-1\|_{L^2} + |u_0^s|_{\infty} + \|u_0^s\|_{\dot{\mathcal{C}}_{\ell}^{2k+11}} + \|\frac{\partial_{yy}u_0^s}{\partial_y u_0^s}\|_{\mathcal{C}_{\ell}^{2k+10}}\leq C,
\end{array}\right.
\end{equation}
for a fixed constant $C>0$.

$(ii)$ There exists a small constant $\e>0$ such that $\tilde{u}_0 = u_0 -u_0^s$ satisfies
\begin{equation}\label{Sect1_Condition2}
\left\{\begin{array}{ll}
\partial_y^{2j}(\partial_y \tilde{u}_0(x,y) - \beta \tilde{u}_0(x,y))|_{y=0} =0,
\quad \forall 0\leq j\leq 2k+8, \\[11pt]
\lim\limits_{y\rto +\infty} \tilde{u}_0(x,y) =0, \\[11pt]

\|\tilde{u}_0\|_{\mathcal{A}_{\ell}^{2k+9}(\mathbb{R}_{+}^2,t=0)}
+ \|\frac{\partial_y \tilde{u}_0}{\partial_y u_0^s}\|_{\mathcal{A}_{\ell}^{2k+9}(\mathbb{R}_{+}^2,t=0)}
 \leq \e.
\end{array}\right.
\end{equation}

Then there exists $T\in (0,+\infty)$, such that the Prandtl system $(\ref{Sect1_PrandtlEq})$ admits a unique classical solution $(u, v)$ satisfying
\begin{equation}\label{Sect1_Solution_Monotonicity}
\begin{array}{ll}
u>0,\quad \partial_y u>0,\quad \beta-\frac{\partial_{yy}u}{\partial_y u}\geq\delta>0,
\end{array}
\end{equation}
and
\begin{equation}\label{Sect1_Solution_Regularity}
\begin{array}{ll}
u-u^s \in \mathcal{A}_{\ell}^k([0,T]\times\mathbb{R}_{+}^2), \hspace{1cm}
\partial_y(u-u^s),\ \frac{\partial_y(u-u^s)}{\partial_y u^s} \in \mathcal{A}_{\ell}^k([0,T]\times\mathbb{R}_{+}^2), \\[10pt]

v \in \mathcal{D}_0^{k-1}([0,T]\times\mathbb{R}_{+}^2), \hspace{1.46cm}
\partial_y v,\ \partial_{yy} v \in \mathcal{A}_{\ell}^{k-1}([0,T]\times\mathbb{R}_{+}^2), \\[10pt]

\partial_y^{j} u|_{y=0}- \partial_y^{j}u^s|_{y=0} \in A^{k-[\frac{j+1}{2}]}([0,T]\times\mathbb{R}), \hspace{0.3cm} 0\leq j\leq 2k, \\[10pt]

\partial_y^{j+1} v|_{y=0} \in A^{k-1-[\frac{j+1}{2}]}([0,T]\times\mathbb{R}), \hspace{0.3cm} 0\leq j\leq 2k-2.
\end{array}
\end{equation}

$(2)$. For any $\delta_{\beta}>0$, $\delta_{\beta}\leq\beta<+\infty$, the classical solution to $(\ref{Sect1_PrandtlEq})$ is stable with respect to the initial data in the following sense: for any given two initial data
\begin{equation*}
\begin{array}{ll}
u_0^1 = u_0^s + \tilde{u}_0^1,\qquad u_0^2 = u_0^s + \tilde{u}_0^2,
\end{array}
\end{equation*}
if $u_0^s$ satisfies $(\ref{Sect1_Condition1})$ and $\tilde{u}_0^1, \tilde{u}_0^2$ satisfy $(\ref{Sect1_Condition2})$, then for all $p\leq k-2$, the corresponding solutions $(u^1,v^1)$ and $(u^2,v^2)$ of the Prandtl system $(\ref{Sect1_PrandtlEq})$ satisfy
\begin{equation}\label{Sect1_Stability}
\begin{array}{ll}
\|u^1-u^2\|_{\mathcal{A}_{\ell}^p([0,T]\times\mathbb{R}_{+}^2)}
+ \|\frac{\partial_y(u^1- u^2)}{\partial_y u^s}\|_{\mathcal{A}_{\ell}^p([0,T]\times\mathbb{R}_{+}^2)} \\[8pt]\quad
+\|v^1-v^2\|_{\mathcal{D}_0^{p-1}([0,T]\times\mathbb{R}_{+}^2)}
+ \sum\limits_{j=0}^{2p}\big\|\partial_y^j u^1|_{y=0} - \partial_y^j u^2|_{y=0} \big\|_{A^{p-[\frac{j+1}{2}]}([0,T]\times\mathbb{R})}
 \\[15pt]

\leq C(T,\e,u_0^s)\Big\|\partial_y (\frac{u_0^1-u_0^2}{\partial_y(u_0^1 + u_0^2)})\Big\|_{\mathcal{A}_{\ell}^{p}(\mathbb{R}_{+}^2,t=0)} \\[15pt]\quad
+ \frac{C(T,\e,u_0^s)}{\max\{\sqrt{\beta -C_{\eta}},\sqrt{\delta}\}} \Big\|\partial_y (\frac{u_0^1-u_0^2}{\partial_y(u_0^1 + u_0^2)})|_{y=0}\Big\|_{A^p(\mathbb{R},t=0)} \ .
\end{array}
\end{equation}

$(3)$. As $\beta\rto +\infty$, $\partial_t^j u^s|_{y=0}=O(\frac{1}{\beta}),\ 0\leq j\leq k$, $\big\|u|_{y=0}-u^s|_{y=0}\big\|_{A^{k}} =O(\frac{1}{\sqrt{\beta}})$ and $(u,v)$ satisfies $(\ref{Sect1_Solution_Regularity})$ uniformly.
When $\beta= +\infty$, $(u,v)$ satisfies $(\ref{Sect1_Solution_Regularity})$ and for $p\leq k-2$,
\begin{equation}\label{Sect1_Solution_Stable_Dirichlet}
\begin{array}{ll}
\|u^1-u^2\|_{\mathcal{A}_{\ell}^p([0,T]\times\mathbb{R}_{+}^2)}
+ \|\frac{\partial_y(u^1- u^2)}{\partial_y u^s}\|_{\mathcal{A}_{\ell}^p([0,T]\times\mathbb{R}_{+}^2)}
+\|v^1-v^2\|_{\mathcal{D}_0^{p-1}([0,T]\times\mathbb{R}_{+}^2)} \\[8pt]\quad
+ \sum\limits_{j=0}^{2p}\big\|\partial_y^j u^1|_{y=0} - \partial_y^j u^2|_{y=0} \big\|_{A^{p-[\frac{j+1}{2}]}([0,T]\times\mathbb{R})}
\leq C\Big\|\partial_y (\frac{u_0^1-u_0^2}{\partial_y(u_0^1 + u_0^2)})\Big\|_{\mathcal{A}_{\ell}^{p}(\mathbb{R}_{+}^2,t=0)} \ .
\end{array}
\end{equation}
\end{theorem}

Next, we give some remarks on the results in Theorem $\ref{Sect1_Main_Thm}$:
\begin{remark}\label{Sect1_Remark}
(i). Though the Robin boundary condition is nontrivial, the Robin boundary case loses the same $k+9$ orders of regularity as the Dirichlet boundary case.
$T$ depends on $\e$ and whether the monotonicity conditions $(\ref{Sect1_Solution_Monotonicity})$ are violated.
If we only consider the Prandtl system $(\ref{Sect1_PrandtlEq})$ in a sufficiently short time interval,
$\beta -\frac{\partial_{yy}u_0^s}{\partial_y u_0^s}\geq \delta_{s,0}$ in $(\ref{Sect1_Condition1})$ can be confined on the boundary, and then
$\beta-\frac{\partial_{yy}u}{\partial_y u}\geq\delta$ in $(\ref{Sect1_Solution_Monotonicity})$ is satisfied on the boundary.

(ii). In the t,x-directions, the regularities and stability results can not be improved.
In the y-direction, the solutions have lower regularities on the boundary than in the interior.
When $\beta<+\infty$, due to the Robin boundary condition, $\partial_y u|_{y=0} -\partial_y u^s|_{y=0} \in A^k([0,T]\times\mathbb{R})$ and
$\big\|\partial_y u|_{y=0} -\partial_y u^s|_{y=0}\big\|_{A^p}$ is stable.

(iii). When $\beta\rto +\infty$, $(\ref{Sect1_Solution_Monotonicity})$ implies Oleinik's monotonicity assumption.
$(\ref{Sect1_Data_Monotone})$ and $(\ref{Sect1_Solution_Monotonicity})$ do not allow $\beta\rto 0+$, otherwise\, $u_y \leq \beta u|_{y=0}e^{(\beta-\delta)y}\rto 0$, but $u_y>0$. In $(\ref{Sect1_Condition2})$, $\e$ is small such that $\beta-\frac{\partial_{yy} u_0}{\partial_y u_0}|_{y=0} \geq\delta>0$, no degeneracy arises on the boundary, which is necessary to get the stability results $(\ref{Sect1_Stability})$.

(iv). If $\beta\neq +\infty$, $(\ref{Sect1_Solution_Stable_Dirichlet})$ does not hold.
$(\ref{Sect1_Solution_Stable_Dirichlet})$ implies the uniqueness and stability of solutions in the Dirichlet boundary case.
When $\beta= +\infty$, the compatibility conditions for the initial data become $\partial_y^{2j} u_0^s|_{y=0} =0,\ 0\leq j\leq 2k+10$ and $\partial_y^{2j} \tilde{u}_0|_{y=0} =0,\ 0\leq j\leq 2k+8$ for the Dirichlet boundary case.
\end{remark}

We introduce the generic constants and notations used in this paper: \\[4pt]
$C$:\hspace{0.1cm} the generic positive constant which may be different line by line. \\[2pt]
$C_{\ell}$: the positive constant which depends on $\int\limits_0^{\infty}<y>^{-2\ell}\,\mathrm{d}y$, where $\ell>\frac{1}{2}$. \\
$C_{\varrho}$: the positive constant which depends on the cut-off function $\varrho$. \\[3pt]
$C_s$: the positive constant relating to the shear flow $(u^s,0)$. \\[2pt]
$C_{\eta}$: the positive constant which bounds $|\eta|_{\infty}$. \\[2pt]
$C_k$:\hspace{0.1cm} the generic positive constant which appears in the $k$-th step of the Nash- \\[1pt]\indent
                     Moser-H$\ddot{o}$rmander iteration, is independent of the index $k$ and may be  \\[1pt]\indent
                     different line by line. \\[2pt]
$f\lem g$: there exists a constant $C>0$ such that $f\leq Cg$.

\vspace{0.2cm}
The rest of the paper is organized as follows: In Section 2, we investigate the shear flow with Robin boundary condition. In Section 3, we study the well-posedness of the linearized Prandtl equations with Robin boundary condition. In Section 4, we construct the approximate solutions to the nonlinear Prandtl equations with Robin boundary condition via the Nash-Moser-H$\ddot{o}$rmander iteration scheme. In Section 5, we prove the convergence of the Nash-Moser-H$\ddot{o}$rmander iteration and obtain the existence of classical solutions to the Prandtl equations with Robin boundary condition. In Section 6, we prove the uniqueness and stability of classical solutions to the Prandtl equations with Robin boundary condition.

\section{Well-Posedness of Shear Flow with Robin Boundary Condition}
In this section, we investigate the shear flow with Robin boundary condition where the tangential velocity is monotonic in the normal variable. The tangential velocity of the shear flow satisfies heat equation with Robin boundary condition. Let $(u^s(t,y),0)$ be the shear flow of the Prandtl equations with Robin boundary condition $(\ref{Sect1_PrandtlEq})$, then $u^s(t,y)$ satisfies the following IBVP:
\begin{equation}\label{Sect2_HeatEq}
\left\{\begin{array}{ll}
u_t^s = u_{yy}^s,\quad y>0,\ t>0, \\[8pt]
(u_y^s - \beta u^s)|_{y=0} = 0, \\[8pt]
\lim\limits_{y\rto +\infty}u^s = 1, \\[8pt]
u^s|_{t=0} = u_0^s(y).
\end{array}\right.
\end{equation}

Noting the sense of the time derivatives of initial data for the heat equation, we state the following theorem for $(\ref{Sect2_HeatEq})$:
\begin{theorem}\label{Sect2_Theorem}
Assume for any $k\geq 3,\ 0<\delta_{\beta}\leq\beta<+\infty,\ \ell>\frac{1}{2}$, the initial data $u_0^s(y)$ satisfies \\[3pt]
$(i)$ the monotonicity conditions:
\begin{equation}\label{Sect2_Data_Monotone}
\begin{array}{ll}
u_0^s(y) >0,\quad  \partial_y u_0^s(y) >0,\quad \beta -\frac{\partial_{yy} u_0^s(y)}{\partial_y u_0^s(y)} >0, \quad  \forall y\in [0,+\infty).
\end{array}
\end{equation}

\noindent
$(ii)$ the compatibility conditions:
\begin{equation}\label{Sect2_Data_Compatible}
\begin{array}{ll}
\lim\limits_{y\rto +\infty} u_0^s(y) =1,\quad \partial_y^{2j}(\partial_y u_0^s(y) - \beta u_0^s(y))|_{y=0} =0,
\quad \forall\ 0\leq j\leq k-1.
\end{array}
\end{equation}

\noindent
$(iii)$ the boundedness condition:
\begin{equation}\label{Sect2_Data_Condition}
\begin{array}{ll}
\|u_0^s-1\|_{L^2} + \|u_0^s\|_{\dot{\mathcal{C}}_{\ell}^k} + \|\frac{\partial_{yy}u_0^s}{\partial_y u_0^s}\|_{\mathcal{C}_{\ell}^{k-1}}\leq C,
\end{array}
\end{equation}
for some constant $C>0$.

Then for any fixed\, $T\in (0,+\infty)$, the problem $(\ref{Sect2_HeatEq})$ admits a unique classical solution $u^s$ satisfying the monotonicity:
\begin{equation}\label{Sect2_Solution_Monotone}
\begin{array}{ll}
u^s >0, \quad  \partial_y u^s >0,\quad \beta -\frac{u_{yy}^s}{u_y^s} >0,\hspace{0.9cm}  \forall y\in [0,+\infty),\ t\in [0,T], \\[10pt]
u^s|_{y=0} >0, \quad  \partial_y u^s|_{y=0} >0,\quad (\partial_y u^s -\beta u^s)|_{y=0} =0,\quad \forall t\in [0,T],
\end{array}
\end{equation}
and the estimates:
\begin{equation}\label{Sect2_Solution_Regularity}
\begin{array}{ll}
\|u^s-1\|_{L^2} +\|u^s\|_{\dot{\mathcal{C}}_0^k} 
+ \beta\sum\limits_{j=0}^{k-1}\big|\partial_t^{j} u^s|_{y=0}\big|_{\infty}
+ \sum\limits_{j=0}^{k-1}\big|\partial_t^{j}\partial_y u^s|_{y=0}\big|_{\infty} \\[12pt]\quad

+ \|\frac{\partial_{yy}u^s}{\partial_y u^s}\|_{\mathcal{C}_0^{k-1}}
+ \|\frac{\partial_{yt}u^s}{\partial_y u^s}\|_{\mathcal{C}_0^{k-2}}
+ \|\frac{\partial_{yyt}u^s}{\partial_y u^s}\|_{\mathcal{C}_0^{k-2}}
+ \sum\limits_{m=0}^{k-2}\big|\partial_t^m \frac{\partial_{yy}u^s}{\partial_y u^s}|_{y=0}\big|_{\infty} \\[12pt]\quad

+ \sum\limits_{m=0}^{k-2}\big|\partial_t^m \frac{\partial_{yt}u^s}{\partial_y u^s}|_{y=0}\big|_{\infty}
+ \sum\limits_{m=0}^{k-3}\big|\partial_t^m \frac{\partial_{yyt}u^s}{\partial_y u^s}|_{y=0}\big|_{\infty}
\leq C(T),\quad \forall t\in [0,T],
\end{array}
\end{equation}
for some constant $C(T)>0$.

Moreover, if $\partial_{yy} u_0^s \leq 0$, then $\partial_{yy} u^s \leq 0$.
If $\beta -\frac{\partial_{yy} u_0^s}{\partial_y u_0^s} \geq\delta_{s,0}>0$, then $\beta -\frac{\partial_{yy} u^s}{\partial_y u^s} \geq \delta_s >0$.
\end{theorem}

\begin{proof}
We prove $(\ref{Sect2_Solution_Monotone})$ first.
Let $w^1(t,y)=u_y^s(t,y) - \beta u^s(t,y)$, then $w^1$ satisfies the following equations:
\begin{equation}\label{Sect2_HeatEq_Transform}
\left\{\begin{array}{ll}
w_t^1 = w_{yy}^1,\quad t>0,\ y>0, \\[9pt]
w^1|_{y=0} = 0,\quad \lim\limits_{y\rto +\infty}w^1 = -\beta, \\[11pt]
w^1|_{t=0} = w_0^1(y) := \partial_y u_0^s(y) - \beta u_0^s(y).
\end{array}\right.
\end{equation}

The solution of $(\ref{Sect2_HeatEq_Transform})$ can be written as
\begin{equation}\label{Sect2_HeatEq_Transform_Solution}
\begin{array}{ll}
w^1(t,y) = \frac{1}{2\sqrt{\pi t}}\int\limits_0^{+\infty}
\Big( e^{-\frac{(y-\tilde{y})^2}{4t}} - e^{-\frac{(y+\tilde{y})^2}{4t}} \Big)
w_0^1(\tilde{y}) \,\mathrm{d}\tilde{y} <0,
\end{array}
\end{equation}
where $w_0^1(y)<0$ comes from $w_0^1|_{y=0}=0$ and $\beta -\frac{\partial_{yy} u_0^s}{\partial_y u_0^s} >0$.

Integrate $(u e^{-\beta y})_y = w^1 e^{-\beta y}$ and note that $\lim\limits_{y\rto +\infty} u e^{-\beta y} =0$, then 
the explicit solution of $(\ref{Sect2_HeatEq})$ is written as
\begin{equation}\label{Sect2_HeatEq_Solution}
\begin{array}{ll}
u^s(t,y) = e^{\beta y} \int\limits_{+\infty}^{y} e^{-\beta \hat{y}} w^1(t,\hat{y}) \,\mathrm{d}\hat{y}
\\[12pt]\hspace{1.15cm}

= \frac{e^{\beta y}}{2\sqrt{\pi t}} \int\limits_{+\infty}^{y} e^{-\beta \hat{y}} \int\limits_0^{+\infty}
\Big( e^{-\frac{(\hat{y}-\tilde{y})^2}{4t}} - e^{-\frac{(\hat{y}+\tilde{y})^2}{4t}} \Big)
w_0^1(\tilde{y}) \,\mathrm{d}\tilde{y} \,\mathrm{d}\hat{y}.
\end{array}
\end{equation}

Since $-\beta\leq w_0^1(y)\leq 0$, it is easy to check that the integral in $(\ref{Sect2_HeatEq_Solution})$ converges for any $t\in (0,T]$ after replacing $\frac{\tilde{y}\pm\hat{y}}{\sqrt{t}}$ into a new variable, $\lim\limits_{t\rto 0}u^s(t,y) = u_0^s(y)$ due to $\lim\limits_{t\rto 0}w^1(t,y) = w^1_0(y)$,
then $u^s \in L^{\infty}([0,T]\times\mathbb{R}_{+})$. Thus, the existence and uniqueness of the solutions to $(\ref{Sect2_HeatEq})$ are obtained.

Note that
\begin{equation}\label{Sect2_Heat_Kernel}
\left\{\begin{array}{ll}
e^{-\frac{(y-\tilde{y})^2}{4t}} - e^{-\frac{(y+\tilde{y})^2}{4t}} >0,\quad \forall y>0, \\[6pt]
e^{-\frac{(y-\tilde{y})^2}{4t}} - e^{-\frac{(y+\tilde{y})^2}{4t}} =0,\quad y=0.
\end{array}\right.
\end{equation}
Then it follows from $(\ref{Sect2_HeatEq_Transform_Solution})$ that
\begin{equation}\label{Sect2_W_Monotone}
\left\{\begin{array}{ll}
w^1 = \partial_y u^s -\beta u^s <0,\quad \forall y>0, \\[7pt]
(\partial_y u^s - \beta u^s)|_{y=0}=0.
\end{array}\right.
\end{equation}

Note that $(\ref{Sect2_Heat_Kernel})$ and $w_0(y)<0,\ \forall y>0$, it follows from $(\ref{Sect2_HeatEq_Solution})$ that
\begin{equation*}
\left\{\begin{array}{ll}
u^s|_{y=0} = \frac{1}{2\sqrt{\pi t}} \int\limits_{+\infty}^0 e^{-\beta \hat{y}} \int\limits_0^{+\infty}
\Big( e^{-\frac{(\hat{y}-\tilde{y})^2}{4t}} - e^{-\frac{(\hat{y}+\tilde{y})^2}{4t}} \Big)
w_0^1(\tilde{y}) \,\mathrm{d}\tilde{y} \,\mathrm{d}\hat{y}>0, \\[13pt]

u^s(t,y) = \frac{e^{\beta y}}{2\sqrt{\pi t}} \int\limits_{+\infty}^{y} e^{-\beta \hat{y}} \int\limits_0^{+\infty}
\Big( e^{-\frac{(\hat{y}-\tilde{y})^2}{4t}} - e^{-\frac{(\hat{y}+\tilde{y})^2}{4t}} \Big)
w_0^1(\tilde{y}) \,\mathrm{d}\tilde{y} \,\mathrm{d}\hat{y}>0.
\end{array}\right.
\end{equation*}
then
$\partial_y u^s|_{y=0} = \beta u^s|_{y=0} >0$, \ $\lim\limits_{\beta\rto +\infty}u^s|_{y=0}= 0$.

In order to obtain the monotonicity of $\partial_y w^1$, we transform $(\ref{Sect2_HeatEq_Transform_Solution})$ into
\begin{equation}\label{Sect2_Maximum_2}
\begin{array}{ll}
w^1 = \frac{1}{\sqrt{\pi}} \Big(
\int\limits_{-\frac{y}{2\sqrt{t}}}^{+\infty} e^{-\xi^2} w_0^1(2\sqrt{t}\xi + y) \,\mathrm{d}\xi
- \int\limits_{\frac{y}{2\sqrt{t}}}^{+\infty} e^{-\xi^2} w_0^1(2\sqrt{t}\xi - y) \,\mathrm{d}\xi
\Big).
\end{array}
\end{equation}

\vspace{-0.25cm}
Apply $\partial_y$ to $(\ref{Sect2_Maximum_2})$ and note that $\partial_y w_0^1<0$, we get $\partial_y w^1(t)<0,\ \forall y\in [0,+\infty)$, due to
\begin{equation}\label{Sect2_Maximum_3}
\begin{array}{ll}
\partial_y w^1 = \frac{1}{\sqrt{\pi}} \Big(
\int\limits_{-\frac{y}{2\sqrt{t}}}^{+\infty} e^{-\xi^2} \partial_y w_0^1(2\sqrt{t}\xi + y) \,\mathrm{d}\xi
+ \int\limits_{\frac{y}{2\sqrt{t}}}^{+\infty} e^{-\xi^2} \partial_y w_0^1(2\sqrt{t}\xi - y) \,\mathrm{d}\xi
\Big).
\end{array}
\end{equation}

\vspace{-0.3cm}
Apply $\partial_y$ to $(\ref{Sect2_HeatEq_Solution})$, we get
\begin{equation}\label{Sect2_Maximum_4}
\begin{array}{ll}
\partial_y u^s(t,y) = \frac{\beta e^{\beta y}}{2\sqrt{\pi t}} \int\limits_{+\infty}^{y} e^{-\beta \hat{y}} \int\limits_0^{+\infty}
\Big( e^{-\frac{(\hat{y}-\tilde{y})^2}{4t}} - e^{-\frac{(\hat{y}+\tilde{y})^2}{4t}} \Big)
w_0^1(\tilde{y}) \,\mathrm{d}\tilde{y} \,\mathrm{d}\hat{y}
\end{array}
\end{equation}

\begin{equation*}
\begin{array}{ll}
\hspace{1.9cm}
+\frac{1}{2\sqrt{\pi t}} \int\limits_0^{+\infty}
\Big( e^{-\frac{(y-\tilde{y})^2}{4t}} - e^{-\frac{(y+\tilde{y})^2}{4t}} \Big)
w_0^1(\tilde{y}) \,\mathrm{d}\tilde{y} \hspace{2.3cm}
\\[13pt]\hspace{1.5cm}

= \beta e^{\beta y} \int\limits_{+\infty}^{y} e^{-\beta \hat{y}} w^1(t,\hat{y}) \,\mathrm{d}\hat{y}
+ w^1(t,y).
\end{array}
\end{equation*}

\vspace{-0.2cm}
Denote $H(t,y) := e^{-\beta y} \partial_y u^s(t,y) = \beta \int\limits_{+\infty}^{y} e^{-\beta \hat{y}} w^1(t,\hat{y}) \,\mathrm{d}\hat{y} + e^{-\beta y} w^1(t,y)$.
Since $H(t,+\infty)=0$ and $\partial_y H(t,y) = e^{-\beta y} \partial_y w^1 <0$, we get $H(t,y)>0,\ \forall y\in [0,+\infty)$.
Thus, $\partial_y u^s(t,y)>0,\ \forall y\in [0,+\infty)$.

Apply $\partial_y$ to $w^1 =\partial_y u^s -\beta u^s$, we get
\begin{equation}\label{Sect2_Maximum_5}
\begin{array}{ll}
\partial_{yy} u^s(t,y) - \beta \partial_y u^s(t,y)= \partial_y w^1(t,y)<0.
\end{array}
\end{equation}

Therefore, $u^s>0,\ \partial_y u^s>0,\ \beta-\frac{\partial_{yy} u^s}{\partial_y u^s}>0,\ \forall y\in [0,+\infty)$ is proved.
$(\partial_y u^s - \beta u^s)|_{y=0}=0$ is proved.

\vspace{0.3cm}
Next, we prove the a priori estimates $(\ref{Sect2_Solution_Regularity})$.

When $1\leq j\leq k$, apply $\partial_t^{j}$ to $(\ref{Sect2_HeatEq})_1$, and multiply with
$e^{-2\lambda t}<y>^{2\ell}\partial_t^{j}u$, then integrate in $[0,t]\times\mathbb{R}_{+}$, we get
\begin{equation}\label{Sect2_Energy_Estimate_1}
\begin{array}{ll}
\int\limits_0^{\infty} | e^{-\lambda t}<y>^{\ell}\partial_t^{j} u^s|^2 \,\mathrm{d}y
+ 2\lambda \int\limits_0^t \int\limits_0^{\infty} |
e^{-\lambda \tilde{t}}<y>^{\ell}\partial_t^{j} u^s|^2 \,\mathrm{d}y \,\mathrm{d}\tilde{t} \\[10pt]\quad

+ 2\int\limits_0^t \int\limits_0^{\infty} | e^{-\lambda \tilde{t}}<y>^{\ell}\partial_t^{j}\partial_{y} u^s|^2 \,\mathrm{d}y \,\mathrm{d}\tilde{t}
\\[11pt]
= \int\limits_0^{\infty} |<y>^{\ell}\partial_t^{j} u_0^s|^2 \,\mathrm{d}y
- 4\ell \int\limits_0^t \int\limits_0^{\infty}  e^{-2\lambda \tilde{t}} y<y>^{2\ell-2}\partial_t^{j} u^s
\partial_t^{j}\partial_{y} u^s\,\mathrm{d}y \,\mathrm{d}\tilde{t} \\[11pt]\quad
- 2\int\limits_0^t e^{-2\lambda \tilde{t}}\partial_t^{j} u^s|_{y=0}
\partial_t^{j}\partial_{y} u^s|_{y=0} \,\mathrm{d}\tilde{t} \\[10pt]

= \int\limits_0^{\infty} |<y>^{\ell}\partial_t^{j} u_0^s|^2 \,\mathrm{d}y
- 4\ell \int\limits_0^t \int\limits_0^{\infty}  e^{-2\lambda \tilde{t}} y<y>^{2\ell-2}\partial_t^{j} u^s
\partial_t^{j}\partial_{y} u^s\,\mathrm{d}y \,\mathrm{d}\tilde{t} \hspace{0.9cm}
\\[11pt]\quad
- 2\beta\int\limits_0^t \big|e^{-\lambda t}\partial_t^{j} u^s|_{y=0} \big|^2 \,\mathrm{d}\tilde{t}.
\end{array}
\end{equation}

Thus, take $\lambda>0$ to be large enough and note that $T<+\infty$, we have
\begin{equation}\label{Sect2_Energy_Estimate_2}
\begin{array}{ll}
\int\limits_0^{\infty} |<y>^{\ell}\partial_t^{j} u^s|^2 \,\mathrm{d}y
+ \lambda \int\limits_0^t \int\limits_0^{\infty} |<y>^{\ell}\partial_t^{j} u^s|^2 \,\mathrm{d}y \,\mathrm{d}\tilde{t} \\[10pt]\quad

+ \int\limits_0^t \int\limits_0^{\infty} |<y>^{\ell}\partial_t^{j}\partial_{y} u^s|^2 \,\mathrm{d}y \,\mathrm{d}\tilde{t}
+ 2\beta\int\limits_0^t \big|\partial_t^{j} u^s|_{y=0} \big|^2 \,\mathrm{d}\tilde{t}
\\[11pt]

\leq C(\lambda T)\int\limits_0^{\infty} |<y>^{\ell}\partial_t^{j} u_0^s|^2 \,\mathrm{d}y
\leq C(\lambda T)\|u^s_0\|_{\dot{\mathcal{C}}_{\ell}^{j}} \leq C(T),\quad 1\leq j\leq k.
\end{array}
\end{equation}

Since $\partial_t^{j}\partial_y u^s|_{y=0} = \int\limits_{+\infty}^0 \partial_t^{j}\partial_{yy} u^s(t,\tilde{y})\,\mathrm{d}\tilde{y}
= \int\limits_{+\infty}^0 \partial_t^{j+1} u^s(t,\tilde{y})\,\mathrm{d}\tilde{y}$, then
\begin{equation}\label{Sect2_Energy_Estimate_4}
\begin{array}{ll}
\big|\partial_t^{j}\partial_y u^s|_{y=0}\big|_{\infty}
\lem \sup\limits_{0\leq t\leq T}\int\limits_0^{+\infty} |\partial_t^{j+1} u^s|\,\mathrm{d}\tilde{y}
\lem C_{\ell}\big(\sup\limits_{0\leq t\leq T}\int\limits_0^{+\infty}<\tilde{y}>^{2\ell} |\partial_t^{j+1} u^s|^2\,\mathrm{d}\tilde{y}\big)^{\frac{1}{2}} \\[13pt]\hspace{2.35cm}
\lem \big(\int\limits_0^{\infty} |<y>^{\ell}\partial_t^{j+1} u_0^s|^2 \,\mathrm{d}y \big)^{\frac{1}{2}} \leq C(T), \quad 0\leq j\leq k-1, \\[15pt]

\int\limits_0^t\big|\partial_t^{j}\partial_y u^s|_{y=0}\big|^2 \,\mathrm{d}\tilde{t}
\lem \int\limits_0^t \Big|\int\limits_{+\infty}^0 \partial_t^{j+1} u^s\,\mathrm{d}\tilde{y}\Big|^2 \mathrm{d}\tilde{t}
\lem C_{\ell}\int\limits_0^t\int\limits_0^{+\infty} |<\tilde{y}>^{\ell}\partial_t^{j+1} u^s|^2\,\mathrm{d}\tilde{y}\,\mathrm{d}\tilde{t} \\[13pt]\hspace{2.83cm}
\lem C_{\ell}\int\limits_0^{\infty} |<\tilde{y}>^{\ell}\partial_t^{j+1} u_0^s|^2 \,\mathrm{d}y \leq C(T), \quad  0\leq j\leq k-1.
\end{array}
\end{equation}

For any $t\in [0,T]$, we have the $L^2$ estimate for $(1- u^s)_t = (1- u^s)_{yy}$\ :
\begin{equation}\label{Sect2_L2_Estimate_0}
\begin{array}{ll}
\frac{1}{2}\partial_t \|1- u^s\|_{L^2(\mathbb{R}_{+})}^2 = \int\limits_0^{+\infty} (1-u^s)_{yy}(1-u^s) \,\mathrm{d}y \\[10pt]\hspace{2.85cm}
= \partial_y u^s|_{y=0} (1- u^s|_{y=0}) - \int\limits_0^{+\infty} |u_y^s|^2 \,\mathrm{d}y \\[10pt]\hspace{2.85cm}
\leq  -\beta \big|u^s|_{y=0}\big|^2 - \int\limits_0^{+\infty} |u_y^s|^2 \,\mathrm{d}y + \partial_y u^s|_{y=0} \, ,
\end{array}
\end{equation}
then by $(\ref{Sect2_Energy_Estimate_4})$, we get
\begin{equation}\label{Sect2_L2_Estimate_1}
\begin{array}{ll}
\|u^s -1\|_{L^2(\mathbb{R_{+}})}^2 + 2\int\limits_0^t \int\limits_0^{+\infty} |u_y^s|^2 \,\mathrm{d}y\mathrm{d}t
+ 2\beta \int\limits_0^t \big|u^s|_{y=0}\big|^2 \,\mathrm{d}t \\[10pt]
\leq \|u_0^s -1\|_{L^2(\mathbb{R_{+}})}^2 + 2\int\limits_0^t\big|\partial_y u^s|_{y=0}\big|_{\infty} \,\mathrm{d}t \leq C(T).
\end{array}
\end{equation}

Denote $w^2 =\partial_y u^s$, then $w^2$ satisfies the following equation:
\begin{equation}\label{Sect2_W2_1}
\left\{\begin{array}{ll}
\partial_t w^2 =\partial_{yy} w^2,\quad y>0,\ t>0, \\[8pt]
(\partial_t w^2 -\beta\partial_y w^2)|_{y=0} =0,\quad \lim\limits_{y\rto +\infty}w^2 =0, \\[8pt]
w^2|_{t=0} = w^2_0(y) := \partial_y u_0^s(y).
\end{array}\right.
\end{equation}

When $0\leq j\leq k-1$, apply $\partial_t^{j}$ to $(\ref{Sect2_W2_1})_1$, and multiply with
$e^{-2\lambda t}<y>^{2\ell}\partial_t^{j}w^2$, then integrate in $[0,t]\times\mathbb{R}_{+}$, we get
\begin{equation}\label{Sect2_W2_2}
\begin{array}{ll}
\int\limits_0^{\infty} | e^{-\lambda t}<y>^{\ell}\partial_t^{j} w^2|^2 \,\mathrm{d}y
+ 2\lambda \int\limits_0^t \int\limits_0^{\infty} |
e^{-\lambda \tilde{t}}<y>^{\ell}\partial_t^{j} w^2|^2 \,\mathrm{d}y \,\mathrm{d}\tilde{t} \\[10pt]\quad

+ 2\int\limits_0^t \int\limits_0^{\infty} | e^{-\lambda \tilde{t}}<y>^{\ell}\partial_t^{j}\partial_{y} w^2|^2 \,\mathrm{d}y \,\mathrm{d}\tilde{t}
\\[11pt]
= \int\limits_0^{\infty} |<y>^{\ell}\partial_t^{j} w^2_0|^2 \,\mathrm{d}y
- 4\ell \int\limits_0^t \int\limits_0^{\infty}  e^{-2\lambda \tilde{t}} y<y>^{2\ell-2}\partial_t^{j} w^2
\partial_t^{j}\partial_{y} w^2\,\mathrm{d}y \,\mathrm{d}\tilde{t} \\[11pt]\quad
- 2\int\limits_0^t e^{-2\lambda \tilde{t}}\partial_t^{j} w^2|_{y=0}
\partial_t^{j}\partial_{y} w^2|_{y=0} \,\mathrm{d}\tilde{t}
\end{array}
\end{equation}

\begin{equation*}
\begin{array}{ll}
= \int\limits_0^{\infty} |<y>^{\ell}\partial_t^{j} w^2_0|^2 \,\mathrm{d}y
- 4\ell \int\limits_0^t \int\limits_0^{\infty}  e^{-2\lambda \tilde{t}} y<y>^{2\ell-2}\partial_t^{j} w^2
\partial_t^{j}\partial_{y} w^2\,\mathrm{d}y \,\mathrm{d}\tilde{t} \\[11pt]\quad
- \frac{2\lambda}{\beta}\int\limits_0^t \big|e^{-\lambda t}\partial_t^{j} w^2|_{y=0} \big|^2 \,\mathrm{d}\tilde{t}
-\frac{1}{\beta}e^{-2\lambda t}\big|\partial_t^{j}w^2|_{y=0}\big|^2 + \frac{1}{\beta}\big|\partial_t^{j}w^2_0|_{y=0}\big|^2.
\end{array}
\end{equation*}

Thus, take $\lambda>0$ to be large enough and note that $T<+\infty$, we have
\begin{equation}\label{Sect2_W2_3}
\begin{array}{ll}
\int\limits_0^{\infty} |<y>^{\ell}\partial_t^{j} w^2|^2 \,\mathrm{d}y
+ \frac{2\lambda}{\beta}\int\limits_0^t \big|\partial_t^{j} w^2|_{y=0} \big|^2 \,\mathrm{d}\tilde{t}
+ \frac{1}{\beta}\big|\partial_t^{j}w^2|_{y=0}\big|_{\infty}^2 \\[10pt]\quad

+ \lambda \int\limits_0^t \int\limits_0^{\infty} |<y>^{\ell}\partial_t^{j} w^2|^2 \,\mathrm{d}y \,\mathrm{d}\tilde{t}
+ \int\limits_0^t \int\limits_0^{\infty} |<y>^{\ell}\partial_t^{j}\partial_{y} w^2|^2 \,\mathrm{d}y \,\mathrm{d}\tilde{t}
\\[11pt]

\leq C(\lambda T)\int\limits_0^{\infty} |<y>^{\ell}\partial_t^{j} w^2_0|_{y=0}|^2 \,\mathrm{d}y
 + C(\lambda T)\frac{1}{\beta}\big|\partial_t^{j}w^2_0|_{y=0}\big|_{\infty}^2 \\[11pt]

= C(\lambda T)\int\limits_0^{\infty} |<y>^{\ell}\partial_t^{j} w^2_0|^2 \,\mathrm{d}y
 + C(\lambda T)\big|\partial_t^{j} u_0^s|_{y=0} \big|_{\infty}^2,
\end{array}
\end{equation}

Since $\partial_t^{j} u_0^s|_{y=0} = \int\limits_{+\infty}^0 \partial_t^{j}\partial_{y} u_0^s(t,\tilde{y})\,\mathrm{d}\tilde{y}$, then
for $0\leq j\leq k-1$,
\begin{equation}\label{Sect2_W2_4}
\begin{array}{ll}
\big|\partial_t^{j} u_0^s|_{y=0}\big|_{\infty}
\lem \int\limits_0^{+\infty} |\partial_t^{j}\partial_{y} u_0^s|\,\mathrm{d}\tilde{y}
\lem C_{\ell}\big(\int\limits_0^{+\infty}<\tilde{y}>^{2\ell} |\partial_t^{j}\partial_{y} u_0^s|^2\,\mathrm{d}\tilde{y}\big)^{\frac{1}{2}}.
\end{array}
\end{equation}

Then for $0\leq j\leq k-1$,
\begin{equation}\label{Sect2_W2_5}
\begin{array}{ll}
\int\limits_0^{\infty} |<y>^{\ell}\partial_t^{j} \partial_y u^s|^2 \,\mathrm{d}y
+ \frac{2\lambda}{\beta}\int\limits_0^t \big|\partial_t^{j} \partial_y u^s|_{y=0} \big|^2 \,\mathrm{d}\tilde{t}
+ \frac{1}{\beta}\big|\partial_t^{j}\partial_y u^s|_{y=0}\big|_{\infty}^2 \\[10pt]

\leq C(\lambda T)\int\limits_0^{\infty} |<y>^{\ell}\partial_t^{j} \partial_y u^s_0|^2 \,\mathrm{d}y
\leq C(\lambda T)\|u^s_0\|_{\dot{\mathcal{C}}_{\ell}^{j+1}} \leq C(T).
\end{array}
\end{equation}

Since $\partial_t^{j} u^s|_{y=0} = \int\limits_{+\infty}^0 \partial_t^{j}\partial_{y} u^s(t,\tilde{y})\,\mathrm{d}\tilde{y}$, then
\begin{equation}\label{Sect2_W2_6}
\begin{array}{ll}
\big|\partial_t^{j} u^s|_{y=0}\big|_{\infty}
\lem \sup\limits_{0\leq t\leq T}\int\limits_0^{+\infty} |\partial_t^{j}\partial_{y} u^s|\,\mathrm{d}\tilde{y}
\lem C_{\ell}\big(\sup\limits_{0\leq t\leq T}\int\limits_0^{+\infty}<\tilde{y}>^{2\ell} |\partial_t^{j}\partial_{y} u^s|^2\,\mathrm{d}\tilde{y}\big)^{\frac{1}{2}} \\[12pt]\hspace{1.95cm}
\lem C(\lambda T)\|u^s_0\|_{\dot{\mathcal{C}}_{\ell}^{j+1}} \leq C(T), \quad 0\leq j\leq k-1.
\end{array}
\end{equation}

Thus, $\|u^s-1\|_{L^2}^2 +\|u^s\|_{\dot{\mathcal{C}}_{0}^k}^2
+ \beta\sum\limits_{j=0}^k\int\limits_0^t \big|\partial_t^{j} u^s|_{y=0} \big|^2 \,\mathrm{d}\tilde{t}
+ \sum\limits_{j=0}^{k-1}\int\limits_0^t\big|\partial_t^{j}\partial_y u^s|_{y=0}\big|^2 \,\mathrm{d}\tilde{t} \\[2pt]
+ \sum\limits_{j=0}^{k-1}\big|\partial_t^{j} u^s|_{y=0}\big|_{\infty}^2
+ \sum\limits_{j=0}^{k-1}\big|\partial_t^{j}\partial_y u^s|_{y=0}\big|_{\infty}^2
\leq \|u_0^s-1\|_{L^2}^2 + C(T)\|u_0^s\|_{\dot{\mathcal{C}}_{\ell}^k}^2 \leq C(T)$.

\vspace{0.2cm}
Define the following variables:
\begin{equation}\label{Sect2_Define_Variables}
\begin{array}{ll}
\alpha(t,y) := \frac{u^s_{yy}(t,y)}{u^s_y(t,y)}, \qquad
\alpha_1(t,y) := \frac{u^s_{yt}(t,y)}{u^s_y(t,y)}, \qquad
\alpha_2(t,y) := \frac{u^s_{yyt}(t,y)}{u^s_y(t,y)}.
\end{array}
\end{equation}

Then $\alpha(t,y)$ satisfies the following semilinear system:
\begin{equation}\label{Sect2_Alpha_System_1}
\left\{\begin{array}{ll}
\partial_t\alpha = \partial_{yy}\alpha + \partial_y(\alpha^2), \quad y>0,\ t>0,\\[8pt]
\partial_y\alpha + \alpha^2 =\beta\alpha,\quad y=0, \\[8pt]
\alpha|_{t=0} = \alpha^0(y) := \frac{\partial_{yy}u_0^s(y)}{\partial_y u_0^s(y)}.
\end{array}\right.
\end{equation}

$\alpha_1(t,y)$ satisfies the following linear system:
\begin{equation}\label{Sect2_Alpha_System_2}
\left\{\begin{array}{ll}
\partial_t\alpha_1 = \partial_{yy}\alpha_1 + 2\alpha\partial_y \alpha_1, \quad y>0,\ t>0,\\[8pt]
\partial_y \alpha_1 = \frac{1}{\beta}\partial_t\alpha_1,\quad y=0, \\[8pt]
\alpha_1|_{t=0} = \alpha_1^0(y) := \frac{\partial_{yyy}u_0^s(y)}{\partial_y u_0^s(y)}.
\end{array}\right.
\end{equation}

For the systems $(\ref{Sect2_Alpha_System_1})$ and $(\ref{Sect2_Alpha_System_2})$,
it is easy to get the following estimates when $\lambda>0$ is large enough:
\begin{equation}\label{Sect2_Alpha_System_Estimates_1}
\begin{array}{ll}
\sum\limits_{k_1=0}^{k-1}\int\limits_0^{\infty} <y>^{2\ell}|e^{-\lambda t}\partial_t^{k_1} \alpha|^2\,\mathrm{d}y
+ \sum\limits_{k_1=0}^{k-1}\int\limits_0^T \int\limits_0^{\infty}
 <y>^{2\ell}|e^{-\lambda t}\partial_t^{k_1} \partial_y\alpha|^2\,\mathrm{d}y\,\mathrm{d}t \\[11pt]\quad

+ \lambda\sum\limits_{k_1=0}^{k-1}\int\limits_0^T \int\limits_0^{\infty}
 <y>^{2\ell}|e^{-\lambda t}\partial_t^{k_1} \alpha|^2\,\mathrm{d}y\,\mathrm{d}t
+ 2\beta\sum\limits_{k_1=0}^{k-1}\int\limits_0^T \big|e^{-\lambda t}\partial_t^{k_1} \alpha|_{y=0}\big|^2 \,\mathrm{d}t \\[11pt]
\lem \sum\limits_{k_1=0}^{k-1}\int\limits_0^{\infty} <y>^{2\ell}|\partial_t^{k_1} \alpha^0|^2\,\mathrm{d}y
\lem \|\alpha^0 \|_{\mathcal{C}_{\ell}^{k-1}}^2 \leq C(T), \\[23pt]

\sum\limits_{k_1=0}^{k-2}\int\limits_0^{\infty} <y>^{2\ell}|e^{-\lambda t}\partial_t^{k_1} \alpha_1|^2\,\mathrm{d}y
+ \frac{1}{\beta}\sum\limits_{k_1=0}^{k-2}\big|e^{-\lambda t}\partial_t^{k_1} \alpha_1|_{y=0}\big|^2\\[11pt]
\lem \sum\limits_{k_1=0}^{k-2}\int\limits_0^{\infty} <y>^{2\ell}|\partial_t^{k_1} \alpha_1^0|^2\,\mathrm{d}y
+ \frac{1}{\beta}\sum\limits_{k_1=0}^{k-2}\big|\partial_t^{k_1} \alpha_1^0|_{y=0}\big|^2 \\[11pt]\quad
+ \sum\limits_{k_1=0}^{k-2}\int\limits_0^T \int\limits_0^{\infty}
 <y>^{2\ell}|e^{-\lambda t}\partial_t^{k_1} \alpha|^2\,\mathrm{d}y\,\mathrm{d}t 
\\[11pt]

\lem \|\alpha_1^0 \|_{\mathcal{C}_{\ell}^{k-2}}^2
+ \sum\limits_{k_1=0}^{k-2}\big|\partial_t^{k_1} \alpha^0|_{y=0}\big|^2 +C(T) \hspace{1.3cm}

\\[10pt]
\lem \|\alpha^0 \|_{\mathcal{C}_{\ell}^{k-1}}^2
+ C_{\ell}\sum\limits_{k_1=0}^{k-2}\int\limits_0^{\infty} <y>^{2\ell} |\partial_t^{k_1}\partial_y \alpha^0|^2\,\mathrm{d}y +C(T) \leq C(T),
\end{array}
\end{equation}
where $\partial_t\alpha^0 = \partial_{yy}\alpha^0 + 2\alpha^0\partial_y\alpha^0$, $\partial_t\alpha_1^0 = \partial_{yy}\alpha_1^0 + 2\alpha^0\partial_y\alpha_1^0$ and $\partial_t^{k_1}\alpha^0,\ \partial_t^{k_1}\alpha_1^0$ are defined by the induction method.
$\alpha$ and $\alpha_1$ have the relationship:
$\partial_t^{k_1}\alpha_1|_{y=0} =\beta\partial_t^{k_1}\alpha|_{y=0},\ 
\alpha_1 = \partial_y\alpha + \alpha^2,\ \partial_y \alpha_1 =\partial_{yy}\alpha + 2\alpha\partial_y\alpha$.

Fix $\lambda>0$ and note that $T<+\infty$, then we have
\begin{equation}\label{Sect2_Alpha_System_Estimates_2}
\begin{array}{ll}
\sum\limits_{k_1=0}^{k-1}\int\limits_0^{\infty} <y>^{2\ell}|\partial_t^{k_1} \alpha|^2\,\mathrm{d}y
\leq C(T), \\[13pt]

\sum\limits_{k_1=0}^{k-2}\int\limits_0^{\infty} <y>^{2\ell}|\partial_t^{k_1} \alpha_1|^2\,\mathrm{d}y
\leq C(T).
\end{array}
\end{equation}

Based on $(\ref{Sect2_Alpha_System_Estimates_2})$ and $\partial_y\alpha = \alpha_1 - \alpha^2$, we get
\begin{equation}\label{Sect2_Alpha_System_Estimates_3}
\begin{array}{ll}
\sum\limits_{k_1=0}^{k-2}\int\limits_0^{\infty} <y>^{2\ell}|\partial_t^{k_1}\partial_y \alpha|^2\,\mathrm{d}y
\leq C(T).
\end{array}
\end{equation}

By using the induction method and $\partial_{yy}\alpha = \alpha_t - 2\alpha\partial_y \alpha$, we get
\begin{equation}\label{Sect2_Alpha_System_Estimates_4}
\begin{array}{ll}
\|\alpha\|_{\mathcal{C}_0^{k-1}} =
\|\frac{u^s_{yy}}{u^s_y}\|_{\mathcal{C}_0^{k-1}} \leq C_{\ell}\|\frac{u^s_{yy}}{u^s_y}\|_{\mathcal{C}_{\ell}^{k-1}}
= C_{\ell}\|\alpha\|_{\mathcal{C}_{\ell}^{k-1}} \leq C(T).
\end{array}
\end{equation}

Since $\alpha_1 = \partial_y\alpha + \alpha^2$, it is easy to have
\begin{equation}\label{Sect2_Alpha_System_Estimates_5}
\begin{array}{ll}
\|\alpha_1\|_{\mathcal{C}_0^{k-2}} =
\|\frac{u^s_{yt}}{u^s_y}\|_{\mathcal{C}_0^{k-2}} \leq C_{\ell}\|\frac{u^s_{yt}}{u^s_y}\|_{\mathcal{C}_{\ell}^{k-2}}
\leq C\|\alpha\|_{\mathcal{C}_{\ell}^{k-1}}\leq C(T).
\end{array}
\end{equation}

Since $\alpha_2 = \partial_y\alpha_1 + \alpha \alpha_1 = \partial_{yy}\alpha + 3\alpha\partial_y\alpha +\alpha^3$, it is easy to have
\begin{equation}\label{Sect2_Alpha_System_Estimates_6}
\begin{array}{ll}
\|\alpha_2\|_{\mathcal{C}_0^{k-2}} =
\|\frac{u^s_{yyt}}{u^s_y}\|_{\mathcal{C}_0^{k-2}} \leq C_{\ell}\|\frac{u^s_{yyt}}{u^s_y}\|_{\mathcal{C}_{\ell}^{k-2}}
\leq C\|\alpha\|_{\mathcal{C}_{\ell}^{k-1}}\leq C(T).
\end{array}
\end{equation}

On the boundary $\{y=0\}$,
\begin{equation}\label{Sect2_Alpha_System_Estimates_8}
\begin{array}{ll}
\sum\limits_{m=0}^{k-2}\big|\partial_t^m\alpha|_{y=0}\big|_{\infty}\lem \|\partial_y\alpha\|_{\mathcal{C}_{\ell}^{k-2}}
\lem \|\alpha\|_{\mathcal{C}_{\ell}^{k-1}} \leq C(T), \\[13pt]

\sum\limits_{m=0}^{k-2}\big|\partial_t^m\alpha_1|_{y=0}\big|_{\infty}\lem \|\partial_y\alpha_1\|_{\mathcal{C}_{\ell}^{k-2}}
\lem \|\alpha\|_{\mathcal{C}_{\ell}^{k-1}} \leq C(T), \\[13pt]

\sum\limits_{m=0}^{k-3}\big|\partial_t^m\alpha_2|_{y=0}\big|_{\infty}\lem \|\partial_y\alpha_2\|_{\mathcal{C}_{\ell}^{k-3}}
\lem \|\alpha\|_{\mathcal{C}_{\ell}^{k-1}} \leq C(T).
\end{array}
\end{equation}

\vspace{0.2cm}
Finally, we prove the last part of Theorem $\ref{Sect2_Theorem}$.
If $\partial_{yy} u_0^s \leq 0$, we have $\partial_{yy} u^s \leq 0$ because $w^3=\partial_{yy} u^s$ satisfies the heat equation with Robin boundary condition:
\begin{equation*}
\left\{\begin{array}{ll}
\partial_t w^3 =\partial_{yy} w^3,\quad y>0,\ t>0, \\[9pt]
(\partial_y w^3 -\beta w^3)|_{y=0} =0, \quad \lim\limits_{y\rto +\infty}w^3 =0, \\[9pt]
w^3|_{t=0} = \partial_{yy} u_0^s.
\end{array}\right.
\end{equation*}
The proof of $w^3\leq 0$ is similar to the proof of $w^1\leq 0$.

If $\beta -\frac{\partial_{yy} u_0^s}{\partial_y u_0^s} \geq\delta_{s,0}\geq\beta$, i.e., $\partial_{yy} u_0^s\leq 0$, then $\partial_{yy} u^s\leq 0$,
so $\beta -\frac{\partial_{yy} u^s}{\partial_y u^s} \geq\beta$. Otherwise, $0<\delta_{s,0}<\beta$, $\alpha^0\leq \beta-\delta_{s,0}$.
Apply the maximum principle to the equation $(\ref{Sect2_Alpha_System_1})$, we get
$\max\{\alpha\} \leq \max\{\alpha^0, \alpha|_{y=0}\}$.

We investigate $\partial_y \alpha|_{y=0} = \alpha|_{y=0}\big(\beta-\alpha|_{y=0}\big)$ where $\big(\beta-\alpha|_{y=0}\big)>0$.
If $\alpha|_{y=0,t=t_1}>0$, then $\alpha$ never reach its maximum at $(t,y)=(t_1,0)$. Thus,
$\alpha\leq\max\{\max\{\alpha^0\},\alpha|_{\{y=0,\alpha(t,0)\leq 0\}}\}\leq\max\{\max\{\alpha^0\},0\}$. Equivalently, $\beta -\alpha\geq \min\{\beta,\delta_{s,0}\}\geq \min\{\delta_{\beta},\delta_{s,0}\}>0$. 

Let $\delta_s =\min\{\delta_{\beta},\delta_{s,0}\}$. Thus, the last part of Theorem $\ref{Sect2_Theorem}$ is proved.
\end{proof}

\section{Well-posedness of the Linearized Prandtl-type Equations with Robin Boundary Condition}
In this section, we investigate the well-posedness of the linearized Prandtl-type equations with Robin boundary condition in the weighted Sobolev spaces.
There are two cases: Case I is zero data and nonzero force, Case II is zero force and nonzero data.

\subsection{Case I: Zero Data, Nonzero Force}
Let $(\tilde{u},\tilde{v})$ be a smooth background state satisfying
\begin{equation}\label{Sect3_Background_Condition_1}
\left\{\begin{array}{ll}
\tilde{u}>0,\quad \tilde{u}_y>0,\quad \beta -\frac{\tilde{u}_{yy}}{\tilde{u}_y} \geq\delta>0,\quad \forall y\in [0,+\infty), \\[7pt]
\tilde{u}_x + \tilde{v}_y =0, \\[7pt]
\tilde{v}|_{y=0}=0, \\[7pt]
\lim\limits_{y\rto +\infty} \tilde{u}(t,x,y)=1,
\end{array}\right.
\end{equation}
where $(\tilde{u}_y -\beta\tilde{u})|_{y=0} =0$ is unnecessary to be satisfied and will not be used here.

We consider the following linearized Prandtl-type equations with Robin boundary condition around the background state $(\tilde{u},\tilde{v})$,
with zero data and nonzero force:
\begin{equation}\label{Sect3_Linearized_Prandtl_ZeroData}
\left\{\begin{array}{ll}
u_t + \tilde{u} u_x + \tilde{v} u_y + u \tilde{u}_x + v \tilde{u}_y - u_{yy} = f, \\[9pt]
u_x + v_y = 0, \\[8pt]
(u_y -\beta u)|_{y=0} = 0,\quad v|_{y=0} = 0, \\[9pt]
\lim\limits_{y\rto +\infty} u(t,x,y) =0, \\[9pt]
u|_{t\leq 0} =0.
\end{array}\right.
\end{equation}

In order to eliminate $v$ in $(\ref{Sect3_Linearized_Prandtl_ZeroData})_1$, we define the following variables:
\begin{equation}\label{Sect3_New_Variables}
\begin{array}{ll}
w = (\frac{u}{\tilde{u}_y})_y,\hspace{0.6cm} \eta = \frac{\tilde{u}_{yy}}{\tilde{u}_y},\hspace{0.65cm}
\bar{\eta} = \frac{u_{yy}^s}{\tilde{u}_y},\hspace{0.6cm} \tilde{f} = \frac{f}{\tilde{u}_y}, \\[11pt]

\zeta = \frac{(\partial_t + \tilde{u}\partial_x + \tilde{v}\partial_y -\partial_{yy})\tilde{u}_y}{\tilde{u}_y},\hspace{0.6cm}

\tilde{\zeta}_1 = \frac{u_{yyt}^s}{\tilde{u}_y} - \bar{\eta}\frac{u_{yt}^s}{\tilde{u}_y},
\\[11pt]

\tilde{\zeta}_2 = \frac{(\tilde{u}_{yyt} -u_{yyt}^s)}{\tilde{u}_y}
+ \frac{\tilde{u}\tilde{u}_{yyx}}{\tilde{u}_y}
-\eta\frac{\tilde{u}_{yt}}{\tilde{u}_y}
+ \bar{\eta}\frac{u_{yt}^s}{\tilde{u}_y}
- \eta \frac{\tilde{u} \tilde{u}_{yx}}{\tilde{u}_y}, \\[11pt]

\tilde{\zeta} :=\tilde{\zeta}_1 + \tilde{\zeta}_2
= \frac{\tilde{u}_{yyt} + \tilde{u}\tilde{u}_{yyx} - \eta(\tilde{u}_{yt} + \tilde{u} \tilde{u}_{yx})}{\tilde{u}_y}.
\end{array}
\end{equation}

$w = (\frac{u}{\tilde{u}_y})_y$ and $\lim\limits_{y\rto +\infty} u(t,x,y) = 0$ imply that
\begin{equation}\label{Sect3_Linearized_Prandtl_Transform_5}
\begin{array}{ll}
u= -\tilde{u}_y \int\limits_y^{+\infty} w(t,x,\tilde{y}) \,\mathrm{d}\tilde{y},\\[8pt]
\frac{u}{\tilde{u}_y} = - \int\limits_y^{+\infty} w(t,x,\tilde{y}) \,\mathrm{d}\tilde{y}.
\end{array}
\end{equation}

On the boundary $\{y=0\}$, $w|_{y=0}$ and $u|_{y=0}$ have the relationship:
\begin{equation*}
\begin{array}{ll}
w|_{y=0} = (\frac{u}{\tilde{u}_y})_y|_{y=0} = \frac{u_y \tilde{u}_y -u \tilde{u}_{yy}}{(\tilde{u}_y)^2}|_{y=0}
= \frac{\beta u \tilde{u}_y -u \tilde{u}_{yy}}{(\tilde{u}_y)^2}|_{y=0}
= \frac{u}{\tilde{u}_y}|_{y=0} (\beta -\eta|_{y=0}).
\end{array}
\end{equation*}
Namely,
\begin{equation}\label{Sect3_Linearized_Prandtl_BC_2}
\begin{array}{ll}
u|_{y=0} = \frac{\tilde{u}_y}{\beta-\eta}|_{y=0} \, w|_{y=0}, \\[9pt]
\frac{u}{\tilde{u}_y}|_{y=0} = \frac{w}{\beta-\eta}|_{y=0}.
\end{array}
\end{equation}

\vspace{0.2cm}
The following lemma gives the equations and boundary condition that $w = (\frac{u}{\tilde{u}_y})_y$ satisfies:
\begin{lemma}\label{Sect3_Equation_Lemma}
$\ w$ satisfies the following IBVP:
\begin{equation}\label{Sect3_VorticityEq}
\left\{\begin{array}{ll}
w_t + (\tilde{u}w)_x + (\tilde{v}w)_y - 2(\eta w)_y
- \big(\zeta \int\limits_y^{+\infty} w(t,x,\tilde{y}) \,\mathrm{d}\tilde{y}\big)_y - w_{yy} = \partial_y\tilde{f}, \\[14pt]

\frac{w_t}{\beta- \eta|_{y=0}} + \frac{(\tilde{u}w)_x}{\beta- \eta|_{y=0}}
- \Big(w_y + 2\eta|_{y=0} w + \tilde{f} + \zeta|_{y=0}\int\limits_0^{+\infty} w(t,x,\tilde{y}) \,\mathrm{d}\tilde{y}\Big) \\[10pt]\hspace{0.5cm}

+ \frac{w\tilde{\zeta}|_{y=0}}{(\beta-\eta|_{y=0})^2} =0,\quad y=0, \\[14pt]

w|_{t\leq 0} =0.
\end{array}\right.
\end{equation}
\end{lemma}

\begin{proof}
Firstly, we determine the equations for $w$.

We have the following transformations:
\begin{equation}\label{Sect3_Linearized_Prandtl_Transform_1}
\begin{array}{ll}
u_t + \tilde{u} u_x + \tilde{v} u_y + u \tilde{u}_x + v \tilde{u}_y - u_{yy} = f, \\[11pt]

\frac{u_t}{\tilde{u}_y} + \frac{\tilde{u} u_x}{\tilde{u}_y} + \frac{\tilde{v} u_y}{\tilde{u}_y}
+ \frac{u \tilde{u}_x}{\tilde{u}_y} + v - \frac{u_{yy}}{\tilde{u}_y}
= \tilde{f}, \\[14pt]

(\frac{u}{\tilde{u}_y})_t + \tilde{u} (\frac{u}{\tilde{u}_y})_x + \tilde{v} (\frac{u}{\tilde{u}_y})_y
+ u\frac{\tilde{u}_x}{\tilde{u}_y} + v - \frac{u_{yy}}{\tilde{u}_y} \\[7pt]
= \tilde{f} - \frac{u\tilde{u}_{yt}}{(\tilde{u}_y)^2} - \frac{\tilde{u}u\tilde{u}_{yx}}{(\tilde{u}_y)^2}
- \frac{\tilde{v}u\tilde{u}_{yy}}{(\tilde{u}_y)^2},
\end{array}
\end{equation}

while
\begin{equation}\label{Sect3_Linearized_Prandtl_Transform_2}
\begin{array}{ll}
\frac{u_{yy}}{\tilde{u}_y} = (\frac{u_y}{\tilde{u}_y})_y + \frac{u_y\tilde{u}_{yy}}{(\tilde{u}_y)^2}
= (\frac{u}{\tilde{u}_y})_{yy} + \frac{2 u_y\tilde{u}_{yy}}{(\tilde{u}_y)^2}
+ u\Big(\frac{\tilde{u}_{yy}}{(\tilde{u}_y)^2}\Big)_y ,
\end{array}
\end{equation}

Then we have
\begin{equation}\label{Sect3_Linearized_Prandtl_Transform_3}
\begin{array}{ll}
(\frac{u}{\tilde{u}_y})_t + \tilde{u} (\frac{u}{\tilde{u}_y})_x + \tilde{v} (\frac{u}{\tilde{u}_y})_y
+ u\frac{\tilde{u}_x}{\tilde{u}_y} + v - (\frac{u}{\tilde{u}_y})_{yy} \\[10pt]
= \tilde{f} - \frac{u\tilde{u}_{yt}}{(\tilde{u}_y)^2} - \frac{\tilde{u}u\tilde{u}_{yx}}{(\tilde{u}_y)^2}
- \frac{\tilde{v}u\tilde{u}_{yy}}{(\tilde{u}_y)^2} + \frac{2 u_y\tilde{u}_{yy}}{(\tilde{u}_y)^2}
+ u\Big(\frac{\tilde{u}_{yy}}{(\tilde{u}_y)^2}\Big)_y \\[10pt]
= \tilde{f} - \frac{u}{(\tilde{u}_y)^2}
(\partial_t + \tilde{u}\partial_x + \tilde{v}\partial_y - \partial_{yy})\tilde{u}_y
- \frac{2u \tilde{u}_{yy}^2}{(\tilde{u}_y)^3}
+ \frac{2 u_y\tilde{u}_{yy}}{(\tilde{u}_y)^2}\\[10pt]

= \tilde{f} - \frac{u}{\tilde{u}_y}\zeta + \frac{2 \tilde{u}_{yy}}{\tilde{u}_y}
(\frac{u}{\tilde{u}_y})_y
= \tilde{f} - \frac{u}{\tilde{u}_y}\zeta + 2\eta w.
\end{array}
\end{equation}

Apply $\partial_y$ to $(\ref{Sect3_Linearized_Prandtl_Transform_3})$ and note that
$\tilde{u}(\frac{u}{\tilde{u}_y})_{x} + \tilde{v}(\frac{u}{\tilde{u}_y})_{y}
=(\tilde{u}\frac{u}{\tilde{u}_y})_{x} + (\tilde{v}\frac{u}{\tilde{u}_y})_{y}$, we get
\begin{equation}\label{Sect3_Linearized_Prandtl_Transform_4}
\begin{array}{ll}
(\frac{u}{\tilde{u}_y})_{yt} + (\tilde{u}\frac{u}{\tilde{u}_y})_{xy}
 + (\tilde{v}\frac{u}{\tilde{u}_y})_{yy}
+ (u\frac{\tilde{u}_x}{\tilde{u}_y})_{y}
 -u_x - (\frac{u}{\tilde{u}_y})_{yyy} \\[10pt]
= \partial_y\tilde{f} - (\frac{u}{\tilde{u}_y}\zeta)_y
+ 2 (\eta w)_y, \\[15pt]

w_t + \tilde{u}w_x + \tilde{v}w_y + \tilde{v}_y w + \tilde{u}_y (\frac{u}{\tilde{u}_y})_x
+ (\frac{\tilde{u}_x}{\tilde{u}_y} u)_y -u_x - w_{yy} \\[10pt]\quad
+ (\frac{u}{\tilde{u}_y}\zeta)_y - 2(\eta w)_y = \partial_y\tilde{f}, \\[15pt]

w_t + (\tilde{u}w)_x + (\tilde{v}w)_y - 2(\eta w)_y
+ (\frac{u}{\tilde{u}_y}\zeta)_y - w_{yy} = \partial_y\tilde{f}.
\end{array}
\end{equation}

Plug $(\ref{Sect3_Linearized_Prandtl_Transform_5})_2$ into $(\ref{Sect3_Linearized_Prandtl_Transform_4})_3$,
we have proved $(\ref{Sect3_VorticityEq})_1$.

Next, we determine the boundary condition for $w$, \, to which we only apply $\partial_{t,x}^j$ on the boundary.

Since $v|_{y=0}=\tilde{v}|_{y=0}=0$, the following equations hold on the boundary $\{y=0\}$:
\begin{equation}\label{Sect3_Linearized_Prandtl_BC_0}
\left\{\begin{array}{ll}
u_t + \tilde{u} u_x + u \tilde{u}_x - u_{yy} = f,\\[8pt]
u_y = \beta u.
\end{array}\right.
\end{equation}

Differentiate $(\ref{Sect3_Linearized_Prandtl_BC_0})_2$ tangentially, then we have
\begin{equation}\label{Sect3_Linearized_Prandtl_BC_1}
\begin{array}{ll}
u_{yt} = \beta u_t,\quad u_{yx} = \beta u_x.
\end{array}
\end{equation}

We calculate $\frac{u_{yy}}{\tilde{u}_y} + \frac{f}{\tilde{u}_y}$ on the boundary $\{y=0\}$:
\begin{equation}\label{Sect3_Linearized_Prandtl_BC_3}
\begin{array}{ll}
\frac{u_{yy}}{\tilde{u}_y} + \frac{f}{\tilde{u}_y} = \frac{u_{yy}}{\tilde{u}_y} + \tilde{f}
 \\[12pt]\hspace{1.45cm}
= w_y + \frac{2 u_y\tilde{u}_{yy}}{(\tilde{u}_y)^2}
+ u\Big(\frac{\tilde{u}_{yy}}{(\tilde{u}_y)^2}\Big)_y + \tilde{f} \\[12pt]\hspace{1.45cm}

= w_y + \frac{2\tilde{u}_{yy}}{\tilde{u}_y}\Big(w+ \frac{u \tilde{u}_{yy}}{(\tilde{u}_y)^2}\Big)
+ u\Big(\frac{\tilde{u}_{yy}}{(\tilde{u}_y)^2}\Big)_y + \tilde{f} \\[12pt]\hspace{1.45cm}
= w_y + 2\eta w + u\Big[\frac{2\tilde{u}_{yy}^2}{(\tilde{u}_y)^3}
+ \Big(\frac{\tilde{u}_{yy}}{(\tilde{u}_y)^2}\Big)_y \Big] + \tilde{f} \\[12pt]\hspace{1.45cm}
= w_y + 2\eta w + u \frac{\tilde{u}_{yyy}}{(\tilde{u}_{y})^2} + \tilde{f}.
\end{array}
\end{equation}

We calculate $\frac{u_t}{\tilde{u}_y}$ on the boundary $\{y=0\}$:
\begin{equation}\label{Sect3_Linearized_Prandtl_BC_4}
\begin{array}{ll}
\frac{u_t}{\tilde{u}_y} = \frac{u_{yt}}{\beta \tilde{u}_y}
= \frac{1}{\beta}(\frac{u_y}{\tilde{u}_y})_t + \frac{1}{\beta}\frac{u_y \tilde{u}_{yt}}{(\tilde{u}_y)^2} \\[11pt]\hspace{0.47cm}

= \frac{1}{\beta}\big[ (\frac{u}{\tilde{u}_y})_y + \frac{u\tilde{u}_{yy}}{(\tilde{u}_y)^2} \big]_t +\frac{1}{\beta}\frac{u_y\tilde{u}_{yt}}{(\tilde{u}_y)^2}
\\[11pt]\hspace{0.47cm}

= \frac{1}{\beta}w_t +\frac{1}{\beta}\frac{\tilde{u}_{yy}}{\tilde{u}_y}\frac{u_t}{\tilde{u}_y}
+\frac{u}{\beta}\frac{\tilde{u}_{yyt}}{(\tilde{u}_y)^2} - \frac{u}{\beta}\frac{2\tilde{u}_{yy}\tilde{u}_{yt}}{(\tilde{u}_y)^3}
+ u\frac{\tilde{u}_{yt}}{(\tilde{u}_y)^2}\ ,
\end{array}
\end{equation}
then
\begin{equation}\label{Sect3_Linearized_Prandtl_BC_5}
\begin{array}{ll}
(1-\frac{\eta}{\beta})\frac{u_t}{\tilde{u}_y}
= \frac{1}{\beta}w_t +\frac{u}{\beta}\frac{\tilde{u}_{yyt}}{(\tilde{u}_y)^2} + (1-\frac{2\eta}{\beta})u\frac{\tilde{u}_{yt}}{(\tilde{u}_y)^2}.
\end{array}
\end{equation}

Similarly, we calculate $\frac{u_x}{\tilde{u}_y}$ on the boundary $\{y=0\}$:
\begin{equation}\label{Sect3_Linearized_Prandtl_BC_6}
\begin{array}{ll}
(1-\frac{\eta}{\beta})\frac{u_x}{\tilde{u}_y}
= \frac{1}{\beta}w_x +\frac{u}{\beta}\frac{\tilde{u}_{yyx}}{(\tilde{u}_y)^2} + (1-\frac{2\eta}{\beta})u\frac{\tilde{u}_{yx}}{(\tilde{u}_y)^2}.
\end{array}
\end{equation}

We transform $(\ref{Sect3_Linearized_Prandtl_BC_0})_1$ on the boundary $\{y=0\}$ and then plug $(\ref{Sect3_Linearized_Prandtl_BC_5})$ and $(\ref{Sect3_Linearized_Prandtl_BC_6})$ into the new equation, we get
\begin{equation}\label{Sect3_Linearized_Prandtl_BC_7}
\begin{array}{ll}
\frac{u_t}{\tilde{u}_y} + \frac{\tilde{u} u_x}{\tilde{u}_y} + \frac{u\tilde{u}_x}{\tilde{u}_y}
= \frac{u_{yy}}{\tilde{u}_y} + \frac{f}{\tilde{u}_y}, \\[15pt]

(1-\frac{\eta}{\beta})\frac{u_t}{\tilde{u}_y} + \tilde{u}(1-\frac{\eta}{\beta})\frac{u_x}{\tilde{u}_y}
+ (1-\frac{\eta}{\beta})\frac{u\tilde{u}_x}{\tilde{u}_y}
= (1-\frac{\eta}{\beta})(\frac{u_{yy}}{\tilde{u}_y} + \frac{f}{\tilde{u}_y}), \\[15pt]

\frac{1}{\beta}w_t +\frac{u}{\beta}\frac{\tilde{u}_{yyt}}{(\tilde{u}_y)^2}
+ (1-\frac{2\eta}{\beta})u\frac{\tilde{u}_{yt}}{(\tilde{u}_y)^2}
+ \frac{\tilde{u}}{\beta}w_x +\frac{u\tilde{u}}{\beta}\frac{\tilde{u}_{yyx}}{(\tilde{u}_y)^2}
+ (1-\frac{2\eta}{\beta})u\frac{\tilde{u}\tilde{u}_{yx}}{(\tilde{u}_y)^2} \\[12pt]\quad
+ (1-\frac{\eta}{\beta})\frac{u\tilde{u}_x}{\tilde{u}_y}
= (1-\frac{\eta}{\beta})(\frac{u_{yy}}{\tilde{u}_y} + \frac{f}{\tilde{u}_y}).
\end{array}
\end{equation}

After plugging $(\ref{Sect3_Linearized_Prandtl_BC_3})$ into $(\ref{Sect3_Linearized_Prandtl_BC_7})$, we get
\begin{equation}\label{Sect3_Linearized_Prandtl_BC_8}
\begin{array}{ll}
\frac{1}{\beta}(w_t + \tilde{u}w_x) +\frac{u}{\beta}\frac{\tilde{u}_{yyt}}{(\tilde{u}_y)^2}
+ (1-\frac{2\eta}{\beta})u\frac{\tilde{u}_{yt}}{(\tilde{u}_y)^2}
+ \frac{u\tilde{u}}{\beta}\frac{\tilde{u}_{yyx}}{(\tilde{u}_y)^2}
+ (1-\frac{2\eta}{\beta})u\frac{\tilde{u}\tilde{u}_{yx}}{(\tilde{u}_y)^2} \\[12pt]\quad
+ (1-\frac{\eta}{\beta})\frac{u\tilde{u}_x}{\tilde{u}_y}
= (1-\frac{\eta}{\beta})(w_y + 2\eta w + u \frac{\tilde{u}_{yyy}}{(\tilde{u}_{y})^2} + \tilde{f}),
\end{array}
\end{equation}
then we immediately obtain $(\ref{Sect3_VorticityEq})_2$:
\begin{equation}\label{Sect3_Linearized_Prandtl_BC_9}
\begin{array}{ll}
w_t + \tilde{u}w_x - (\beta-\eta)\Big(w_y + 2\eta w + \tilde{f}+ \zeta|_{y=0}\int\limits_0^{+\infty} w(t,x,\tilde{y}) \,\mathrm{d}\tilde{y} \Big) \\[12pt]
= - u\frac{\tilde{u}_{yyt}}{(\tilde{u}_y)^2}
- (\beta-2\eta)u\frac{\tilde{u}_{yt}}{(\tilde{u}_y)^2}
- u\tilde{u}\frac{\tilde{u}_{yyx}}{(\tilde{u}_y)^2}
- (\beta-2\eta)u\frac{\tilde{u}\tilde{u}_{yx}}{(\tilde{u}_y)^2}
- (\beta-\eta)\frac{u\tilde{u}_x}{\tilde{u}_y} \\[12pt]\quad
+ (\beta-\eta)u \frac{\tilde{u}_{yyy}}{(\tilde{u}_{y})^2}
+ (\beta-\eta)\zeta|_{y=0}\frac{u}{\tilde{u}_y}|_{y=0} \\[13pt]

= - \frac{u}{\tilde{u}_y}\Big[\frac{\tilde{u}_{yyt}}{\tilde{u}_y}
+ (\beta-2\eta)\frac{\tilde{u}_{yt}}{\tilde{u}_y}
+ \tilde{u}\frac{\tilde{u}_{yyx}}{\tilde{u}_y}
+ (\beta-2\eta)\frac{\tilde{u}\tilde{u}_{yx}}{\tilde{u}_y}
+ (\beta-\eta)\tilde{u}_x \\[12pt]\quad
- (\beta-\eta)\frac{\tilde{u}_{yyy}}{\tilde{u}_{y}}
- (\beta-\eta)\frac{\tilde{u}_{yt} + \tilde{u}\tilde{u}_{yx} -\tilde{u}_{yyy}}{\tilde{u}_y}\Big]
 \\[15pt]

= - \frac{u}{\tilde{u}_y}|_{y=0}\Big[\frac{\tilde{u}_{yyt}}{\tilde{u}_y}
-\eta\frac{\tilde{u}_{yt}}{\tilde{u}_y}
+ \tilde{u}\frac{\tilde{u}_{yyx}}{\tilde{u}_y}
-\eta\frac{\tilde{u}\tilde{u}_{yx}}{\tilde{u}_y} + (\beta-\eta)\tilde{u}_x \Big]\Big|_{y=0} \\[15pt]
= -\frac{w}{\beta-\eta}|_{y=0}\, \tilde{\zeta}|_{y=0} -w\tilde{u}_x|_{y=0}.
\end{array}
\end{equation}

Thus, Lemma $\ref{Sect3_Equation_Lemma}$ is proved.
\end{proof}

\begin{remark}\label{Sect3_Remark_Beta}
As $\beta\rto +\infty$, $u|_{y=0}\rto 0$, $\int\limits_0^{+\infty} w(t,x,\tilde{y}) \,\mathrm{d}\tilde{y} = -\frac{u}{\tilde{u}_y}|_{y=0} \rto 0$ and then
$\int\limits_y^{+\infty} w(t,x,\tilde{y}) \,\mathrm{d}\tilde{y} = -\int\limits_0^y w(t,x,\tilde{y}) \,\mathrm{d}\tilde{y}$.
Thus, the system $(\ref{Sect3_VorticityEq})$ covers the Dirichlet boundary case when $\beta = +\infty$.
\end{remark}

\vspace{0.15cm}
In order to obtain the weighted energy estimates about the system $(\ref{Sect3_VorticityEq})$, we introduce the useful inequalities (see $\cite{Alexandre_Yang_2014}$):
\begin{equation*}
\begin{array}{ll}
\|u\|_{\mathcal{C}_{\ell}^k} \leq C_s \|u\|_{\mathcal{A}_{\ell}^{k+2}}, \\[6pt]

\|u\|_{\mathcal{D}_{\ell}^k} \leq C_s \|u\|_{\mathcal{A}_{\ell}^{k+1}}, \\[6pt]

\|u\|_{\mathcal{A}_{\ell}^k(\Omega_T)} \approx
\sum\limits_{0\leq k_1+[\frac{k_2+1}{2}]\leq k}
\|u\|_{\mathcal{B}_{0,\ell}^{k_1,k_2}(\Omega_T)}.
\end{array}
\end{equation*}
and the following notations:
\begin{equation}\label{Sect3_Notions}
\begin{array}{ll}
\lambda_{0,0} = \|\tilde{u}-u^s\|_{\mathcal{B}_{0,0}^{0,0}} + \|\tilde{u}_x\|_{\mathcal{B}_{0,0}^{0,0}} + \|\tilde{v}\|_{L_y^{\infty}(L_{t,x}^2)}
+ \|\bar{\eta}\|_{L_y^2(L_{t,x}^{\infty})} + \|\eta-\bar{\eta}\|_{\mathcal{B}_{0,0}^{0,0}} \\[6pt]\hspace{1.1cm}
+ \|\zeta\|_{\mathcal{B}_{0,\ell}^{0,0}}, \\[11pt]

\lambda_{k_1,k_2} = \|\tilde{u}-u^s\|_{\mathcal{B}_{0,0}^{k_1,k_2}} + \|\tilde{u}_x\|_{\mathcal{B}_{0,0}^{k_1,k_2}}
+ \sum\limits_{0\leq m\leq k_1,0\leq q\leq k_2}\|\partial_{t,x}^{m}\partial_y^{q} u^s\|_{L_{y}^2(L_t^{\infty})} \\[14pt]\hspace{1.4cm}
+ \sum\limits_{0\leq m\leq k_1,0\leq q\leq k_2}\|\partial_{t,x}^{m}\partial_y^{q} \tilde{v}\|_{L_y^{\infty}(L_{t,x}^2)}
+ \sum\limits_{0\leq m\leq k_1,0\leq q\leq k_2}\|\partial_{t,x}^{m}\partial_y^{q} \bar{\eta}\|_{L_y^2(L_{t,x}^{\infty})} \\[15pt]\hspace{1.4cm}
+ \|\eta-\bar{\eta}\|_{\mathcal{B}_{0,0}^{k_1,k_2}} + \|\zeta\|_{\mathcal{B}_{0,\ell}^{k_1,k_2}}\ , \quad
k_1+[\frac{k_2+1}{2}]>0,
\\[12pt]

\lambda_{k_1}|_{\partial\Omega} = \big\|\tilde{u}|_{y=0}-u^s|_{y=0}\big\|_{A^{k_1}} + \big\|\partial_x\tilde{u}|_{y=0}\big\|_{A^{k_1}}
+ \sum\limits_{m=0}^{k_1}\big|\partial_t^{m}u^s|_{y=0}\big|_{\infty}
\\[9pt]\hspace{1.5cm}
+ \sum\limits_{m=0}^{k_1}\big\|\partial_{t,x}^{m}\tilde{\zeta}_1|_{y=0}\big\|_{L_{t,x}^{\infty}}
 + \big\|\tilde{\zeta}_2|_{y=0}\big\|_{A^{k_1}}
 + \sum\limits_{m=0}^{k_1}\big\|\partial_{t,x}^{m}\bar{\eta}|_{y=0}\big\|_{L_{t,x}^{\infty}} \\[11pt]\hspace{1.5cm}
+ \big\|\eta|_{y=0}-\bar{\eta}|_{y=0}\big\|_{A^{k_1}}, \quad k_1\geq 0,
\\[12pt]

\lambda_k = \sum\limits_{0\leq k_1 + [\frac{k_2+1}{2}]\leq k} \lambda_{k_1,k_2}
 + \sum\limits_{0\leq k_1\leq k}\lambda_{k_1}|_{\partial\Omega}\ .
\end{array}
\end{equation}

Via the weighted energy method, we have the following theorem for the equations $(\ref{Sect3_VorticityEq})$.
\begin{theorem}\label{Sect3_Main_Estimate_Thm}
Assume $0<\delta_{\beta}\leq\beta<+\infty$, $\ell>\frac{1}{2},\ \beta-\eta|_{y=0}\geq\delta>0,\ |\eta|_{\infty}\leq C_{\eta}$, then we have the following a priori estimate for $(\ref{Sect3_VorticityEq})$:
\begin{equation}\label{Sect3_Main_Estimate_Eq_1}
\begin{array}{ll}
\|w\|_{\mathcal{A}_{\ell}^k} + \frac{1}{\sqrt{\beta + C_{\eta}}}\big\|w|_{y=0}\big\|_{A^k}
\leq C(T) \|\tilde{f}\|_{\mathcal{A}_{\ell}^k}
+ C(T) \lambda_{k} \|\tilde{f}\|_{\mathcal{A}_{\ell}^3}, \hspace{0.6cm} 0\leq k\leq k_0.
\end{array}
\end{equation}
If $\lambda_{p}<+\infty$ where $p\geq 3$, then we have the following a priori estimate for $(\ref{Sect3_VorticityEq})$:
\begin{equation}\label{Sect3_Main_Estimate_Eq_2}
\begin{array}{ll}
\|w\|_{\mathcal{A}_{\ell}^k} + \frac{1}{\sqrt{\beta + C_{\eta}}}\big\|w|_{y=0}\big\|_{A^k} \leq C(T) \|\tilde{f}\|_{\mathcal{A}_{\ell}^k},\hspace{0.4cm} 0\leq k\leq p.
\end{array}
\end{equation}
\end{theorem}

\begin{proof}
1. $L^2$-estimate:

Multiple $(\ref{Sect3_VorticityEq})_1$ with $e^{-2\lambda t}<y>^{2\ell}w$, we get
\begin{equation}\label{Sect3_Estimate1_1}
\begin{array}{ll}
w_t e^{-2\lambda t}<y>^{2\ell}w + (\tilde{u}w)_x e^{-2\lambda t}<y>^{2\ell}w + (\tilde{v}w)_y e^{-2\lambda t}<y>^{2\ell}w
\\[6pt]\quad
- 2(\eta w)_y e^{-2\lambda t}<y>^{2\ell}w
- \big(\zeta \int\limits_y^{+\infty} w(t,x,\tilde{y}) \,\mathrm{d}\tilde{y}\big)_y e^{-2\lambda t}<y>^{2\ell}w \\[12pt]\quad
- w_{yy} e^{-2\lambda t}<y>^{2\ell}w
- (\partial_y\tilde{f})e^{-2\lambda t}<y>^{2\ell}w =0.
\end{array}
\end{equation}

Integrate $(\ref{Sect3_Estimate1_1})$ in $\mathbb{R}_{+}^2$, we get
\begin{equation}\label{Sect3_Estimate1_2}
\begin{array}{ll}
\frac{1}{2}\frac{\mathrm{d}}{\mathrm{d}t}\|e^{-\lambda t} w\|_{L_{\ell}^2(\mathbb{R}_{+}^2)}^2
+\lambda \|e^{-\lambda t} w\|_{L_{\ell}^2(\mathbb{R}_{+}^2)}^2
+ \|e^{-\lambda t} w_y\|_{L_{\ell}^2(\mathbb{R}_{+}^2)}^2 \\[6pt]\quad

+ \int\limits_{\mathbb{R}}\Big(w_y + 2\eta w + \tilde{f}
+ \zeta\int\limits_0^{+\infty} w(t,x,\tilde{y}) \,\mathrm{d}\tilde{y} \Big)\Big|_{y=0} e^{-2\lambda t}w|_{y=0}
\,\mathrm{d}x \\[12pt]

=  \iint\limits_{\mathbb{R}_{+}^2} e^{-2\lambda t}<y>^{2\ell} \tilde{v} w_y w \mathrm{d}x\mathrm{d}y
+ 2\ell\iint\limits_{\mathbb{R}_{+}^2} e^{-2\lambda t}y<y>^{2\ell-2} \tilde{v} w^2 \mathrm{d}x\mathrm{d}y \\[15pt]\quad

- 4\ell\iint\limits_{\mathbb{R}_{+}^2} e^{-2\lambda t} y<y>^{2\ell-2} \eta w^2 \,\mathrm{d}x\mathrm{d}y
- 2\iint\limits_{\mathbb{R}_{+}^2} e^{-2\lambda t}<y>^{2\ell} \eta w w_y \,\mathrm{d}x\mathrm{d}y \\[12pt]\quad

- 2\ell\iint\limits_{\mathbb{R}_{+}^2} \zeta \int\limits_y^{+\infty} w(t,x,\tilde{y}) \,\mathrm{d}\tilde{y}\cdot
e^{-2\lambda t}y<y>^{2\ell-2}w \,\mathrm{d}x\mathrm{d}y \\[12pt]\quad

- \iint\limits_{\mathbb{R}_{+}^2} \zeta \int\limits_y^{+\infty} w(t,x,\tilde{y}) \,\mathrm{d}\tilde{y}\cdot
e^{-2\lambda t}<y>^{2\ell}w_y \,\mathrm{d}x\mathrm{d}y \\[15pt]\quad

- 2\ell\iint\limits_{\mathbb{R}_{+}^2} e^{-2\lambda t} y<y>^{2\ell-2} w w_y \,\mathrm{d}x\mathrm{d}y
- \frac{1}{2}\iint\limits_{\mathbb{R}_{+}^2} e^{-2\lambda t}<y>^{2\ell} \tilde{u}_x w^2 \mathrm{d}x\mathrm{d}y \\[15pt]\quad

-2\ell\iint\limits_{\mathbb{R}_{+}^2} e^{-2\lambda t} y<y>^{2\ell-2} w\tilde{f} \,\mathrm{d}x\mathrm{d}y
-\iint\limits_{\mathbb{R}_{+}^2} e^{-2\lambda t}<y>^{2\ell} w_y\tilde{f} \,\mathrm{d}x\mathrm{d}y := I_1.
\end{array}
\end{equation}

In order to eliminate the integrals on the boundary $\{y=0\}$ in $(\ref{Sect3_Estimate1_2})$, we need to estimate the boundary condition $(\ref{Sect3_VorticityEq})_2$.
Then multiple $(\ref{Sect3_VorticityEq})_2$ with $e^{-2\lambda t}w$, we get
\begin{equation}\label{Sect3_Estimate1_3}
\begin{array}{ll}
\frac{w_t}{\beta-\eta|_{y=0}} e^{-2\lambda t}w
- \Big(w_y + 2\eta|_{y=0} w + \tilde{f}
+ \zeta|_{y=0}\int\limits_0^{+\infty} w(t,x,\tilde{y}) \,\mathrm{d}\tilde{y} \Big) e^{-2\lambda t}w \\[11pt]
= - \frac{w\tilde{\zeta}|_{y=0}}{(\beta-\eta|_{y=0})^2} e^{-2\lambda t}w
- \frac{(\tilde{u}w)_x}{\beta-\eta|_{y=0}} e^{-2\lambda t}w.
\end{array}
\end{equation}

Integrate $(\ref{Sect3_Estimate1_3})$ in $\mathbb{R}$, we get the estimate on $\{y=0\}$,
\begin{equation}\label{Sect3_Estimate1_4}
\begin{array}{ll}
\frac{1}{2}\frac{\mathrm{d}}{\mathrm{d}t} \int\limits_{\mathbb{R}}\frac{e^{-2\lambda t}}{\beta-\eta|_{y=0}} (w|_{y=0})^2
\,\mathrm{d}x
+ \lambda \int\limits_{\mathbb{R}}\frac{e^{-2\lambda t}}{\beta-\eta|_{y=0}} (w|_{y=0})^2 \,\mathrm{d}x\\[12pt]\quad
- \int\limits_{\mathbb{R}}\Big(w_y + 2\eta w + \tilde{f}
+ \zeta\int\limits_0^{+\infty} w(t,x,\tilde{y}) \,\mathrm{d}\tilde{y} \Big)\Big|_{y=0} e^{-2\lambda t}w|_{y=0}
\,\mathrm{d}x \\[12pt]

= - \frac{1}{2}\int\limits_{\mathbb{R}}\tilde{u}_x|_{y=0}\frac{e^{-2\lambda t}}{\beta-\eta|_{y=0}} (w|_{y=0})^2 \,\mathrm{d}x
+ \frac{1}{2}\int\limits_{\mathbb{R}}\eta_x|_{y=0} \frac{e^{-2\lambda t}}{(\beta-\eta|_{y=0})^2} \tilde{u} (w|_{y=0})^2 \,\mathrm{d}x
\\[14pt]\quad
+ \frac{1}{2}\int\limits_{\mathbb{R}}\eta_t|_{y=0}\frac{e^{-2\lambda t}}{(\beta-\eta|_{y=0})^2} (w|_{y=0})^2 \,\mathrm{d}x
- \int\limits_{\mathbb{R}}\frac{\tilde{\zeta}|_{y=0}\, e^{-2\lambda t}}{(\beta-\eta|_{y=0})^2} (w|_{y=0})^2 \,\mathrm{d}x
:= I_2.
\end{array}
\end{equation}

By $(\ref{Sect3_Estimate1_2})+(\ref{Sect3_Estimate1_4})$, the integrals on the boundary $\{y=0\}$ are eliminated, then we get
\begin{equation}\label{Sect3_Estimate1_5}
\begin{array}{ll}
\frac{1}{2}\frac{\mathrm{d}}{\mathrm{d}t}\|e^{-\lambda t} w\|_{L_{\ell}^2(\mathbb{R}_{+}^2)}^2
+\lambda \|e^{-\lambda t} w\|_{L_{\ell}^2(\mathbb{R}_{+}^2)}^2
+ \|e^{-\lambda t} w_y\|_{L_{\ell}^2(\mathbb{R}_{+}^2)}^2 \\[12pt]\quad

+ \frac{1}{2}\frac{\mathrm{d}}{\mathrm{d}t} \int\limits_{\mathbb{R}}\frac{e^{-2\lambda t}}{\beta-\eta|_{y=0}} (w|_{y=0})^2
\,\mathrm{d}x
+ \lambda \int\limits_{\mathbb{R}}\frac{e^{-2\lambda t}}{\beta-\eta|_{y=0}} (w|_{y=0})^2 \,\mathrm{d}x

= I_1 + I_2.
\end{array}
\end{equation}

Integrate $(\ref{Sect3_Estimate1_5})$, we have
\begin{equation}\label{Sect3_Estimate1_6}
\begin{array}{ll}
\|e^{-\lambda t} w\|_{L_{\ell}^2(\mathbb{R}_{+}^2)}^2
+ 2\lambda \|e^{-\lambda t} w\|_{L_{\ell}^2([0,t]\times\mathbb{R}_{+}^2)}^2
+ 2\|e^{-\lambda t} w_y\|_{L_{\ell}^2([0,t]\times\mathbb{R}_{+}^2)}^2 \\[9pt]\quad

+ \int\limits_{\mathbb{R}}\frac{e^{-2\lambda t}}{\beta-\eta|_{y=0}} (w|_{y=0})^2
\,\mathrm{d}x
+ 2\lambda \int\limits_0^t\int\limits_{\mathbb{R}}\frac{e^{-2\lambda t}}{\beta-\eta|_{y=0}} (w|_{y=0})^2 \,\mathrm{d}x\,\mathrm{d}t \\[9pt]

= \big\|w|_{t=0}\big\|_{L_{\ell}^2(\mathbb{R}_{+}^2)}^2 + \int\limits_{\mathbb{R}}\frac{1}{\beta-\eta|_{y=0}} (w|_{y=0,t=0})^2 \,\mathrm{d}x
+ 2\int\limits_0^t I_1 + I_2 \,\mathrm{d}t.
\end{array}
\end{equation}

We need to estimate the following two terms in $\int\limits_0^t I_1 + I_2 \,\mathrm{d}t$: 
\vspace{-0.3cm}
\begin{equation*}
\begin{array}{ll}
- 2\ell\int\limits_0^t\iint\limits_{\mathbb{R}_{+}^2} \zeta \int\limits_y^{+\infty} w(t,x,\tilde{y}) \,\mathrm{d}\tilde{y}\cdot
e^{-2\lambda t}y<y>^{2\ell-2}w \,\mathrm{d}x\mathrm{d}y\mathrm{d}t \\[12pt]\qquad

\leq 2\ell C_{\ell}\|\zeta\|_{L_{t,x}^{\infty}(L_{\ell,y}^2)}\|e^{-\lambda t} w\|_{L_{\ell}^2([0,t]\times\mathbb{R}_{+}^2)}^2,
\\[14pt]

- \int\limits_0^t\iint\limits_{\mathbb{R}_{+}^2} \zeta \int\limits_y^{+\infty} w(t,x,\tilde{y}) \,\mathrm{d}\tilde{y}\cdot
e^{-2\lambda t}<y>^{2\ell}w_y \,\mathrm{d}x\mathrm{d}y\mathrm{d}t \\[9pt]\qquad

\leq q\|\zeta\|_{L_{t,x}^{\infty}(L_{\ell,y}^2)}\|e^{-\lambda t} w_y\|_{L_{\ell}^2([0,t]\times\mathbb{R}_{+}^2)}^2 +
\frac{C_{\ell}}{q}\|\zeta\|_{L_{t,x}^{\infty}(L_{\ell,y}^2)}\|e^{-\lambda t} w\|_{L_{\ell}^2([0,t]\times\mathbb{R}_{+}^2)}^2,
\end{array}
\end{equation*}
where $q>0$ is a small constant.

Then it is easy to get the estimate:
\begin{equation}\label{Sect3_Estimate1_7}
\begin{array}{ll}
\int\limits_0^t I_1 + I_2 \,\mathrm{d}t
\\[11pt]

\leq (2\ell+\frac{2}{q}) \| e^{-\lambda t}\tilde{f} \|_{L_\ell^2([0,t]\times\mathbb{R}_{+}^2)}^2
+ q c_1(\lambda_{3})\| e^{-\lambda t}w_y\|_{L_\ell^2([0,t]\times\mathbb{R}_{+}^2)}^2 \\[11pt]\quad
+ c_2(\lambda_{3},q)\| e^{-\lambda t}w\|_{L_\ell^2([0,t]\times\mathbb{R}_{+}^2)}^2
+ c_3(\lambda_{3})\int\limits_{\mathbb{R}}\frac{e^{-2\lambda t}}{\beta-\eta|_{y=0}} (w|_{y=0})^2 \,\mathrm{d}x,
\end{array}
\end{equation}

Let $q$ to be small enough such that $q c_1(\lambda_{3})\leq \frac{1}{2}$,
let $\lambda>0$  to be large enough such that $\lambda\geq \max\{c_2(\lambda_{3},q),  c_3(\lambda_{3})\}$, then
\begin{equation}\label{Sect3_Estimate1_8}
\begin{array}{ll}
\|e^{-\lambda t} w\|_{L_{\ell}^2(\mathbb{R}_{+}^2)}^2
+\lambda \|e^{-\lambda t} w\|_{L_{\ell}^2([0,t]\times\mathbb{R}_{+}^2)}^2
+ \|e^{-\lambda t} w_y\|_{L_{\ell}^2([0,t]\times\mathbb{R}_{+}^2)}^2 \\[8pt]\quad

+ \int\limits_{\mathbb{R}}\frac{e^{-2\lambda t}}{\beta-\eta|_{y=0}} (w|_{y=0})^2
\,\mathrm{d}x
+ \lambda \int\limits_0^t\int\limits_{\mathbb{R}}\frac{e^{-2\lambda t}}{\beta-\eta|_{y=0}} (w|_{y=0})^2 \,\mathrm{d}x\mathrm{d}t\\[13pt]

\lem \| e^{-\lambda t}\tilde{f} \|_{L_{\ell}^2([0,T]\times\mathbb{R}_{+}^2)}^2 + \big\|w|_{t=0}\big\|_{L_{\ell}^2(\mathbb{R}_{+}^2)}^2
+ \int\limits_{\mathbb{R}}\frac{1}{\beta-\eta|_{y=0,t=0}} (w|_{y=0,t=0})^2
\,\mathrm{d}x.
\end{array}
\end{equation}

Fix $\lambda>0$ and note that $T<+\infty,\ w|_{t\leq 0} =0$, for any $t\in [0,T]$, we have
\begin{equation}\label{Sect3_Estimate1_9}
\begin{array}{ll}
\|w\|_{\mathcal{B}_{0,\ell}^{0,1}}
+\frac{1}{\sqrt{\beta + C_{\eta}}}\big\|w|_{y=0} \big\|_{A^0}
\leq C(\lambda T)\|\tilde{f}\|_{\mathcal{B}_{0,\ell}^{0,0}} \ .
\end{array}
\end{equation}

\vspace{0.2cm}
2. tangential derivatives estimates:

Apply $\partial_{t,x}^{k}$ to the linearized system $(\ref{Sect3_VorticityEq})$, where $k\geq 1$ we have
\begin{equation}\label{Sect3_VorticityEq2}
\left\{\begin{array}{ll}
(\partial_{t,x}^{k}w)_t
- \partial_y \Big[\partial_{t,x}^{k}
\Big(w_y + 2\eta w + \tilde{f} + \zeta \int\limits_y^{+\infty} w(t,x,\tilde{y}) \,\mathrm{d}\tilde{y} \Big)\Big] \\[11pt]\hspace{0.6cm}
= - \partial_{t,x}^{k}(\tilde{u}w)_x -\partial_{t,x}^{k}(\tilde{v}w)_y, \\[13pt]

\frac{(\partial_{t,x}^{k}w)_t}{\beta-\eta|_{y=0}}
- \partial_{t,x}^{k}\Big(w_y + 2\eta|_{y=0} w + \tilde{f} + \zeta|_{y=0}\int\limits_0^{+\infty} w(t,x,\tilde{y}) \,\mathrm{d}\tilde{y}\Big) \\[12pt]\hspace{0.6cm}

= -\sum\limits_{k_1<k} \partial_t(\partial_{t,x}^{k_1}w) \partial_{t,x}^{k_2}(\frac{1}{\beta-\eta|_{y=0}})
- \partial_{t,x}^{k}[\frac{(\tilde{u}w)_x}{\beta-\eta|_{y=0}}] \\[12pt]\hspace{0.95cm}

- \sum\limits_{k_1+k_2=k}\partial_{t,x}^{k_1}(\frac{\tilde{\zeta}}{(\beta-\eta)^2})|_{y=0}
\partial_{t,x}^{k_2}w, \quad y=0, \\[15pt]

\partial_{t,x}^{k} w|_{t\leq 0} =0.
\end{array}\right.
\end{equation}

Multiple $(\ref{Sect3_VorticityEq2})_1$ with $e^{-2\lambda t}<y>^{2\ell}\partial_{t,x}^{k}w$, we get
\begin{equation}\label{Sect3_Estimate2_1}
\begin{array}{ll}
(\partial_{t,x}^{k}w)_t e^{-2\lambda t}<y>^{2\ell}\partial_{t,x}^{k}w  \\[7pt]\hspace{0.5cm}

- \partial_y \Big[\partial_{t,x}^{k}
\Big(w_y + 2\eta w + \tilde{f} + \zeta \int\limits_y^{+\infty} w(t,x,\tilde{y}) \,\mathrm{d}\tilde{y} \Big)\Big]
e^{-2\lambda t}<y>^{2\ell}\partial_{t,x}^{k}w  \\[17pt]

= - \partial_{t,x}^{k}(\tilde{u}w)_x e^{-2\lambda t}<y>^{2\ell}\partial_{t,x}^{k}w
-\partial_{t,x}^{k}(\tilde{v}w)_y e^{-2\lambda t}<y>^{2\ell}\partial_{t,x}^{k}w.
\end{array}
\end{equation}

Integrate $(\ref{Sect3_Estimate2_1})$ in $\mathbb{R}_{+}^2$, we get
\begin{equation}\label{Sect3_Estimate2_2}
\begin{array}{ll}
\frac{1}{2}\frac{\mathrm{d}}{\mathrm{d}t}\|e^{-\lambda t} \partial_{t,x}^{k}w\|_{L_{\ell}^2(\mathbb{R}_{+}^2)}^2
+\lambda \|e^{-\lambda t} \partial_{t,x}^{k}w\|_{L_{\ell}^2(\mathbb{R}_{+}^2)}^2
+\|e^{-\lambda t} \partial_{t,x}^{k}w_y\|_{L_{\ell}^2(\mathbb{R}_{+}^2)}^2 \\[9pt]\quad

+ \int\limits_{\mathbb{R}} \partial_{t,x}^{k}
\Big(w_y + 2\eta w + \tilde{f} + \zeta \int\limits_y^{+\infty} w(t,x,\tilde{y}) \,\mathrm{d}\tilde{y} \Big)\Big|_{y=0}
e^{-2\lambda t}\partial_{t,x}^{k}w|_{y=0} \,\mathrm{d}x \\[12pt]

= - \iint\limits_{\mathbb{R}_{+}^2}\partial_{t,x}^{k}
\Big(2\eta w + \tilde{f} + \zeta \int\limits_y^{+\infty} w(t,x,\tilde{y}) \,\mathrm{d}\tilde{y} \Big)\,
e^{-2\lambda t}<y>^{2\ell}\partial_y\partial_{t,x}^{k}w \,\mathrm{d}x\mathrm{d}y \\[12pt]\quad

- 2\ell\iint\limits_{\mathbb{R}_{+}^2}\partial_{t,x}^{k}
\Big(w_y + 2\eta w + \tilde{f} + \zeta \int\limits_y^{+\infty} w(t,x,\tilde{y}) \,\mathrm{d}\tilde{y} \Big)\,
e^{-2\lambda t}y<y>^{2\ell-2}\partial_{t,x}^{k}w \,\mathrm{d}x\mathrm{d}y \\[18pt]\quad

- \iint\limits_{\mathbb{R}_{+}^2} \partial_{t,x}^{k}(\tilde{u}w)_x e^{-2\lambda t}<y>^{2\ell}\partial_{t,x}^{k}w \mathrm{d}x\mathrm{d}y \\[15pt]\quad
- \iint\limits_{\mathbb{R}_{+}^2} \partial_{t,x}^k(\tilde{v}w)_y e^{-2\lambda t}<y>^{2\ell}\partial_{t,x}^{k}w
\,\mathrm{d}x\mathrm{d}y := I_3.
\end{array}
\end{equation}

In order to eliminate the integrals on the boundary $\{y=0\}$ in $(\ref{Sect3_Estimate2_2})$, we need to estimate the boundary condition $(\ref{Sect3_VorticityEq2})_2$. Then multiple $(\ref{Sect3_VorticityEq2})_2$ with $e^{-2\lambda t}\partial_{t,x}^{k}w$, we get
\begin{equation}\label{Sect3_Estimate2_3}
\begin{array}{ll}
\frac{(\partial_{t,x}^{k}w)_t}{\beta-\eta|_{y=0}}e^{-2\lambda t}\partial_{t,x}^{k}w \\[7pt]\quad
- \partial_{t,x}^{k}\Big(w_y + 2\eta|_{y=0} w + \tilde{f} + \zeta|_{y=0}\int\limits_0^{+\infty} w(t,x,\tilde{y}) \,\mathrm{d}\tilde{y}\Big)
\, e^{-2\lambda t}\partial_{t,x}^{k}w
\end{array}
\end{equation}

\begin{equation*}
\begin{array}{ll}
= -\sum\limits_{k_1<k} \partial_t(\partial_{t,x}^{k_1}w) \partial_{t,x}^{k_2}(\frac{1}{\beta-\eta|_{y=0}}) e^{-2\lambda t}\partial_{t,x}^{k}w
- \partial_{t,x}^{k}[\frac{(\tilde{u}w)_x}{\beta-\eta|_{y=0}}] e^{-2\lambda t}\partial_{t,x}^{k}w
\\[11pt]\quad

- \sum\limits_{k_1+k_2=k}\partial_{t,x}^{k_1}(\frac{\tilde{\zeta}}{(\beta-\eta)^2})|_{y=0}
\partial_{t,x}^{k_2}w \, e^{-2\lambda t}\partial_{t,x}^{k}w.
\end{array}
\end{equation*}

Integrate $(\ref{Sect3_Estimate2_3})$ in $\mathbb{R}$, we get
\begin{equation}\label{Sect3_Estimate2_4}
\begin{array}{ll}
\frac{1}{2}\frac{\mathrm{d}}{\mathrm{d}t} \int\limits_{\mathbb{R}}\frac{e^{-2\lambda t}}{\beta-\eta|_{y=0}} (\partial_{t,x}^{k}w|_{y=0})^2
\,\mathrm{d}x
+ \lambda \int\limits_{\mathbb{R}}\frac{e^{-2\lambda t}}{\beta-\eta|_{y=0}} (\partial_{t,x}^{k}w|_{y=0})^2 \,\mathrm{d}x\\[11pt]\quad
- \int\limits_{\mathbb{R}}\partial_{t,x}^{k}\Big(w_y + 2\eta|_{y=0} w + \tilde{f} + \zeta|_{y=0}\int\limits_0^{+\infty} w(t,x,\tilde{y}) \,\mathrm{d}\tilde{y}\Big)\Big|_{y=0}
\cdot e^{-2\lambda t}\partial_{t,x}^{k}w|_{y=0} \,\mathrm{d}x \\[14pt]

= \frac{1}{2}\int\limits_{\mathbb{R}}\eta_t|_{y=0}\frac{e^{-2\lambda t}}{(\beta-\eta|_{y=0})^2} (\partial_{t,x}^{k}w|_{y=0})^2 \,\mathrm{d}x  \\[12pt]\quad
-\sum\limits_{k_1<k} \int\limits_{\mathbb{R}} \partial_t(\partial_{t,x}^{k_1}w) \partial_{t,x}^{k_2}(\frac{1}{\beta-\eta|_{y=0}})e^{-2\lambda t}\partial_{t,x}^{k}w\,\mathrm{d}x
- \int\limits_{\mathbb{R}}\partial_{t,x}^{k}[\frac{(\tilde{u}w)_x}{\beta-\eta|_{y=0}}] e^{-2\lambda t}\partial_{t,x}^{k}w \,\mathrm{d}x
\\[12pt]\quad

- \sum\limits_{k_1+k_2=k}\int\limits_{\mathbb{R}} \partial_{t,x}^{k_1}(\frac{\tilde{\zeta}}{(\beta-\eta)^2})|_{y=0}
\partial_{t,x}^{k_2}w\, e^{-2\lambda t}\partial_{t,x}^{k}w \,\mathrm{d}x :=I_4.
\end{array}
\end{equation}

By $(\ref{Sect3_Estimate2_2})+(\ref{Sect3_Estimate2_4})$, the integrals on the boundary $\{y=0\}$ are eliminated, then we get
\begin{equation}\label{Sect3_Estimate2_5}
\begin{array}{ll}
\frac{1}{2}\frac{\mathrm{d}}{\mathrm{d}t}\|e^{-\lambda t} \partial_{t,x}^{k}w\|_{L_{\ell}^2(\mathbb{R}_{+}^2)}^2
+\lambda \|e^{-\lambda t} \partial_{t,x}^{k}w\|_{L_{\ell}^2(\mathbb{R}_{+}^2)}^2
+ \|e^{-\lambda t} \partial_{t,x}^{k}w_y\|_{L_{\ell}^2(\mathbb{R}_{+}^2)}^2 \\[10pt]\quad

+ \frac{1}{2}\frac{\mathrm{d}}{\mathrm{d}t} \int\limits_{\mathbb{R}}\frac{e^{-2\lambda t}}{\beta-\eta|_{y=0}} (\partial_{t,x}^{k}w|_{y=0})^2
\,\mathrm{d}x
+ \lambda \int\limits_{\mathbb{R}}\frac{e^{-2\lambda t}}{\beta-\eta|_{y=0}} (\partial_{t,x}^{k}w|_{y=0})^2 \,\mathrm{d}x
= I_3+I_4.
\end{array}
\end{equation}

Integrate $(\ref{Sect3_Estimate2_5})$, we have
\begin{equation}\label{Sect3_Estimate2_6}
\begin{array}{ll}
\|e^{-\lambda t} \partial_{t,x}^{k}w\|_{L_{\ell}^2(\mathbb{R}_{+}^2)}^2
+ 2\lambda \|e^{-\lambda t} \partial_{t,x}^{k}w\|_{L_{\ell}^2([0,t]\times\mathbb{R}_{+}^2)}^2
+ 2\|e^{-\lambda t} \partial_{t,x}^{k}w_y\|_{L_{\ell}^2([0,t]\times\mathbb{R}_{+}^2)}^2 \\[8pt]\quad

+ \int\limits_{\mathbb{R}}\frac{e^{-2\lambda t}}{\beta-\eta|_{y=0}} (\partial_{t,x}^{k}w|_{y=0})^2
\,\mathrm{d}x
+ 2\lambda \int\limits_0^t\int\limits_{\mathbb{R}}\frac{e^{-2\lambda t}}{\beta-\eta|_{y=0}} (\partial_{t,x}^{k}w|_{y=0})^2 \,\mathrm{d}x\mathrm{d}t\\[8pt]

= \big\|\partial_{t,x}^{k}w|_{t=0} \big\|_{L_{\ell}^2(\mathbb{R}_{+}^2)}^2
+\int\limits_{\mathbb{R}}\frac{1}{\beta-\eta|_{y=0}} (\partial_{t,x}^{k}w|_{y=0,t=0})^2 \,\mathrm{d}x
+ 2\int\limits_0^t I_3+I_4 \,\mathrm{d}t.
\end{array}
\end{equation}

We estimate the terms in $\int\limits_0^t I_3\,\mathrm{d}t$, note that $e^{-\lambda t}\leq 1$:
the first estimate is
\begin{equation}\label{Sect3_Estimate2_7_1}
\begin{array}{ll}
- \int\limits_0^t\iint\limits_{\mathbb{R}_{+}^2}\partial_{t,x}^{k}
\Big(2\eta w + \tilde{f} + \zeta \int\limits_y^{+\infty} w(t,x,\tilde{y}) \,\mathrm{d}\tilde{y} \Big)\,
e^{-2\lambda t}<y>^{2\ell}\partial_y\partial_{t,x}^{k}w \,\mathrm{d}x\mathrm{d}y\mathrm{d}t \\[12pt]

\leq q\|w\|_{\mathcal{B}_{\lambda,\ell}^{k,1}}^2 + C(q)\Big(\|\tilde{f}\|_{\mathcal{B}_{\lambda,\ell}^{k,0}}^2
+ |\eta|_{\infty}^2\|w\|_{\mathcal{B}_{\lambda,\ell}^{k,0}}^2 + \|w\|_{L_{\ell}^{\infty}}^2\|\partial_{t,x}^{k}(\eta-\bar{\eta})\|_{L^2([0,t]\times\mathbb{R}_{+}^2)}^2 \\[10pt]\quad

+ \|\partial_{t,x}^{k}\bar{\eta}\|_{L_y^2(L_{t,x}^{\infty})}^2\|w\|_{L_{t,x}^2(L_{y,\ell}^{\infty})}^2
+ \|\zeta\|_{L_{y,\ell}^2(L_{t,x}^{\infty})}^2\|w\|_{\mathcal{B}_{\lambda,\ell}^{k,0}}^2
+ \|\zeta\|_{\mathcal{B}_{0,\ell}^{k,0}}^2\|w\|_{L_{y,\ell}^2(L_{t,x}^{\infty})}^2 \Big)

\\[12pt]
\leq  q\|w\|_{\mathcal{B}_{\lambda,\ell}^{k,1}}^2 + C(q)\|\tilde{f}\|_{\mathcal{B}_{\lambda,\ell}^{k,0}}^2 + C(q,(\lambda_{2})^2)\|w\|_{\mathcal{B}_{\lambda,\ell}^{k,0}}^2 + C(q)\cdot(\lambda_{k})^2\|w\|_{\mathcal{B}_{0,\ell}^{2,1}}^2 \ .
\end{array}
\end{equation}

The second estimate is
\begin{equation}\label{Sect3_Estimate2_7_2}
\begin{array}{ll}
- 2\ell\int\limits_0^t\iint\limits_{\mathbb{R}_{+}^2}\partial_{t,x}^{k}
\Big(w_y + 2\eta w + \tilde{f} + \zeta \int\limits_y^{+\infty} w(t,x,\tilde{y}) \,\mathrm{d}\tilde{y} \Big)\\[14pt]\qquad
\cdot\, e^{-2\lambda t}y<y>^{2\ell-2}\partial_{t,x}^{k}w \,\mathrm{d}x\mathrm{d}y\mathrm{d}t \\[9pt]

\leq q\|w\|_{\mathcal{B}_{\lambda,\ell}^{k,1}}^2 + C(q)\|w\|_{\mathcal{B}_{\lambda,\ell}^{k,0}}^2+ C(\|\tilde{f}\|_{\mathcal{B}_{\lambda,\ell}^{k,0}}^2
+ |\eta|_{\infty}^2\|w\|_{\mathcal{B}_{\lambda,\ell}^{k,0}}^2 \\[10pt]\quad
+ \|w\|_{L_{\ell}^{\infty}}^2\|\partial_{t,x}^{k}(\eta-\bar{\eta})\|_{L^2([0,t]\times\mathbb{R}_{+}^2)}^2
+ \|\partial_{t,x}^{k}\bar{\eta}\|_{L_y^2(L_{t,x}^{\infty})}^2\|w\|_{L_{t,x}^2(L_{y,\ell}^{\infty})}^2 \\[10pt]\quad
+ \|\zeta\|_{L_{y,\ell}^2(L_{t,x}^{\infty})}^2\|w\|_{\mathcal{B}_{\lambda,\ell}^{k,0}}^2
+ \|\zeta\|_{\mathcal{B}_{0,\ell}^{k,0}}^2\|w\|_{L_{y,\ell}^2(L_{t,x}^{\infty})}^2)

\\[10pt]
\leq  q\|w\|_{\mathcal{B}_{\lambda,\ell}^{k,1}}^2 + C(q)\|\tilde{f}\|_{\mathcal{B}_{\lambda,\ell}^{k,0}}^2
+ C(q,(\lambda_{2})^2)\|w\|_{\mathcal{B}_{\lambda,\ell}^{k,0}}^2 + C\cdot(\lambda_{k})^2\|w\|_{\mathcal{B}_{0,\ell}^{2,1}}^2 \ .
\end{array}
\end{equation}

The third estimate is
\begin{equation}\label{Sect3_Estimate2_7_3}
\begin{array}{ll}
- \int\limits_0^t\iint\limits_{\mathbb{R}_{+}^2} \partial_{t,x}^{k}(\tilde{u}w)_x e^{-2\lambda t}<y>^{2\ell}\partial_{t,x}^{k}w \mathrm{d}x\mathrm{d}y
\mathrm{d}t\\[10pt]

= - \int\limits_0^t\iint\limits_{\mathbb{R}_{+}^2} \partial_{t,x}^{k}(\tilde{u}_x w) e^{-2\lambda t}<y>^{2\ell}\partial_{t,x}^{k}w \mathrm{d}x\mathrm{d}y
\mathrm{d}t \\[8pt]\quad
- \int\limits_0^t\iint\limits_{\mathbb{R}_{+}^2} (\sum\limits_{k_2<k}\partial_{t,x}^{k_1}\tilde{u} \partial_{t,x}^{k_2}w_x) e^{-2\lambda t}<y>^{2\ell}\partial_{t,x}^{k}w \mathrm{d}x\mathrm{d}y\mathrm{d}t
\\[8pt]\quad

+ \frac{1}{2}\int\limits_0^t\iint\limits_{\mathbb{R}_{+}^2} \tilde{u}_x e^{-2\lambda t}<y>^{2\ell} |\partial_{t,x}^{k}w|^2 \mathrm{d}x\mathrm{d}y
\mathrm{d}t \\[14pt]

\lem (|\partial_{t,x}^1\tilde{u}|_{\infty}+|\tilde{u}_x|_{\infty} + 1) \|w\|_{\mathcal{B}_{\lambda,\ell}^{k,0}}^2
+ \|\partial_{t,x}^k \tilde{u}_x\|_{L^2([0,t]\times\mathbb{R}_{+}^2)}^2 \|w\|_{L_{\ell}^{\infty}}^2  \\[12pt]\quad
+ \|\partial_{t,x}^k(\tilde{u}-u^s)\|_{L^2([0,t]\times\mathbb{R}_{+}^2)}^2 \|w_x\|_{L_{\ell}^{\infty}}^2
+ \|\partial_{t,x}^k u^s\|_{L_{y}^2(L_{t,x}^{\infty})}^2 \|w_x\|_{L_{y,\ell}^{\infty}(L_{t,x}^2)}^2
\\[12pt]

\lem ((\lambda_3)^2 + 1) \|w\|_{\mathcal{B}_{\lambda,\ell}^{k,0}}^2 + (\lambda_{k})^2 \|w\|_{\mathcal{B}_{0,\ell}^{3,1}}^2 \ .
\end{array}
\end{equation}

\vspace{0.2cm}
Since $\partial_{t,x}^k \tilde{v}|_{y=0} =0,\ k\geq 0$, we have the fourth estimate
\begin{equation}\label{Sect3_Estimate2_7_4}
\begin{array}{ll}
- \int\limits_0^t\iint\limits_{\mathbb{R}_{+}^2}\partial_{t,x}^k(\tilde{v}w)_y e^{-2\lambda t}<y>^{2\ell}\partial_{t,x}^{k}w
\,\mathrm{d}x\mathrm{d}y\mathrm{d}t \\[10pt]

= \int\limits_0^t\iint\limits_{\mathbb{R}_{+}^2}\partial_{t,x}^k(\tilde{v}w) e^{-2\lambda t}(<y>^{2\ell}\partial_{t,x}^{k}w_y
+ 2\ell y<y>^{2\ell-2}\partial_{t,x}^{k}w)
\,\mathrm{d}x\mathrm{d}y\mathrm{d}t \\[10pt]

\leq q\|w\|_{\mathcal{B}_{\lambda,\ell}^{k,1}}^2 + C(q)
\Big(|\tilde{v}|_{\infty}\|w\|_{\mathcal{B}_{\lambda,\ell}^{k,0}}^2 + \|w\|_{\mathcal{B}_{\lambda,\ell}^{k,0}}^2
+ \|\partial_{t,x}^k\tilde{v}\|_{L_y^{\infty}(L_{t,x}^2)}\|w\|_{L_{y,\ell}^2(L_{t,x}^{\infty})}^2 \Big) \\[11pt]

\leq q\|w\|_{\mathcal{B}_{\lambda,\ell}^{k,1}}^2 + C(q,(\lambda_2)^2)\|w\|_{\mathcal{B}_{\lambda,\ell}^{k,0}}^2
+ C(q)\cdot(\lambda_k)^2\|w\|_{\mathcal{B}_{0,\ell}^{2,1}}^2 \ .
\end{array}
\end{equation}

We estimate the terms in $\int\limits_0^t I_4\,\mathrm{d}t$ and note that $e^{-2\lambda t}\leq 1$:

The first estimate on $\{y=0\}$ is
\begin{equation}\label{Sect3_Estimate2_7_5}
\begin{array}{ll}
\frac{1}{2}\int\limits_0^t\int\limits_{\mathbb{R}}\eta_t|_{y=0}\frac{e^{-2\lambda t}}{(\beta-\eta|_{y=0})^2} (\partial_{t,x}^{k}w|_{y=0})^2 \,\mathrm{d}x\mathrm{d}t
\\[8pt]\quad
-\sum\limits_{k_1<k} \int\limits_0^t\int\limits_{\mathbb{R}} \partial_t(\partial_{t,x}^{k_1}w) \partial_{t,x}^{k_2}(\frac{1}{\beta-\eta|_{y=0}})
e^{-2\lambda t}\partial_{t,x}^{k}w\,\mathrm{d}x\mathrm{d}t \\[11pt]

\lem \left(\frac{1}{\delta}\big|\eta_t|_{y=0}\big|_{L_{t,x}^{\infty}}+1 + \big|\partial_{t,x}^1\eta|_{y=0}\big|_{L_{t,x}^{\infty}} \right)
\int\limits_0^t\int\limits_{\mathbb{R}}\frac{e^{-2\lambda t}}{\beta-\eta|_{y=0}} (\partial_{t,x}^{k}w|_{y=0})^2 \,\mathrm{d}x\mathrm{d}t \\[9pt]\quad

+ \big\|(\eta-\bar{\eta})|_{y=0}\big\|_{A^k}^2 \frac{\big|\partial_t w|_{y=0}\big|_{L_{t,x}^{\infty}}^2 }{C(\max\{{\delta},\,{\beta-C_{\eta}}\})}
+ \big\|\partial_{t,x}^{k}\bar{\eta}|_{y=0}\big\|_{L_{t,x}^{\infty}}^2
\frac{\big\|\partial_t w|_{y=0}\big\|_{A^0}^2}{C(\max\{{\delta},\,{\beta-C_{\eta}}\})} \\[9pt]

\lem (\lambda_3 +1)\int\limits_0^t\int\limits_{\mathbb{R}}\frac{e^{-2\lambda t}}{\beta-\eta|_{y=0}} (\partial_{t,x}^{k}w|_{y=0})^2 \,\mathrm{d}x\mathrm{d}t
+ \frac{(\lambda_{k})^2}{C(\max\{{\delta},\, {\beta-C_{\eta}}\})}\big \|w|_{y=0}\big\|_{A^3}^2.
\end{array}
\end{equation}

\vspace{-0.2cm}
The second estimate on $\{y=0\}$ is
\begin{equation}\label{Sect3_Estimate2_7_6}
\begin{array}{ll}
- \int\limits_0^t\int\limits_{\mathbb{R}}\partial_{t,x}^{k}[\frac{(\tilde{u}w)_x}{\beta-\eta|_{y=0}}] e^{-2\lambda t}\partial_{t,x}^{k}w \,\mathrm{d}x
\mathrm{d}t \\[10pt]
= - \int\limits_0^t\int\limits_{\mathbb{R}}\partial_{t,x}^{k}(\frac{\tilde{u}_x}{\beta-\eta|_{y=0}}w)e^{-2\lambda t}\partial_{t,x}^{k}w \,\mathrm{d}x\mathrm{d}t
+ \frac{1}{2}\int\limits_0^t\int\limits_{\mathbb{R}}\partial_x (\frac{\tilde{u} }{\beta-\eta|_{y=0}})e^{-2\lambda t}|\partial_{t,x}^{k}w|^2 \,\mathrm{d}x\mathrm{d}t \\[9pt]\quad
- \sum\limits_{k_2<k}\int\limits_0^t\int\limits_{\mathbb{R}}\partial_{t,x}^{k_1}(\frac{\tilde{u} }{\beta-\eta|_{y=0}})
\partial_{t,x}^{k_2}(w_x) e^{-2\lambda t}\partial_{t,x}^{k}w \,\mathrm{d}x\mathrm{d}t \\[8pt]

\lem (|\tilde{u}_x|_{\infty} + |\tilde{u}_t|_{\infty} + |\tilde{u}|_{\infty} + |\eta_x|_{\infty} + |\eta_t|_{\infty} + 1)
\int\limits_0^t\int\limits_{\mathbb{R}}\frac{e^{-2\lambda t}}{\beta-\eta|_{y=0}} (\partial_{t,x}^{k}w|_{y=0})^2 \,\mathrm{d}x\mathrm{d}t \hspace{0.3cm}
\\[12pt]\quad

+ (\big|\partial_{t,x}^{k}\bar{\eta}|_{y=0}\big|_{L_{t,x}^{\infty}}^2 + \big|\partial_{t,x}^{k}u^s|_{y=0}\big|_{L_{t,x}^{\infty}}^2)
\frac{\big\|w|_{y=0}\big\|_{A^0}^2 + \big\|w_x|_{y=0}\big\|_{A^0}^2}{C(\max\{{\delta},\,{\beta-C_{\eta}}\})} \\[13pt]\quad

+ \frac{|w|_{L_{t,x}^{\infty}}^2+ |w_x|_{L_{t,x}^{\infty}}^2}{C(\max\{{\delta},\,{\beta-C_{\eta}}\})}
\Big(\big\|\tilde{u}_x|_{y=0}\big\|_{A^k}^2
+ \big\|(\tilde{u}-u^s)|_{y=0}\big\|_{A^k}^2
+ \big\|(\eta-\bar{\eta})|_{y=0}\big\|_{A^k}^2 \Big) \\[9pt]

\lem (\lambda_3 +1)\int\limits_0^t\int\limits_{\mathbb{R}}\frac{e^{-2\lambda t}}{\beta-\eta|_{y=0}} (\partial_{t,x}^{k}w|_{y=0})^2 \,\mathrm{d}x\mathrm{d}t
+ \frac{(\lambda_{k})^2}{C(\max\{{\delta},\,{\beta-C_{\eta}}\})}\big\|w|_{y=0}\big\|_{A^3}^2.
\end{array}
\end{equation}

\vspace{-0.2cm}
The third estimate on $\{y=0\}$ is
\begin{equation}\label{Sect3_Estimate2_7_7}
\begin{array}{ll}
- \sum\limits_{k_1+k_2=k}\int\limits_0^t\int\limits_{\mathbb{R}} \partial_{t,x}^{k_1}(\frac{\tilde{\zeta}_1 + \tilde{\zeta}_2}{(\beta-\eta)^2})|_{y=0}
\partial_{t,x}^{k_2}w\,  e^{-2\lambda t}\partial_{t,x}^{k}w \,\mathrm{d}x\mathrm{d}t \\[9pt]

\lem (\frac{1}{\delta}\big|\tilde{\zeta}|_{y=0}\big|_{\infty} +1) \int\limits_0^t\int\limits_{\mathbb{R}}\frac{e^{-2\lambda t}}{\beta-\eta|_{y=0}} (\partial_{t,x}^{k}w|_{y=0})^2 \,\mathrm{d}x\mathrm{d}t  \\[7pt]\quad

+ \sum\limits_{0\leq k_2\leq k}(\big\|\partial_{t,x}^{k_2}\bar{\zeta}_1|_{y=0}\big\|_{L_{t,x}^{\infty}}^2
+ \big\|\partial_{t,x}^{k_2}\bar{\eta}|_{y=0}\big\|_{L_{t,x}^{\infty}}^2)
\frac{\big\|w|_{y=0}\big\|_{A^0}^2}{C(\max\{{\delta},\,{\beta-C_{\eta}}\})} \\[11pt]\quad

+ (\big\|\bar{\zeta}_2|_{y=0}\big\|_{A^k}^2 + \big\|(\eta-\bar{\eta})|_{y=0}\big\|_{A^k}^2)
\frac{|w|_{L_{t,x}^{\infty}}^2}{C(\max\{{\delta},\,{\beta-C_{\eta}}\})}  \\[8pt]

\lem (\lambda_2+1) \int\limits_0^t\int\limits_{\mathbb{R}}\frac{e^{-2\lambda t}}{\beta-\eta|_{y=0}} (\partial_{t,x}^{k}w|_{y=0})^2 \,\mathrm{d}x\mathrm{d}t
+ \frac{(\lambda_{k})^2}{C(\max\{{\delta},\,{\beta-C_{\eta}}\})}\big\|w|_{y=0}\big\|_{A^2}^2.
\end{array}
\end{equation}

\vspace{0.5cm}
Then
\begin{equation}\label{Sect3_Estimate2_8}
\begin{array}{ll}
\|e^{-\lambda t} \partial_{t,x}^{k}w\|_{L_{\ell}^2(\mathbb{R}_{+}^2)}^2
+ 2\lambda \|e^{-\lambda t} \partial_{t,x}^{k}w\|_{L_{\ell}^2([0,t]\times\mathbb{R}_{+}^2)}^2
+ 2\|e^{-\lambda t} \partial_{t,x}^{k}w_y\|_{L_{\ell}^2([0,t]\times\mathbb{R}_{+}^2)}^2 \\[11pt]\quad

+ \int\limits_{\mathbb{R}}\frac{e^{-2\lambda t}}{\beta-\eta|_{y=0}} (\partial_{t,x}^{k}w|_{y=0})^2
\,\mathrm{d}x
+ 2\lambda \int\limits_0^t\int\limits_{\mathbb{R}}\frac{e^{-2\lambda t}}{\beta-\eta|_{y=0}} (\partial_{t,x}^{k}w|_{y=0})^2 \,\mathrm{d}x\mathrm{d}t \\[11pt]

\leq  3 q\|w\|_{\mathcal{B}_{\lambda,\ell}^{k,1}}^2 + c_4(\lambda_3 +1)\int\limits_0^t\int\limits_{\mathbb{R}}\frac{e^{-2\lambda t}}{\beta-\eta|_{y=0}} (\partial_{t,x}^{k}w|_{y=0})^2 \,\mathrm{d}x\mathrm{d}t
+ C(q,(\lambda_{3})^2)\|w\|_{\mathcal{B}_{\lambda,\ell}^{k,0}}^2\\[14pt]\quad
+ C(q)\|\tilde{f}\|_{\mathcal{B}_{\lambda,\ell}^{k,0}}^2 + C(q)\cdot(\lambda_{k})^2\|w\|_{\mathcal{B}_{0,\ell}^{3,1}}^2
+ C(q)\cdot\frac{(\lambda_{k})^2}{C(\max\{{\delta},\,{\beta-C_{\eta}}\})}\big\|w|_{y=0}\big\|_{A^3}^2 \ .
\end{array}
\end{equation}

\vspace{0.5cm}
Sum $(\ref{Sect3_Estimate2_8})$ from $0$-th order to $k$-th order,
let $q$ to be small enough such that the coefficient of $\sum\limits_{0\leq m\leq k}\|w\|_{\mathcal{B}_{\lambda,\ell}^{m,1}}^2$ is less than $1$,
let $\lambda>0$  to be large enough such that not only the coefficient of $\sum\limits_{0\leq m\leq k}\|w\|_{\mathcal{B}_{\lambda,\ell}^{m,0}}^2$  but also the coefficient of $\sum\limits_{0\leq m\leq k}\int\limits_0^t\int\limits_{\mathbb{R}}\frac{e^{-2\lambda t}}{\beta-\eta|_{y=0}} (\partial_{t,x}^{m}w|_{y=0})^2 \,\mathrm{d}x\mathrm{d}t$
are less than $\lambda$, then we have
\begin{equation}\label{Sect3_Estimate2_9}
\begin{array}{ll}
\|w\|_{\tilde{\mathcal{B}}_{\lambda,\ell}^{k,0}}^2 + \lambda \|w\|_{\mathcal{B}_{\lambda,\ell}^{k,0}}^2
+\|w\|_{\mathcal{B}_{\lambda,\ell}^{k,1}}^2
+ \sum\limits_{0\leq m\leq k}\int\limits_{\mathbb{R}}\frac{e^{-2\lambda t}}{\beta-\eta} (\partial_{t,x}^{m}w|_{y=0})^2
\,\mathrm{d}x \\[9pt]\quad
+ \lambda \sum\limits_{0\leq m\leq k}\int\limits_0^T\int\limits_{\mathbb{R}}\frac{e^{-2\lambda t}}{\beta-\eta} (\partial_{t,x}^{m}w|_{y=0})^2 \,\mathrm{d}x\mathrm{d}t\\[15pt]

\leq \sum\limits_{0\leq m\leq k}\|\partial_{t,x}^{m}w|_{t=0}\|_{L_{\ell}^2(\mathbb{R}_{+}^2)}^2
+ \sum\limits_{0\leq m\leq k}\int\limits_{\mathbb{R}}\frac{1}{\beta-\eta|_{y=0,t=0}} (\partial_{t,x}^{m}w|_{y=0,t=0})^2
\,\mathrm{d}x \\[13pt]\quad

+ C(q)\|\tilde{f}\|_{\mathcal{B}_{\lambda,\ell}^{k,0}}^2 + C(q)\cdot(\lambda_{k})^2\|w\|_{\mathcal{B}_{0,\ell}^{3,1}}^2
+ C(q)\cdot\frac{(\lambda_{k})^2}{C(\max\{{\delta},\,{\beta-C_{\eta}}\})}\big\|w|_{y=0}\big\|_{A^3}^2.
\end{array}
\end{equation}

Fix $\lambda>0$ and note that $T<+\infty,\ w|_{t\leq 0} =0$, we have
\begin{equation}\label{Sect3_Estimate2_10}
\begin{array}{ll}
\|w\|_{\mathcal{B}_{0,\ell}^{k,1}} + \frac{1}{\sqrt{\beta + C_{\eta}}}\big\|w|_{y=0}\big\|_{A^k}
\leq C(q,\lambda T)\|\tilde{f}\|_{\mathcal{B}_{0,\ell}^{k,0}} \\[11pt]\quad
+ C(q,\lambda T)\cdot \lambda_{k}\|w\|_{\mathcal{B}_{0,\ell}^{3,1}}
+ C(q,\lambda T)\cdot\frac{\lambda_{k}}{C(\max\{\sqrt{\delta},\,\sqrt{\beta-C_{\eta}}\})}\big\|w|_{y=0}\big\|_{A^3}.
\end{array}
\end{equation}

Similar to $L^2$-estimate,
\begin{equation}\label{Sect3_Estimate2_11}
\begin{array}{ll}
\|w\|_{\mathcal{B}_{0,\ell}^{k,1}} + \frac{1}{\sqrt{\beta+C_{\eta}}}\|w|_{y=0}\|_{A^k}\lem \|\tilde{f}\|_{\mathcal{B}_{0,\ell}^{k,0}},\quad 0\leq k\leq 3.
\end{array}
\end{equation}

Thus, when $\lambda_{p}<+\infty$ where $p\geq 3$, for any $3\leq k\leq p$, we have
\begin{equation}\label{Sect3_Estimate2_12}
\begin{array}{ll}
\|w\|_{\mathcal{B}_{0,\ell}^{k,1}} + \frac{1}{\sqrt{\beta + C_{\eta}}}\big\|w|_{y=0}\big\|_{A^k}
\leq C(q,\lambda T)\|\tilde{f}\|_{\mathcal{B}_{0,\ell}^{k,0}} \\[9pt]\quad
+ C(q,\lambda T)\cdot\lambda_{k} \Big(\|w\|_{\mathcal{B}_{0,\ell}^{3,1}}
+ \frac{\sqrt{\beta+C_{\eta}}}{C(\max\{\sqrt{\delta},\,\sqrt{\beta-C_{\eta}}\})}\,\frac{1}{\sqrt{\beta+C_{\eta}}}\big\|w|_{y=0}\big\|_{A^3} \Big) \\[12pt]
\leq C(q,\lambda T)\|\tilde{f}\|_{\mathcal{B}_{0,\ell}^{k,0}} + C(q,\lambda T)\cdot\lambda_{k} \|\tilde{f}\|_{\mathcal{B}_{0,\ell}^{3,0}}
\leq C(q,\lambda T)\|\tilde{f}\|_{\mathcal{B}_{0,\ell}^{k,0}} \ ,
\end{array}
\end{equation}
where $\frac{\sqrt{\beta+C_{\eta}}}{C(\max\{\sqrt{\delta},\,\sqrt{\beta-C_{\eta}}\})}<+\infty$ for $0<\delta_{\beta}\leq \beta\leq +\infty$.

\vspace{0.4cm}
3. mixed derivatives estimates:

It follows from $(\ref{Sect3_VorticityEq})$ that
\begin{equation}\label{Sect3_Estimate3_0}
\begin{array}{ll}
w_{yy} = w_t + (\tilde{u} w)_x + (\tilde{v}w)_y - 2(\eta w)_y
- \partial_y\zeta \int\limits_y^{+\infty} w(t,x,\tilde{y}) \,\mathrm{d}\tilde{y} + \zeta w - \partial_y\tilde{f}.
\end{array}
\end{equation}

Applying $\partial_{t,x}^{k_1}$ to $(\ref{Sect3_Estimate3_0})$, it is easy to get the estimate:
\begin{equation}\label{Sect3_Estimate3_1}
\begin{array}{ll}
\|w\|_{\mathcal{B}_{0,\ell}^{k_1,2}}
\lem \|\tilde{f}\|_{\mathcal{B}_{0,\ell}^{k_1,1}}+ (\lambda_{3}+1)(\|w\|_{\mathcal{B}_{0,\ell}^{k_1+1,0}}
+ \|w\|_{\mathcal{B}_{0,\ell}^{k_1,1}})+ \lambda_{k_1+1}\|w\|_{\mathcal{B}_{0,\ell}^{2,1}} \ .
\end{array}
\end{equation}

While by $(\ref{Sect3_Estimate2_10})$,  $(\ref{Sect3_Estimate2_11})$ and $(\ref{Sect3_Estimate2_12})$, we have
\begin{equation}\label{Sect3_Estimate3_2}
\begin{array}{ll}
\|w\|_{\mathcal{B}_{0,\ell}^{k_1+1,0}}
\leq C(q,\lambda T)\|\tilde{f}\|_{\mathcal{B}_{0,\ell}^{k_1+1,0}} + C(q,\lambda T)\cdot\lambda_{k_1+1}\|\tilde{f}\|_{\mathcal{B}_{0,\ell}^{3,0}} \ , \\[10pt]
\|w\|_{\mathcal{B}_{0,\ell}^{k_1,1}}
\leq C(q,\lambda T)\|\tilde{f}\|_{\mathcal{B}_{0,\ell}^{k_1,0}} + C(q,\lambda T)\cdot\lambda_{k_1}\|\tilde{f}\|_{\mathcal{B}_{0,\ell}^{3,0}} \ , \\[10pt]
\|w\|_{\mathcal{B}_{0,\ell}^{2,1}}
\leq C(q,\lambda T)\|\tilde{f}\|_{\mathcal{B}_{0,\ell}^{2,0}} \ .
\end{array}
\end{equation}

Plug $(\ref{Sect3_Estimate3_2})$ into $(\ref{Sect3_Estimate3_1})$, we get
\begin{equation}\label{Sect3_Estimate3_3}
\begin{array}{ll}
\|w\|_{\mathcal{B}_{0,\ell}^{k_1,2}}
\leq \|\tilde{f}\|_{\mathcal{B}_{0,\ell}^{k_1,1}}
+ C(q,\lambda T,\lambda_3)\|\tilde{f}\|_{\mathcal{B}_{0,\ell}^{k_1+1,0}}
+ C(q,\lambda T)\cdot\lambda_{k_1+1}\|\tilde{f}\|_{\mathcal{B}_{0,\ell}^{3,0}} \ .
\end{array}
\end{equation}

When $k_2>2$, apply $\partial_{t,x}^{k_1}\partial_y^{k_2 -2}$ to $(\ref{Sect3_Estimate3_0})$, we get
\begin{equation}\label{Sect3_Estimate3_4}
\begin{array}{ll}
\partial_{t,x}^{k_1}\partial_y^{k_2}w = \partial_t\partial_{t,x}^{k_1}\partial_y^{k_2 -2}w + \partial_x\partial_{t,x}^{k_1}\partial_y^{k_2 -2}(\tilde{u} w)
+ \partial_{t,x}^{k_1}\partial_y^{k_2 -1}(\tilde{v}w) - 2\partial_{t,x}^{k_1}\partial_y^{k_2 -1}(\eta w) \\[8pt]\hspace{1.8cm}
- \partial_{t,x}^{k_1}\partial_y^{k_2-1}\Big( \zeta \int\limits_y^{+\infty} w(t,x,\tilde{y}) \,\mathrm{d}\tilde{y} \Big)
- \partial_{t,x}^{k_1}\partial_y^{k_2 -1}\tilde{f}, \\[10pt]
\end{array}
\end{equation}
then
\begin{equation}\label{Sect3_Estimate3_5}
\begin{array}{ll}
\|w\|_{\mathcal{B}_{0,\ell}^{k_1,k_2}}
\lem \|w\|_{\mathcal{B}_{0,\ell}^{k_1+1,k_2-2}} + \|\tilde{u}\|_{L^{\infty}}\| w\|_{\mathcal{B}_{0,\ell}^{k_1+1,k_2-2}}
+ \|\tilde{u} - u^s\|_{\mathcal{B}_{0,0}^{k_1+1,k_2-2}}\|w\|_{L_{\ell}^{\infty}} \\[12pt]\quad
+ \|\partial_{t,x}^{k_1+1}\partial_y^{k_2 -2} u^s\|_{L_y^2(L_{t}^{\infty})}\|w\|_{L_{y,\ell}^{\infty}(L_{t,x}^2)}
+ \|\tilde{v}\|_{L^{\infty}}\| w\|_{\mathcal{B}_{0,\ell}^{k_1,k_2-1}} \\[12pt]\quad
+ \|\partial_{t,x}^{k_1}\partial_y^{k_2 -1}\tilde{v}\|_{L_y^{\infty}(L_{t,x}^2)}\|w\|_{L_{y,\ell}^2(L_{t,x}^{\infty})}

+ 2\|\partial_{t,x}^{k_1}\partial_y^{k_2-1}\bar{\eta}\|_{L_y^2(L_{t,x}^{\infty})}\|w\|_{L_{y,\ell}^{\infty}(L_{t,x}^2)} \\[12pt]\quad
+ 2\|\eta-\bar{\eta}\|_{\mathcal{B}_{0,\ell}^{k_1,k_2-1}}\|w\|_{L^{\infty}}
+ 2\|\eta\|_{L^{\infty}}\|w\|_{\mathcal{B}_{0,\ell}^{k_1,k_2- 1}}
+ \|\tilde{f}\|_{\mathcal{B}_{0,\ell}^{k_1,k_2-1}} \\[12pt]\quad

+ \|\zeta\|_{\mathcal{B}_{0,\ell}^{k_1,k_2-1}}\|w\|_{L_{y,\ell}^2(L_{t,x}^{\infty})}
+\|\zeta\|_{L^{\infty}}\|w\|_{\mathcal{B}_{0,\ell}^{k_1,k_2-2}} \\[14pt]

\lem \|\tilde{f}\|_{\mathcal{B}_{0,\ell}^{k_1,k_2-1}}
+ (\lambda_{3}+1)(\|w\|_{\mathcal{B}_{0,\ell}^{k_1+1,k_2-2}} + \|w\|_{\mathcal{B}_{0,\ell}^{k_1,k_2-1}}) \\[12pt]\quad
+ (\lambda_{k_1+1,k_2-2} + \lambda_{k_1,k_2-1})\|w\|_{\mathcal{B}_{0,\ell}^{2,1}} \ .
\end{array}
\end{equation}

Since $\lambda_3<+\infty$, $(\ref{Sect3_Estimate3_5})$ implies
\begin{equation}\label{Sect3_Estimate3_6}
\begin{array}{ll}
\|w\|_{\mathcal{B}_{0,\ell}^{k_1,k_2}}

\lem \|\tilde{f}\|_{\mathcal{B}_{0,\ell}^{k_1,k_2-1}}
+ \|w\|_{\mathcal{B}_{0,\ell}^{k_1+1,k_2-2}} + \|w\|_{\mathcal{B}_{0,\ell}^{k_1,k_2-1}} \\[10pt]\hspace{2cm}
+ (\lambda_{k_1+1,k_2-2} + \lambda_{k_1,k_2-1}) \|w\|_{\mathcal{B}_{0,\ell}^{2,1}} \ .
\end{array}
\end{equation}

Apply the induction method to $(\ref{Sect3_Estimate3_6})$, we get the estimates:
\begin{equation}\label{Sect3_Estimate3_8}
\begin{array}{ll}
\|w\|_{\mathcal{A}_{\ell}^{k}}
\lem \|\tilde{f}\|_{\mathcal{A}_{\ell}^{k}}
+ \lambda_{k} \|\tilde{f}\|_{\mathcal{A}_{\ell}^3}.
\end{array}
\end{equation}

Since $\lambda_{p}<+\infty$, for any $0\leq k\leq p$, we have $\|w\|_{\mathcal{A}_{\ell}^{k}}\lem \|\tilde{f}\|_{\mathcal{A}_{\ell}^{k}}$.

By $(\ref{Sect3_Estimate2_10}),(\ref{Sect3_Estimate2_11}),(\ref{Sect3_Estimate2_12}),(\ref{Sect3_Estimate3_8})$, Theorem $\ref{Sect3_Main_Estimate_Thm}$ is proved.
\end{proof}

\begin{remark}\label{Sect3_Coefficience_Remark}
In \cite{Alexandre_Yang_2014}, the linearized Prandtl-type equations with Dirichlet boundary condition has the following estimate
with the different definition of $\lambda_k$:
\begin{equation}\label{Sect3_Coefficience_Remark_Estimate}
\begin{array}{ll}
\|w\|_{\mathcal{A}_{\ell}^k} \leq C_1(\lambda_3) \|\tilde{f}\|_{\mathcal{A}_{\ell}^k} + C_2(\lambda_3)\lambda_k \|\tilde{f}\|_{\mathcal{A}_{\ell}^3}.
\end{array}
\end{equation}
While in this paper, the estimate of $\frac{1}{\sqrt{\beta+C_{\eta}}}\big\|w|_{y=0}\big\|_{A^k}$ is necessary.
\end{remark}

\subsection{Case II: Zero Force, Nonzero Data}
Let $(\tilde{u},\tilde{v})$ be a smooth background state satisfying
\begin{equation}\label{Sect3_Background_Condition_2}
\left\{\begin{array}{ll}
\tilde{u}>0,\quad \tilde{u}_y>0,\quad \beta -\frac{\tilde{u}_{yy}}{\tilde{u}_y} \geq \delta>0,\quad \forall y\in [0,+\infty), \\[7pt]
\tilde{u}_x + \tilde{v}_y =0, \\[7pt]
\tilde{v}|_{y=0}=0, \\[7pt]
\lim\limits_{y\rto +\infty} \tilde{u}(t,x,y)=1,
\end{array}\right.
\end{equation}
where $(\tilde{u}_y -\beta\tilde{u})|_{y=0} =0$ is unnecessary to be satisfied and will not be used here.

We consider the following linearized Prandtl-type equations with Robin boundary condition around the background state $(\tilde{u},\tilde{v})$, with zero force and nonzero data:
\begin{equation}\label{Sect3_Linearized_Prandtl_ZeroForce}
\left\{\begin{array}{ll}
u_t + \tilde{u} u_x + \tilde{v} u_y + u \tilde{u}_x +  v \tilde{u}_y - u_{yy} = 0, \\[8pt]
u_x + v_y =0, \\[8pt]
(\partial_y u - \beta u)|_{y=0} =0,\quad  v|_{y=0} =0,\\[8pt]
\lim\limits_{y\rto +\infty} u(t,x,y) =0, \\[8pt]
u|_{t\leq 0} = u_0.
\end{array}\right.
\end{equation}

Then $w = (\frac{u}{\tilde{u}_y})_y$ satisfies the following IBVP:
\begin{equation}\label{Sect3_VorticityEq_ZeroForce}
\left\{\begin{array}{ll}
w_t + (\tilde{u}w)_x + (\tilde{v}w)_y - 2(\eta w)_y
- \big(\zeta \int\limits_y^{+\infty} w(t,x,\tilde{y}) \,\mathrm{d}\tilde{y} \big)_y - w_{yy} = 0, \\[11pt]

\frac{w_t}{\beta- \eta|_{y=0}} + \frac{(\tilde{u}w)_x}{\beta- \eta|_{y=0}}
- \Big(w_y + 2\eta|_{y=0} w + \tilde{f} + \zeta|_{y=0}\int\limits_0^{+\infty} w(t,x,\tilde{y}) \,\mathrm{d}\tilde{y} \Big) \\[11pt]\hspace{0.5cm}

+ \frac{w\tilde{\zeta}|_{y=0}}{(\beta-\eta|_{y=0})^2} =0,\quad y=0, \\[12pt]

w|_{t\leq 0} =w_0 :=(\frac{u_0}{\partial_y\tilde{u}_0})_y \, .
\end{array}\right.
\end{equation}

When $\lambda_p<+\infty$, we have the following theorem for the equations $(\ref{Sect3_VorticityEq_ZeroForce})$:
\begin{theorem}\label{Sect3_Main_Estimate_ZeroForce_Thm}
Assume $0<\delta_{\beta}\leq\beta<+\infty$, $\ell>\frac{1}{2},\ \beta-\eta|_{y=0}\geq\delta>0,\ |\eta|_{\infty}\leq C_{\eta},\ \lambda_{p}<+\infty$ where $p\geq 3$, then we have the following a priori estimate for $(\ref{Sect3_VorticityEq_ZeroForce})$:
\begin{equation}\label{Sect3_Main_Estimate_Eq_ZeroForce}
\begin{array}{ll}
\|w\|_{\mathcal{A}_{\ell}^{k}} + \frac{1}{\sqrt{\beta + C_{\eta}}}\big\|w|_{y=0}\big\|_{A^{k}} \\[12pt]
\leq C(T,\lambda_p)\Big(\|w_0\|_{\mathcal{A}_{\ell}^{k}(\mathbb{R}_{+}^2,t=0)}
+ \frac{1}{\max\{\sqrt{\beta -C_{\eta}},\sqrt{\delta}\}} \big\|w_0|_{y=0} \big\|_{A^k(\mathbb{R},t=0)} \Big),\quad k\leq p,
\end{array}
\end{equation}
where the time derivatives of the initial data can be represented by the space derivatives of the initial data by solving the equations.
\end{theorem}

\begin{proof}
1. $L^2$-estimate:

Since $\tilde{f}\equiv 0$, it follows from $(\ref{Sect3_Estimate1_8})$ that
\begin{equation}\label{Sect3_Estimate4_1}
\begin{array}{ll}
\|e^{-\lambda t} w\|_{L_{\ell}^2(\mathbb{R}_{+}^2)}^2
+\lambda \|e^{-\lambda t}<y>^{\ell} w\|_{L^2([0,t]\times\mathbb{R}_{+}^2)}^2
+ \|e^{-\lambda t}<y>^{\ell} w_y\|_{L^2([0,t]\times\mathbb{R}_{+}^2)}^2 \\[10pt]\quad

+ \int\limits_{\mathbb{R}}\frac{e^{-2\lambda t}}{\beta-\eta|_{y=0}} (w|_{y=0})^2
\,\mathrm{d}x
+ \lambda \int\limits_0^t\int\limits_{\mathbb{R}}\frac{e^{-2\lambda t}}{\beta-\eta|_{y=0}} (w|_{y=0})^2 \,\mathrm{d}x\mathrm{d}t\\[14pt]

\leq \|w|_{t=0}\|_{\mathcal{A}_{\ell}^{0}(\mathbb{R},t=0)}^2
+ \int\limits_{\mathbb{R}}\frac{1}{\beta-\eta|_{y=0,t=0}} (w|_{y=0,t=0})^2
\,\mathrm{d}x.
\end{array}
\end{equation}

Fix $\lambda>0$ and note that $T<+\infty,\ w|_{t\leq 0} =0$, for any $t\in [0,T]$, we have
\begin{equation}\label{Sect3_Estimate4_2}
\begin{array}{ll}
\|w\|_{\mathcal{B}_{0,\ell}^{0,1}} + \frac{1}{\sqrt{\beta + C_{\eta}}}\big\|w|_{y=0}\big\|_{A^0} \\[12pt]
\leq C(\lambda T)\|w|_{t=0}\|_{\mathcal{A}_{\ell}^{0}(\mathbb{R}_{+}^2,t=0)}
+ \frac{C(\lambda T)}{\max\{\sqrt{\beta -C_{\eta}},\sqrt{\delta}\}} \big\|w|_{y=0,t=0} \big\|_{A^0(\mathbb{R},t=0).}
\end{array}
\end{equation}

\vspace{0.2cm}
2. tangential derivatives estimates:

Since $\tilde{f}\equiv 0$, it follows from $(\ref{Sect3_Estimate2_9})$ that
\begin{equation}\label{Sect3_Estimate5_1}
\begin{array}{ll}
\|w\|_{\mathcal{B}_{\lambda,\ell}^{k,1}}^2
+ \frac{1}{\beta + C_{\eta}}\big\|e^{-\lambda t} w|_{y=0} \big\|_{A^k}^2 \\[12pt]

\leq \sum\limits_{0\leq m\leq k}\|\partial_{t,x}^{m}w|_{t=0}\|_{L_{\ell}^2(\mathbb{R}_{+}^2)}^2
+ \sum\limits_{0\leq m\leq k}\int\limits_{\mathbb{R}}\frac{1}{\beta-\eta|_{y=0,t=0}} (\partial_{t,x}^{m}w|_{y=0,t=0})^2
\,\mathrm{d}x \\[15pt]\quad

+ C(q)\cdot(\lambda_{k})^2\|w\|_{\mathcal{B}_{\lambda,\ell}^{3,1}}^2
+ C(q)\cdot\frac{(\lambda_{k})^2}{C(\max\{\delta,\beta-C_{\eta}\})}\big\|w|_{y=0}\big\|_{A^3}^2.
\end{array}
\end{equation}

Fix $\lambda>0$ and note that $T<+\infty$, we have
\begin{equation}\label{Sect3_Estimate5_2}
\begin{array}{ll}
\|w\|_{\mathcal{B}_{0,\ell}^{k,1}}
+ \frac{1}{\sqrt{\beta + C_{\eta}}}\big\|w|_{y=0} \big\|_{A^k}

\leq C(\lambda T)\sum\limits_{0\leq m\leq k}\|\partial_{t,x}^{m}w|_{t=0}\|_{L_{\ell}^2(\mathbb{R}_{+}^2)} \\[15pt]\quad
+ \frac{C(\lambda T)}{\max\{\sqrt{\beta -C_{\eta}},\sqrt{\delta}\}}
\big\|w|_{y=0,t=0} \big\|_{A^k(\mathbb{R},t=0)}

+ C(q,\lambda T)\cdot\lambda_{k}\|w\|_{\mathcal{B}_{\lambda,\ell}^{3,1}} \\[14pt]\quad
+ C(q,\lambda T)\lambda_{k} \cdot\frac{\sqrt{\beta + C_{\eta}}}{C(\max\{\sqrt{\delta},\sqrt{\beta-C_{\eta}}\})}
\frac{1}{\sqrt{\beta + C_{\eta}}}\big\|w|_{y=0}\big\|_{A^3},
\end{array}
\end{equation}
where $\frac{\sqrt{\beta + C_{\eta}}}{C(\max\{\sqrt{\delta},\sqrt{\beta-C_{\eta}}\})}<+\infty$ for $0<\delta_{\beta}\leq \beta\leq +\infty$.

Similar to $L^2$-estimate, for any $0\leq k\leq 3$,
\begin{equation}\label{Sect3_Estimate5_3}
\begin{array}{ll}
\|w\|_{\mathcal{B}_{0,\ell}^{k,1}} + \frac{1}{\sqrt{\beta + C_{\eta}}}\big\|w|_{y=0} \big\|_{A^k} \\[11pt]
\lem \|w|_{t=0}\|_{\mathcal{A}_{\ell}^k(\mathbb{R}_{+}^2,t=0)}
+ \frac{1}{\max\{\sqrt{\beta -C_{\eta}},\sqrt{\delta}\}}
\big\|w|_{y=0,t=0} \big\|_{A^k(\mathbb{R},t=0)} .
\end{array}
\end{equation}

Plug $(\ref{Sect3_Estimate5_3})$ into $(\ref{Sect3_Estimate5_2})$ and note that $\lambda_{p}<+\infty$ where $p\geq 3$, we have
\begin{equation}\label{Sect3_Estimate5_4}
\begin{array}{ll}
\|w\|_{\mathcal{B}_{0,\ell}^{k,1}} + \frac{1}{\sqrt{\beta + C_{\eta}}}\big\|w|_{y=0} \big\|_{A^k} \\[11pt]

\lem \|w|_{t=0}\|_{\mathcal{A}_{\ell}^k(\mathbb{R}_{+}^2,t=0)}
+ \frac{1}{\max\{\sqrt{\beta -C_{\eta}},\sqrt{\delta}\}}
\big\|w|_{y=0,t=0} \big\|_{A^k(\mathbb{R},t=0)}\, , \quad 3\leq k\leq p.
\end{array}
\end{equation}

\vspace{0.2cm}
3. mixed derivatives estimates:
When $k=1$, we estimate $(\ref{Sect3_Estimate3_0})$ directly and note that $\partial_y \tilde{f} =0$, we get
\begin{equation}\label{Sect3_Estimate6_1}
\begin{array}{ll}
\|w\|_{\mathcal{B}_{0,\ell}^{0,2}} \lem \|w\|_{\mathcal{B}_{0,\ell}^{1,0}} + \|w\|_{\mathcal{B}_{0,\ell}^{0,1}}\, .
\end{array}
\end{equation}

By $(\ref{Sect3_Estimate5_4})$ and $(\ref{Sect3_Estimate6_1})$, we have
\begin{equation}\label{Sect3_Estimate6_2}
\begin{array}{ll}
\|w\|_{\mathcal{A}_{\ell}^1}
+ \frac{1}{\sqrt{\beta + C_{\eta}}}\big\|w|_{y=0} \big\|_{A^1} \\[12pt]

\lem \big\|w|_{t=0}\big\|_{\mathcal{A}_{\ell}^1(\mathbb{R}_{+}^2,t=0)}
+ \frac{1}{\max\{\sqrt{\beta -C_{\eta}},\sqrt{\delta}\}}
\big\|w|_{y=0,t=0} \big\|_{A^1(\mathbb{R},t=0).}
\end{array}
\end{equation}

When $2\leq k\leq p$, we need to estimate mixed derivatives.

Since $\tilde{f}\equiv 0$, it follows from $(\ref{Sect3_Estimate3_6})$ that
\begin{equation}\label{Sect3_Estimate6_3}
\begin{array}{ll}
\|w\|_{\mathcal{B}_{0,\ell}^{k_1,k_2}}

\lem \|w\|_{\mathcal{B}_{0,\ell}^{k_1+1,k_2-2}} + \|w\|_{\mathcal{B}_{0,\ell}^{k_1,k_2-1}}
+ (\lambda_{k_1+1,k_2-2} + \lambda_{k_1,k_2-1}) \|w\|_{\mathcal{B}_{0,\ell}^{2,1}}\ .
\end{array}
\end{equation}

Apply the induction method to $(\ref{Sect3_Estimate6_3})$, we get the estimates:

\begin{equation}\label{Sect3_Estimate6_4}
\begin{array}{ll}
\|w\|_{\mathcal{B}_{0,\ell}^{k_1,k_2}}
\lem \sum\limits_{0\leq m\leq [\frac{k_2}{2}]}\|w\|_{\mathcal{B}_{0,\ell}^{k_1+m,1}} + \lambda_{p} \|w\|_{\mathcal{B}_{0,\ell}^{2,1}}\ .
\end{array}
\end{equation}

By the inequalities $(\ref{Sect3_Estimate6_4})$ and $(\ref{Sect3_Estimate5_3})$, we get
\begin{equation}\label{Sect3_Estimate6_5}
\begin{array}{ll}
\|w\|_{\mathcal{A}_{\ell}^{k}}
\lem \|w\|_{\mathcal{B}_{0,\ell}^{k,1}}
+ \lambda_{p} \|w\|_{\mathcal{B}_{0,\ell}^{2,1}} \\[15pt]

\lem \big\|w|_{t=0}\big\|_{\mathcal{A}_{\ell}^k(\mathbb{R}_{+}^2,t=0)}
+ \frac{1}{\max\{\sqrt{\beta -C_{\eta}},\sqrt{\delta}\}}
\big\|w|_{y=0,t=0} \big\|_{A^k(\mathbb{R},t=0)}, \quad 2\leq k\leq p.
\end{array}
\end{equation}

By $(\ref{Sect3_Estimate5_3}), (\ref{Sect3_Estimate5_4}), (\ref{Sect3_Estimate6_2}), (\ref{Sect3_Estimate6_5})$, Theorem $\ref{Sect3_Main_Estimate_ZeroForce_Thm}$ is proved.
\end{proof}

\section{Iteration Scheme for the Nonlinear Prandtl Equations with Robin Boundary Condition}
In this section, we establish the Nash-Moser-H$\ddot{o}$rmander iteration scheme and construct the approximate solutions satisfying Robin boundary condition. To start the iteration, we need to construct the zero-th order approximate solution satisfying Robin boundary condition.

\subsection{Establishment of the Nash-Moser-H$\ddot{o}$rmander Iteration Scheme}
In order to mollify some variables and quantities in the Nash-Moser-H$\ddot{o}$rmander iteration scheme, we need to define three smoothing operators $S_{\theta},\ S_{\theta}^u,\ S_{\theta}^v$ and extension operators $\mathrm{E},\ \mathrm{E}^u,\ \mathrm{E}^v$.

Let $\varrho_{\theta}(\tau)=\theta\varrho(\theta\tau),\ \varrho\in C_0^{\infty}(\mathbb{R}),
\ \varrho(-\tau) =\varrho(\tau),\ \varrho\geq 0,\ \|\varrho\|_{L^1}=1,\ Supp \varrho\subset [-1,1]$. We define the following smoothing operator for which there is no restriction near the boundary.
\begin{equation*}
\begin{array}{ll}
(S_{\theta} f)(t,x,y) = \iiint\limits_{\mathbb{R}^3} \varrho_{\theta}(\tau)\varrho_{\theta}(\xi)\varrho_{\theta}(\eta)
E f(t-\tau+\theta^{-1},x-\xi,y-\eta+\theta^{-1})\,\mathrm{d}\tau\mathrm{d}\xi\mathrm{d}\eta,
\end{array}
\end{equation*}
where $E f$ is the extension of $f$ to $\mathbb{R}^3$ by zero. Note that $(S_{\theta} f)(t,x,0)$ may not vanish.

Similar to $\cite{Alexandre_Yang_2014}$, $S_{\theta}$ has the following properties:
\begin{proposition}\label{Sect4_Nonlinear_Op_Proposition_1}
Assume $\|\frac{u_{yy}^s}{u_y^s}\|_{\mathcal{C}_0^{k_0+2}} <+\infty$,
the operator $S_{\theta}$ follows the following rules:
\begin{equation}\label{Sect4_Nonlinear_Op_Properties}
\begin{array}{ll}
\|S_{\theta} f\|_{\mathcal{A}_{\ell}^{s}} \leq C_{\varrho}\theta^{(s-\alpha)_{+}}\|f\|_{\mathcal{A}_{\ell}^{\alpha}},
\hspace{2.43cm} \forall s,\alpha\geq 0, \\[12pt]
\|(1 - S_{\theta}) f\|_{\mathcal{A}_{\ell}^{s}} \leq C_{\varrho}\theta^{s-\alpha}\|f\|_{\mathcal{A}_{\ell}^{\alpha}},
\hspace{1.96cm} \forall 0\leq s\leq\alpha, \\[12pt]
\|(S_{\theta_n} - S_{\theta_{n-1}}) f\|_{\mathcal{A}_{\ell}^{s}} \leq C_{\varrho}\theta_n^{s-\alpha}\triangle\theta_n \|f\|_{\mathcal{A}_{\ell}^{\alpha}},
\hspace{0.44cm} \forall s,\alpha\geq 0, \\[14pt]

\|[\frac{1}{\partial_y u^s},S_{\theta}]f\|_{\mathcal{A}_{\ell}^k} \leq C_{\varrho} \|\frac{f}{\partial_y u^s}\|_{\mathcal{A}_{\ell}^k},
\hspace{2.05cm} \forall k\leq k_0,\\[14pt]
\|\partial_y[\frac{1}{\partial_y u^s},S_{\theta}]f\|_{\mathcal{A}_{\ell}^k} \leq C_{\varrho} \theta\|\frac{f}{\partial_y u^s}\|_{\mathcal{A}_{\ell}^k},
\hspace{1.52cm} \forall k\leq k_0,
\end{array}
\end{equation}
where $\theta_n =\sqrt{\theta_0^2 + n}$, $\theta_0\gg 1$,
$\triangle\theta_n =\theta_{n+1} - \theta_n$.
\end{proposition}

\begin{remark}\label{Sect4_Remark_Dirichlet}
For the Prandtl equations with Dirichlet boundary condition, the smooth operator $S_{\theta}$ needs to be defined as
\begin{equation*}
\begin{array}{ll}
(S_{\theta} f)(t,x,y) = \iiint\limits_{\mathbb{R}^3} \varrho_{\theta}(\tau)\varrho_{\theta}(\xi)\varrho_{\theta}(\eta)
E f(t-\tau+\theta^{-1},x-\xi,y-\eta-\theta^{-1})\,\mathrm{d}\tau\mathrm{d}\xi\mathrm{d}\eta,
\end{array}
\end{equation*}
such that $S_{\theta} u|_{y=0} =S_{\theta} v|_{y=0} =0$ when $u|_{y=0}=v|_{y=0}=0$. See \cite{Alexandre_Yang_2014}, where
$S_{\theta} u|_{y=0} =S_{\theta} v|_{y=0} =0$ are used in the derivation of boundary conditions for the linearized Prandtl-type equations.
\end{remark}

\vspace{0.1cm}
In order to introduce the operators $S_{\theta}^u,S_{\theta}^v,E^u,E^v$ simply, we still use $u(t,x,y),v(t,x,y)$ here as the abstract functions, where  $\lim\limits_{y\rto +\infty} u(t,x,y)=0$, $u\in\mathcal{A}_{\ell}^{s_1}, v\in\mathcal{D}_0^{s_2}$ for some $s_1,s_2$. Actually, the functions that need to be mollified are $\tilde{u}^0, \delta u^j$ and $\tilde{v}^0,\delta v^j$ instead of $u,v$.

We define the smoothing operators $S_{\theta}^u,S_{\theta}^v$ as follows:
\begin{equation*}
\begin{array}{ll}
(S_{\theta}^u u)(t,x,y) = \iiint\limits_{\mathbb{R}^3} \varrho_{\theta}(\tau)\varrho_{\theta}(\xi)\varrho_{\theta}(\eta)
E^u u(t-\tau+\theta^{-1},x-\xi,y-\eta)\,\mathrm{d}\tau\mathrm{d}\xi\mathrm{d}\eta, \\[14pt]

(S_{\theta}^v v)(t,x,y) = \iiint\limits_{\mathbb{R}^3} \varrho_{\theta}(\tau)\varrho_{\theta}(\xi)\varrho_{\theta}(\eta)
E^v v(t-\tau+\theta^{-1},x-\xi,y-\eta)\,\mathrm{d}\tau\mathrm{d}\xi\mathrm{d}\eta,
\end{array}
\end{equation*}
where the extension operators $E^u, E^v$ are defined as follows:
\begin{equation}\label{Sect4_Extended_Fuctions}
\begin{array}{ll}
E^u u(t,x,y) := \left\{\begin{array}{ll}
u(t,x,y),\hspace{1.32cm} y\geq 0, \\[9pt]
u(t,x,-y), \hspace{0.99cm} -\theta^{-1}<y<0, \\[9pt]
0, \hspace{2.41cm} y\leq -\theta^{-1}.
\end{array}\right. \\[33pt]

E^v v(t,x,y) := \left\{\begin{array}{ll}
v(t,x,y),\hspace{1.34cm} y\geq 0, \\[9pt]
-v(t,x,-y),\hspace{0.75cm} -\theta^{-1}<y<0, \\[9pt]
0, \hspace{2.43cm} y\leq -\theta^{-1}.
\end{array}\right.
\end{array}
\end{equation}

\begin{remark}\label{Sect4_Operator_SU_Remark}
Though $(\partial_y u -\beta u)|_{y=0} =0$, it is unnecessary to have $[\partial_y (S_{\theta}^u u) -\beta (S_{\theta}^u u)]|_{y=0} =0$, because
$(\ref{Sect3_Background_Condition_1})$, $(\ref{Sect3_Background_Condition_2})$ do not need the condition $(\partial_y\tilde{u} -\beta\tilde{u})|_{y=0} =0$.
Since $u|_{y=0}$ may not vanish, $(S_{\theta}^u u)|_{y=0}$ does not equal zero in general,
due to the even symmetry of $E^u u$ with respect to $y$.
\end{remark}

The operators $E^u,E^v,S_{\theta}^u, S_{\theta}^v$ have the following basic properties:
\begin{proposition}\label{Sect4_Nonlinear_Op_Proposition_2}
If $u_x+v_y=0,\ v|_{y=0}=0$, then
\begin{equation}\label{Sect4_Nonlinear_Op_Properties_2}
\begin{array}{ll}
(S^v_{\theta} v)|_{y=0} =0, \\[10pt]
(S^u_{\theta} u)_x + (S^v_{\theta} v)_y = 0,\\[10pt]

\|E^u u\|_{\mathcal{A}_{\ell}^{\alpha}([0,T]\times\mathbb{R}^2)} \leq 2\|u\|_{\mathcal{A}_{\ell}^{\alpha}([0,T]\times\mathbb{R}_{+}^2)}, \\[10pt]
\|E^v v\|_{\mathcal{D}_0^{\alpha}([0,T]\times\mathbb{R}^2)} \leq 2\|v\|_{\mathcal{D}_0^{\alpha}([0,T]\times\mathbb{R}_{+}^2)}.
\end{array}
\end{equation}
\end{proposition}

\begin{proof}
Since $(E^u u)_x + (E^v v)_y =0$ in $[0,T]\times\mathbb{R}^2$, we have $(S^u_{\theta} u)_x + (S^v_{\theta} v)_y = 0$.
It is obvious that $(S^v_{\theta} v)|_{y=0} =0$ by the definition of $S_{\theta}^v$ and the odd symmetry of $E^v v$ with respect to $y$.

By the definition of $(\ref{Sect4_Extended_Fuctions})$, it is easy to prove $(\ref{Sect4_Nonlinear_Op_Properties_2})_3$ and $(\ref{Sect4_Nonlinear_Op_Properties_2})_4$.
\end{proof}

Furthermore, $S_{\theta}^u$ and $S_{\theta}^v$ have the following properties:
\begin{proposition}\label{Sect4_Nonlinear_Op_Proposition_3}
Assume $\|\frac{u_{yy}^s}{u_y^s}\|_{\mathcal{C}_0^{k_0+2}} <+\infty$,
the operators $S_{\theta}^u$ and $S_{\theta}^v$ follow the following rules:
\begin{equation}\label{Sect4_Nonlinear_Op_Properties_3}
\begin{array}{ll}
\|S_{\theta}^u u\|_{\mathcal{A}_{\ell}^{s}} \leq C_{\varrho}\theta^{(s-\alpha)_{+}}\|u\|_{\mathcal{A}_{\ell}^{\alpha}},\hspace{2.47cm}
\forall s\geq 0, \\[10pt]

\|(1 - S_{\theta}^u) u\|_{\mathcal{A}_{\ell}^{s}} \leq C_{\varrho}\theta^{s-\alpha}\|u\|_{\mathcal{A}_{\ell}^{\alpha}},\hspace{2cm}
\forall 0\leq s\leq\alpha, \\[10pt]

\|(S_{\theta_n}^u - S_{\theta_{n-1}}^u) u\|_{\mathcal{A}_{\ell}^{s}} \leq C_{\varrho}\theta_n^{s-\alpha}\triangle\theta_n
\|u\|_{\mathcal{A}_{\ell}^{\alpha}},\hspace{0.53cm}
\forall s, \alpha\geq 0,
\end{array}
\end{equation}

\begin{equation*}
\begin{array}{ll}
\|[\frac{1}{\partial_y u^s},S_{\theta}^u]u\|_{\mathcal{A}_{\ell}^k}
\leq C_{\varrho} \|\frac{u}{\partial_y u^s}\|_{\mathcal{A}_{\ell}^k}, \hspace{2.1cm} \forall k\leq k_0,\\[10pt]
\|\partial_y[\frac{1}{\partial_y u^s},S_{\theta}^u]u\|_{\mathcal{A}_{\ell}^k}
\leq C_{\varrho} \theta\|\frac{u}{\partial_y u^s}\|_{\mathcal{A}_{\ell}^k}, \hspace{1.58cm} \forall k\leq k_0,\\[10pt]

\|S_{\theta}^v v\|_{\mathcal{D}_0^{s}} \leq C_{\varrho}\theta^{(s-\alpha)_{+}}\|v\|_{\mathcal{D}_{0}^{\alpha}},\hspace{2.56cm}
\forall s\geq 0, \\[10pt]

\|(1 - S_{\theta}^v) v\|_{\mathcal{D}_0^{s}} \leq C_{\varrho}\theta^{s-\alpha}\|v\|_{\mathcal{D}_{0}^{\alpha}},\hspace{2.09cm}
\forall 0\leq s\leq\alpha.
\end{array}
\end{equation*}
\end{proposition}

\begin{proof}
Based on the properties of the convolution, it is easy to prove $(\ref{Sect4_Nonlinear_Op_Properties_3})$. But note that
\begin{equation}\label{Sect4_Nonlinear_Op_Properties_Proof1}
\begin{array}{ll}
[\frac{1}{\partial_y u^s},S_{\theta}^u]u(t,x,y) = \iiint\limits_{\mathbb{R}^3} \varrho_{\theta}(\tau)\varrho_{\theta}(\xi)\varrho_{\theta}(\eta)
[I_5(t-\tau+\theta^{-1},x-\xi,y-\eta)-1]\,\cdot \\[13pt]\hspace{3.2cm}
E^u(\frac{u}{\partial_y u^s}) (t-\tau+\theta^{-1},x-\xi,y-\eta)\,\mathrm{d}\tau\mathrm{d}\xi\mathrm{d}\eta,
\end{array}
\end{equation}
where \vspace{-0.5cm}
\begin{equation}\label{Sect4_Nonlinear_Op_Properties_Proof2}
\begin{array}{ll}
I_5(t,x,y-\eta) :=\cfrac{\frac{E^u u(t,x,y-\eta)}{\partial_y u^s(t,y)}}
{E^u\, \frac{u}{\partial_y u^s}(t,x,y-\eta)}

= \left\{\begin{array}{ll}
0, \hspace{1.3cm} \eta\geq y+ \theta^{-1}, \\[7pt]
1, \hspace{1.3cm} \eta\leq y, \\[8pt]
\frac{u_y^s(t,y-\eta)}{u_y^s(t,y)},
\hspace{0.15cm} y<\eta< y+ \theta^{-1}.
\end{array}\right.
\end{array}
\end{equation}

\vspace{-0.2cm}
Then $\sum\limits_{k_1+[\frac{k_2+1}{2}]\leq k_0} |\partial_{t,x}^{k_1}\partial_y^{k_2} I_5|_{\infty} \leq \max\big\{1, C\big(\|\frac{u_{yy}^s}{u_y^s}\|_{\mathcal{C}_0^{k_0+2}}\big)\big\}<+\infty$ by using $\frac{\partial_y u^s(t,y_2)}{\partial_y u^s(t,y_1)} 
=\exp\big\{\int\limits_{y_1}^{y_2} \frac{u_{yy}^s}{u_y^s}(t,\tilde{y}) \,\mathrm{d}\tilde{y} \big\}$, thus $(\ref{Sect4_Nonlinear_Op_Properties_3})_4$ and $(\ref{Sect4_Nonlinear_Op_Properties_3})_5$ can be proved.
\end{proof}

Similar to \cite{Alexandre_Yang_2014}, we introduce the nonlinear operator which corresponds to the Prandtl system $(\ref{Sect1_PrandtlEq})$:
\begin{equation}\label{Sect4_Nonlinear_Op}
\begin{array}{ll}
\mathcal{P}(u,v) = u_t + u u_x + v u_y - u_{yy},
\end{array}
\end{equation}
and its linearized operator around the background state $(\tilde{u},\tilde{v})$:
\begin{equation}\label{Sect4_Linearized_Op}
\begin{array}{ll}
\mathcal{P}_{(\tilde{u},\tilde{v})}^{\prime}(u,v) = u_t + \tilde{u} u_x + \tilde{v} u_y + u \tilde{u}_x + v \tilde{u}_y - u_{yy}.
\end{array}
\end{equation}

Assume that for $k=0,\ldots,n$, we have constructed the approximate solutions $(u^k,v^k)$ of $(\ref{Sect1_PrandtlEq})$, then we need to construct the $(n+1)-th$ approximate solution $(u^{n+1},v^{n+1})$ of $(\ref{Sect1_PrandtlEq})$. Set
\begin{equation}\label{Sect4_Solutions_Decomp}
\left\{\begin{array}{ll}
u^{n+1} = u^n + \delta u^n = u^s + \tilde{u}^n + \delta u^n,\\[7pt]
v^{n+1} = v^n + \delta v^n,
\end{array}\right.
\end{equation}
where $(u^s,0)$ is the shear flow, the increments $(\delta u^n,\delta v^n)$ are the solutions to the following IBVP:
\begin{equation}\label{Sect4_Increment_Eq}
\left\{\begin{array}{ll}
\mathcal{P}_{(u_{\theta_n}^n, v_{\theta_n}^n)}^{\prime}
(\delta u^n,\delta v^n) = f^n, \\[10pt]
(\delta u^n)_x + (\delta v^n)_y =0, \\[10pt]
(\partial_y \delta u^n -\beta\delta u^n)|_{y=0} =0,\quad \delta v^n |_{y=0} =0,\\[10pt]
\lim\limits_{y\rto +\infty}\delta u^n =0, \\[10pt]
\delta u^n|_{t\leq 0} =0,
\end{array}\right.
\end{equation}
where $u_{\theta_n}^n = u^s + S_{\theta_n}^u\tilde{u}^n,\ v_{\theta_n}^n =S_{\theta_n}^v v^n$.

Let
\begin{equation}\label{Sect4_Increment_Eq_1}
\begin{array}{ll}
\mathcal{P}(u^{n+1},v^{n+1}) - \mathcal{P}(u^n,v^n)
= \mathcal{P}_{(u_{\theta_n}^n, v_{\theta_n}^n)}^{\prime} (\delta u^n, \delta v^n) + e_n,
\end{array}
\end{equation}
then the error $e_n$ can be decomposed:
$$e_n = e_n^{(1)} + e_n^{(2)},$$
where $e_n^{(1)}$ is the error comes from the iteration scheme, $e_n^{(2)}$ is the error comes from mollifying the coefficients.
Thus, $e_n^{(1)}, e_n^{(2)}$ are defined as
\begin{equation}\label{Sect4_Increment_en1}
\begin{array}{ll}
e_n^{(1)} = \mathcal{P}(u^{n+1},v^{n+1}) - \mathcal{P}(u^n,v^n)
- \mathcal{P}_{(u^n, v^n)}^{\prime} (\delta u^n, \delta v^n)
\\[7pt]\hspace{0.62cm}
= \mathcal{P}(u^{n} + \delta u^{n},v^{n}+ \delta v^{n}) - \mathcal{P}(u^n,v^n)
- \mathcal{P}_{(u^n, v^n)}^{\prime} (\delta u^n, \delta v^n)
\\[7pt]\hspace{0.6cm}

= \delta u^{n}(\delta u^n)_x + \delta v^{n}(\delta u^n)_y \ ,
\end{array}
\end{equation}
\noindent
and
\begin{equation}\label{Sect4_Increment_en2}
\begin{array}{ll}
e_n^{(2)} = \mathcal{P}_{(u^n, v^n)}^{\prime} (\delta u^n, \delta v^n)
-\mathcal{P}_{(u_{\theta_n}^n, v_{\theta_n}^n)}^{\prime} (\delta u^n, \delta v^n)
\\[9pt]\hspace{0.65cm}

= [(u^n - u^n_{\theta_n})(\delta u^n)]_x + (v^n- v^n_{\theta_n})(\delta u^n)_y
+ \delta v^{n} (u^{n}-u_{\theta_n}^{n})_y

\\[9pt]\hspace{0.66cm}
= [(1- S_{\theta_n}^u)\tilde{u}^n \delta u^n]_x + (1- S_{\theta_n}^v)v^n(\delta u^n)_y
+ \delta v^{n} [(1- S_{\theta_n}^u) \tilde{u}^{n}]_y \ .
\end{array}
\end{equation}

Sum $(\ref{Sect4_Increment_Eq_1})$ from $0$ to $n$, we get
\begin{equation}\label{Sect4_Increment_Eq_2}
\begin{array}{ll}
\mathcal{P}(u^{n+1},v^{n+1}) = \sum\limits_{j=0}^n
[\mathcal{P}_{(u_{\theta_n}^j, v_{\theta_n}^j)}^{\prime} (\delta u^j,\delta v^j)
+ e_j] + f^a,
\end{array}
\end{equation}
where
\begin{equation}\label{Sect4_Fa_Define}
\begin{array}{ll}
f^a := \mathcal{P}(u^0,v^0)= u_t^0 + u^0 u_x^0 + v^0 u_y^0 - u_{yy}^0 \ ,
\end{array}
\end{equation}
and $u^0,v^0$ will be determined in Subsection 4.2.

Note that $f^a, e_n^{(1)}, e_n^{(2)}$ have no restrictions on the boundary, so we use $S_{\theta_n}$ to mollify them.
Similar to $\cite{Alexandre_Yang_2014}$, $f^n$ can be defined by induction on $n$, i.e.,
\begin{equation}\label{Sect4_Increment_Eq_3}
\begin{array}{ll}
\sum\limits_{j=0}^n f^j = -S_{\theta_n}(\sum\limits_{j=0}^{n-1}e_j) - S_{\theta_n}f^a,
\end{array}
\end{equation}
then we have the formulae as follows:
\begin{equation}\label{Sect4_Increment_Eq_4}
\left\{\begin{array}{ll}
f^0 = - S_{\theta_0}f^a,\\[9pt]
f^1 = (S_{\theta_0} - S_{\theta_1})f^a - S_{\theta_1} e_0, \\[9pt]
f^n = (S_{\theta_{n-1}} - S_{\theta_n})(\sum\limits_{j=0}^{n-2}e_j) - S_{\theta_n} e_{n-1}
+ (S_{\theta_{n-1}} - S_{\theta_n})f^a,\quad \forall n\geq 2.
\end{array}\right.
\end{equation}

Finally, we have
\begin{equation}\label{Sect4_Increment_Eq_5}
\begin{array}{ll}
\mathcal{P}(u^{n+1},v^{n+1}) = \sum\limits_{j=0}^n f^j + \sum\limits_{j=0}^n e_j + f^a
\\[11pt]\hspace{2.3cm}
= e_n + (1 -S_{\theta_n})(\sum\limits_{j=0}^{n-1}e_j) + (1 - S_{\theta_n})f^a.
\end{array}
\end{equation}

In order to prove the convergence of the iteration,
\begin{equation*}
\begin{array}{ll}
\mathcal{P}(u^{n+1},v^{n+1}) \rto 0,\ \ pointwisely,\ \ as\ n\rto +\infty,
\end{array}
\end{equation*}
we need to prove in Section 5 that for some values of $k$,
\begin{equation*}
\begin{array}{ll}
\sum\limits_{j=0}^{+\infty}\|e_j\|_{\mathcal{A}_{\ell}^{k}}<+\infty.
\end{array}
\end{equation*}

\subsection{Construction of the Zero-th Order Approximate Solutions Satisfying Robin Boundary Condition}
In this subsection, we construct the zero-th order approximate solution by using the initial data $\tilde{u}_{0}(x,y) = u_{0}(x,y)-u_0^s(y)$.
Denote
\begin{equation}\label{Sect4_Zeroth_Order_Sol_1}
\left\{\begin{array}{ll}
\tilde{u}^0(t,x,y) := \sum\limits_{j=0}^{k_0}\frac{t^j}{j!} \partial_t^j \tilde{u}_0(x,y), \\[12pt]
v^0(t,x,y) := \sum\limits_{j=0}^{k_0}\frac{t^j}{j!} \partial_t^j v_0(x,y), \\[14pt]
u^0(t,x,y) = u^s(t,y) + \tilde{u}^0(t,x,y).
\end{array}\right.
\end{equation}
where $\partial_t^j \tilde{u}_0(x,y)$ and $\partial_t^j v_0(x,y)$ are defined by induction:
\begin{equation}\label{Sect4_Zeroth_Order_Sol_2}
\left\{\begin{array}{ll}
\partial_t^j \tilde{u}_0 = \partial_{yy} (\partial_t^{j-1} \tilde{u}_0)
- \sum\limits_{m=0}^{j-1} C_{j-1}^{m} (\partial_t^{m} u^s_0 + \partial_t^{m} \tilde{u}_0)\partial_x(\partial_t^{j-1-m} \tilde{u}_0)
\\[8pt]\hspace{1.2cm}
- \sum\limits_{m=0}^{j-1} C_{j-1}^{m} (\partial_t^{m} v_0)\partial_y(\partial_t^{j-1-m} u^s_0 + \partial_t^{j-1-m} \tilde{u}_0),\hspace{0.4cm} 1\leq j\leq 2k_0+1, \\[12pt]

\partial_t^j v_0(x,y) = -\int\limits_0^y \partial_x \partial_t^j \tilde{u}_0(x,\tilde{y}) \,\mathrm{d}\tilde{y}, \hspace{3.66cm} 0\leq j\leq 2k_0.
\end{array}\right.
\end{equation}

Then we have the following proposition relating to $\tilde{u}^0,\ v^0$:
\begin{proposition}\label{Sect4_Data_Proposition}
If $\tilde{u}_0$ satisfies $\partial_y^{2j}(\partial_y \tilde{u}_0(x,y) - \beta \tilde{u}_0(x,y))|_{y=0}=0,\ 0\leq j\leq 2k_0$
and $u_0^s$ satisfies $\partial_y^{2j}(\partial_y u_0^s(y) - \beta u_0^s(y))|_{y=0}=0,\ 0\leq j\leq 2k_0$,
$\lim\limits_{y\rto +\infty} \tilde{u}_0(x,y) = 0$, $\tilde{u}_0\in \mathcal{A}_{\ell}^{2k_0+1}(\mathbb{R}_{+}^2, t=0)$,
then $(\tilde{u}^0,v^0)$ satisfies
\begin{equation}\label{Sect4_Zeroth_Order_Sol_3}
\left\{\begin{array}{ll}
\tilde{u}_t^0 + (u^s+ \tilde{u}^0) \tilde{u}_x^0 + v^0 (u^s+ \tilde{u}^0)_y - \tilde{u}_{yy}^0 = f^a, \quad (x,y)\in\mathbb{R}_{+}^2,\ t>0, \\[9pt]
\tilde{u}_x^0 + v_y^0 =0, \quad u_x^0 + v_y^0 =0, \\[9pt]
\partial_{t}^{j_1}\partial_{x}^{j_2} v^0|_{y=0} =0, \hspace{2.2cm} 0\leq j_1+j_2\leq k_0, \\[11pt]
\partial_{t}^{j_1}\partial_{x}^{j_2}(\partial_y \tilde{u}^0 - \beta \tilde{u}^0)|_{y=0} = 0,\hspace{0.55cm} 0\leq j_1+j_2\leq k_0, \\[11pt]
\lim\limits_{y\rto +\infty} \tilde{u}^0(t,x,y) = 0, \\[11pt]
\tilde{u}^0|_{t=0} = \tilde{u}_0(x,y),\quad v^0|_{t=0} = v_0(x,y).
\end{array}\right.
\end{equation}

If $u_0^s$ satisfies $\|u_0^s-1\|_{L^2}+ |u_0^s|_{\infty} + \|u_0^s\|_{\dot{\mathcal{C}}_{\ell}^{2k_0+3}} + \|\frac{\partial_{yy}u_0^s}{\partial_y u_0^s}\|_{\mathcal{C}_{\ell}^{2k_0+2}}\leq C$ and
$\tilde{u}_0 = u_0 -u_0^s$ satisfies
\begin{equation}\label{Sect4_Data_Regularity_1}
\begin{array}{ll}
\|\tilde{u}_0\|_{\mathcal{A}_{\ell}^{2k_0+1}(\mathbb{R}_{+}^2,t=0)}
+ \|\frac{\partial_y \tilde{u}_0}{\partial_y u_0^s}\|_{\mathcal{A}_{\ell}^{2k_0+1}(\mathbb{R}_{+}^2,t=0)}
 \leq \e,
\end{array}
\end{equation}
for a small constant $\e>0$, then
\begin{equation}\label{Sect4_Data_Regularity_2}
\begin{array}{ll}
\|\tilde{u}^0\|_{\mathcal{A}_{\ell}^{k_0+1}([0,T]\times\mathbb{R}_{+}^2)}
+\|\frac{\partial_y \tilde{u}^0}{\partial_y u^s}\|_{\mathcal{A}_{\ell}^{k_0+1}([0,T]\times\mathbb{R}_{+}^2)}
+\|v^0\|_{\mathcal{D}_{0}^{k_0}([0,T]\times\mathbb{R}_{+}^2)} \\[12pt]\quad
+ \|f^a\|_{\mathcal{A}_{\ell}^{k_0}([0,T]\times\mathbb{R}_{+}^2)}
+ \|\frac{f^a}{\partial_y u^s}\|_{\mathcal{A}_{\ell}^{k_0}([0,T]\times\mathbb{R}_{+}^2)} \\[9pt]\quad
+ \sum\limits_{m=0}^{2k_0+2}\|\partial_y^{m}\tilde{u}^0|_{y=0}\|_{A^{k_0+1-[\frac{m+1}{2}]}} \leq C(T)\e.
\end{array}
\end{equation}
\end{proposition}

\begin{proof}
As to $(\ref{Sect4_Zeroth_Order_Sol_3})$, we only need to prove $\partial_{t}^{j_1}\partial_{x}^{j_2}(\partial_y \tilde{u}^0 - \beta \tilde{u}^0)|_{y=0} = 0,\ 0\leq j_1+j_2\leq k_0$.

By the definition of $\partial_t \tilde{u}_0$, namely $(\ref{Sect4_Zeroth_Order_Sol_2})_1$ with $j=1$,
it is trivial to verify that $(\partial_y(\partial_t \tilde{u}_0) - \beta\partial_t \tilde{u}_0)|_{y=0} =\partial_{yy}(\partial_y \tilde{u}_0 -\beta \tilde{u}_0)|_{y=0} =0.$

Assume that $(\partial_y(\partial_t^i \tilde{u}_0) - \beta\partial_t^i \tilde{u}_0)|_{y=0} =\partial_y^{2i}(\partial_y \tilde{u}_0 -\beta \tilde{u}_0)|_{y=0}=0$ for $0\leq i\leq j-1,\ 0\leq j\leq 2k_0$. Note that $\partial_t^i v_0|_{y=0}=0,\ \partial_x(\partial_{y}\partial_t^i \tilde{u}_0 - \beta \partial_t^i \tilde{u}_0)|_{y=0} =0,\ 0\leq i\leq j-1$, then we have
\begin{equation}\label{Sect4_Zeroth_Order_Sol_5}
\begin{array}{ll}
(\partial_y(\partial_t^j \tilde{u}_0) - \beta \partial_t^j \tilde{u}_0)|_{y=0}
= \partial_{y}\partial_{yy} \partial_t^{j-1} \tilde{u}_0|_{y=0} - \beta \partial_{yy}\partial_t^{j-1} \tilde{u}_0|_{y=0} \\[5pt]\quad

- \sum\limits_{m=0}^{j-1} C_{j-1}^{m} (\partial_t^{m} u^s_0 + \partial_t^{m} \tilde{u}_0)\Big((\partial_t^{j-1-m} \tilde{u}_0)_{yx}
- \beta(\partial_t^{j-1-m} \tilde{u}_0)_x \Big)\Big|_{y=0} \\[8pt]\quad

- \sum\limits_{m=0}^{j-1} C_{j-1}^{m} (\partial_t^{m} v_0)\Big((\partial_t^{j-1-m} u^s_0 + \partial_t^{j-1-m} \tilde{u}_0)_{yy} \\[8pt]\quad

- \beta(\partial_t^{j-1-m} u^s_0 + \partial_t^{j-1-m} \tilde{u}_0)_y \Big)\Big|_{y=0} \\[8pt]\quad

- \Big(\sum\limits_{m=0}^{j-1} C_{j-1}^{m} (\partial_t^{m} u^s_0 + \partial_t^{m} \tilde{u}_0)_y(\partial_t^{j-1-m} \tilde{u}_0)_x  \\[8pt]\quad
+ \sum\limits_{m=0}^{j-1} C_{j-1}^{m} (\partial_t^{m} v_0)_y(\partial_t^{j-1-m} u^s_0 + \partial_t^{j-1-m} \tilde{u}_0)_y \Big)\Big|_{y=0}
\hspace{3.04cm} \\[10pt]
= \partial_y^{2j} (\partial_{y} \tilde{u}_0 -\beta \tilde{u}_0)|_{y=0} =0,
\end{array}
\end{equation}
where $C_{j-1}^m =C_{j-1}^{j-1-m} = \frac{(j-1)!}{m!(j-1-m)!}$.

Thus, $(\partial_y \tilde{u}^0 - \beta \tilde{u}^0)|_{y=0} = \sum\limits_{j=0}^{k_0}\frac{t^j}{j!}
\Big(\partial_y(\partial_t^j \tilde{u}_0) - \beta \partial_t^j \tilde{u}_0 \Big)\Big|_{y=0} =0$, then for $0\leq j_1 + j_2\leq k_0$,
\begin{equation}\label{Sect4_Zeroth_Order_Sol_6}
\begin{array}{ll}
\partial_{t}^{j_1}\partial_{x}^{j_2} (\partial_y \tilde{u}^0 - \beta \tilde{u}^0)|_{y=0}  \\[4pt]
= \sum\limits_{j=0}^{k_0}\sum\limits_{m=0}^{j_1} C_{j_1}^m \frac{\partial_t^{j_1-m} t^j}{j!} \partial_{x}^{j_2}\Big(\partial_y(\partial_t^{j+m} \tilde{u}_0) - \beta \partial_t^{j+m} \tilde{u}_0 \Big)\Big|_{y=0} =0.
\end{array}
\end{equation}

Next, we prove the estimate $(\ref{Sect4_Data_Regularity_2})$ under the assumption $(\ref{Sect4_Data_Regularity_1})$.

Consider the definition $(\ref{Sect4_Zeroth_Order_Sol_2})_1$, it is easy to prove that
\begin{equation}\label{Sect4_Zeroth_Order_Sol_7}
\begin{array}{ll}
\|\partial_t^j \tilde{u}_0\|_{\mathcal{A}_{\ell}^{k}(\mathbb{R}_{+}^2,t=0)} \lem \|u_0^s\|_{\mathcal{C}_0^{2k_0+2}}
\sum\limits_{m=0}^{j-1}\|\partial_t^m \tilde{u}_0\|_{\mathcal{A}_{\ell}^{k+1}(\mathbb{R}_{+}^2,t=0)} \\[11pt]\hspace{2.7cm}
\lem \|u^s_0\|_{\mathcal{C}_0^{2k_0+2}}\|\tilde{u}_0\|_{\mathcal{A}_{\ell}^{k+j}(\mathbb{R}_{+}^2,t=0)},
\end{array}
\end{equation}
where $0\leq j\leq k_0,\ 0\leq k\leq k_0+1$. Then
\begin{equation*}
\begin{array}{ll}
\|\tilde{u}^0\|_{\mathcal{A}_{\ell}^{k_0+1}([0,T]\times\mathbb{R}_{+}^2)}
\lem C(T)\sum\limits_{m=0}^{k_0}\|\partial_t^m \tilde{u}_0\|_{\mathcal{A}_{\ell}^{k_0+1}(\mathbb{R}_{+}^2,t=0)} \\[12pt]\hspace{3.04cm}
\lem C(T)\|u^s_0\|_{\mathcal{C}_0^{2k_0+2}}\|\tilde{u}_0\|_{\mathcal{A}_{\ell}^{2k_0+1}(\mathbb{R}_{+}^2,t=0)} \leq C(T)\e,
\end{array}
\end{equation*}
and $\|v^0\|_{\mathcal{D}_0^{k_0}} + \|f^a\|_{\mathcal{A}_{\ell}^{k_0}}
\lem \|\tilde{u}^0\|_{\mathcal{A}_{\ell}^{k_0+1}}
\lem C(T)\|\tilde{u}_0\|_{\mathcal{A}_{\ell}^{2k_0+1}(\mathbb{R}_{+}^2,t=0)}\leq C(T)\e$.

Similarly, $\|\frac{\partial_y \tilde{u}^0}{\partial_y u^s}\|_{\mathcal{A}_{\ell}^{k_0+1}}
\lem C(T)\|\frac{\partial_{ty} u^s}{\partial_y u^s}\|_{\mathcal{C}^{2k_0+1}}
\|\frac{\partial_y \tilde{u}_0}{\partial_y u^s}\|_{\mathcal{A}_{\ell}^{2k_0+1}(\mathbb{R}_{+}^2,t=0)}\leq C(T)\e$.
\begin{equation*}
\begin{array}{ll}
\Big\|\frac{\tilde{u}^0}{\partial_y u^s}\Big\|_{\mathcal{A}_{\ell}^{k}}
\leq \Big\|\frac{1}{\partial_y u^s} \int\limits_{\infty}^y \partial_y\tilde{u}^0(\tilde{y}) \,\mathrm{d}\tilde{y}\Big\|_{\mathcal{A}_{\ell}^{k}}
\leq \Big\|\int\limits_{\infty}^y \frac{\partial_y u^s(\tilde{y})}{\partial_y u^s(y)}\frac{\partial_y\tilde{u}^0(\tilde{y})}{\partial_y u^s(\tilde{y})} \,\mathrm{d}\tilde{y}\Big\|_{\mathcal{A}_{\ell}^{k}} \\[10pt]\hspace{1.55cm}
\leq C\Big(\big\|\frac{\partial_{yy}u^s}{\partial_y u^s}\big\|_{\mathcal{C}_0^{k+2}}\Big)
C_{\ell}\Big\|\frac{\partial_y\tilde{u}^0}{\partial_y u^s}\Big\|_{\mathcal{A}_{\ell}^{k}},
\end{array}
\end{equation*}
then $\|\frac{\tilde{u}^0}{\partial_y u^s}\|_{\mathcal{A}_{\ell}^{k_0+1}}\lem \|\frac{\partial_y\tilde{u}^0}{\partial_y u^s}\|_{\mathcal{A}_{\ell}^{k_0+1}}$ and then
\begin{equation*}
\begin{array}{ll}
\big\|\frac{f^a}{\partial_y u^s}\big\|_{\mathcal{A}_{\ell}^{k_0}} \lem \big(1+\big\|\frac{\partial_{ty} u^s}{\partial_y u^s}\big\|_{\mathcal{C}^{k_0+1}}\big)
\big\|\frac{\tilde{u}^0}{\partial_y u^s}\big\|_{\mathcal{A}_{\ell}^{k_0+1}}
+ \big\|\frac{\partial_{yy}u^s}{\partial_y u^s}\big\|_{\mathcal{C}_0^{k_0+1}}\big\|\frac{\partial_y\tilde{u}^0}{\partial_y u^s}\big\|_{\mathcal{A}_{\ell}^{k_0+1}}
\\[9pt]\hspace{1.7cm}
\lem \big\|\frac{\partial_y \tilde{u}_0}{\partial_y u^s}\big\|_{\mathcal{A}_{\ell}^{2k_0+1}(\mathbb{R}_{+}^2,t=0)} \leq C(T)\e.
\end{array}
\end{equation*}

Especially, we have the estimate on the boundary $\{y=0\}$:
\begin{equation}\label{Sect4_Zeroth_Order_Boundary}
\begin{array}{ll}
\sum\limits_{m=0}^{2k_0+2}\|\partial_y^m\tilde{u}^0|_{y=0}\|_{A^{k_0+1-[\frac{m+1}{2}]}([0,T]\times\mathbb{R})}  \\[12pt]
\lem \|\partial_y \tilde{u}^0\|_{\mathcal{A}_{\ell}^{k_0+1}}
\lem \|\partial_y u^s\|_{\mathcal{C}_0^{k_0+2}}\|\frac{\partial_y \tilde{u}^0}{\partial_y u^s}\|_{\mathcal{A}_{\ell}^{k_0+1}} \leq C(T)\e.
\end{array}
\end{equation}
\end{proof}

\vspace{-0.4cm}
\begin{remark}\label{Sect4_Zeroth_Order_Sol_Remark}
As $\beta\rto +\infty$, $\partial_t^j u^s|_{y=0}=O(\frac{1}{\beta}), 0\leq k\leq k_0, \|\tilde{u}^0|_{y=0}\|_{A^{k_0}([0,T]\times\mathbb{R})} =O(\frac{1}{\beta})$, since we have the estimates:
\begin{equation*}
\begin{array}{ll}
\|\tilde{u}^0|_{y=0}\|_{A^{k_0}([0,T]\times\mathbb{R})}\leq C\beta^{-1}, \\[10pt]
\|\partial_y \tilde{u}^0|_{y=0}\|_{A^{k_0}([0,T]\times\mathbb{R})}\leq C.
\end{array}
\end{equation*}
\end{remark}

\vspace{0.1cm}
\section{Existence of Classical Solutions to the Nonlinear Prandtl Equations with Robin Boundary Condition}
In this section, we prove the estimates for some variables and quantities arising in the process of the Nash-Moser-H$\ddot{o}$rmander iteration, based on which we prove the convergence of the iteration which implies the existence of classical solutions to the Prandtl system $(\ref{Sect1_PrandtlEq})$.

Define
\begin{equation}\label{Sect5_Increment_Eq_2}
\left\{\begin{array}{ll}
u_{\theta_n}^n = u^s + S_{\theta_n}^u(\tilde{u}^0 + \sum\limits_{0\leq j\leq n-1}\delta u^j), \\[14pt]
v_{\theta_n}^n =S_{\theta_n}^v(v^0 + \sum\limits_{0\leq j\leq n-1}\delta v^j),
\end{array}\right.
\end{equation}
then $u_{\theta_n}^n$ and $v_{\theta_n}^n$ satisfy
\begin{equation}\label{Sect5_Background_Condition}
\left\{\begin{array}{ll}
\partial_x u_{\theta_n}^n + \partial_y v_{\theta_n}^n =0, \\[12pt]
v_{\theta_n}^n|_{y=0}=0, \\[12pt]
\lim\limits_{y\rto +\infty} u_{\theta_n}^n(t,x,y)=1.
\end{array}\right.
\end{equation}
and it will be verified in Subsection 5.2 that
\begin{equation}\label{Sect5_Iteration_Monotonicity}
\begin{array}{ll}
u_{\theta_n}^n>0,\quad \partial_y u_{\theta_n}^n>0,\quad \beta -\frac{\partial_{yy}u_{\theta_n}^n}{\partial_y u_{\theta_n}^n}\geq \delta>0,
\quad y\geq 0,\ t\in[0,T].
\end{array}
\end{equation}

Note that the condition $(\ref{Sect5_Background_Condition})$ does not contain $(\partial_y u_{\theta_n}^n -\beta u_{\theta_n}^n)|_{y=0} =0$, which is unnecessary because $(\ref{Sect3_Background_Condition_1})$ does not need the condition $(\partial_y\tilde{u} -\beta\tilde{u})|_{y=0} =0$.
As pointed out in Remark \ref{Sect4_Operator_SU_Remark}, $[S_{\theta_n}^u(\partial_y \tilde{u}^0 -\beta \tilde{u}^0)]|_{y=0} =0$ and $[S_{\theta_n}^u(\partial_y \delta u^j -\beta \delta u^j)]|_{y=0} =0$ are unnecessary to be satisfied, $(S_{\theta_n}^u \tilde{u}^0)|_{y=0}$ and $(S_{\theta_n}^u \delta u^j)|_{y=0}$ do not equal zero in general.

\vspace{0.5cm}
Moreover, $f^n$ is defined in $(\ref{Sect4_Increment_Eq_4})$, then the problem $(\ref{Sect4_Increment_Eq})$ is equivalent to the following system:
\begin{equation}\label{Sect5_Increment_Eq_1}
\left\{\begin{array}{ll}
(\delta u^n)_t + u_{\theta_n}^n (\delta u^n)_x + v_{\theta_n}^n (\delta u^n)_y
+ \delta u^n (u_{\theta_n}^n)_x + \delta v^n (u_{\theta_n}^n)_y - (\delta u^n)_{yy} = f^n, \\[13pt]
(\delta u^n)_x + (\delta v^n)_y =0, \\[13pt]
(\partial_y\delta u^n -\beta\delta u^n)|_{y=0} =0,\quad \delta v^n |_{y=0} =0,\\[13pt]
\lim\limits_{y\rto +\infty}\delta u^n =0, \\[13pt]
\delta u^n|_{t\leq 0} =0.
\end{array}\right.
\end{equation}

\noindent
Set
\begin{equation}\label{Sect5_New_Variables}
\begin{array}{ll}
w^n = \partial_y(\frac{\delta u^n}{\partial_y u_{\theta_n}^n}),\hspace{0.6cm}
\eta^n = \frac{\partial_{yy} u_{\theta_n}^n}{\partial_y u_{\theta_n}^n},\hspace{0.6cm}
\bar{\eta}^n = \frac{u_{yy}^s}{\partial_y u_{\theta_n}^n},\hspace{0.6cm}
\tilde{f}^n = \frac{f^n}{\partial_y u_{\theta_n}^n}, \\[17pt]

\zeta^n = \frac{(\partial_t + u_{\theta_n}^n\partial_x + v_{\theta_n}^n\partial_y -\partial_{yy})\partial_y u_{\theta_n}^n}{\partial_y u_{\theta_n}^n},
\hspace{0.67cm}

\tilde{\zeta}_1^n = \frac{\partial_{yyt} u^s}{\partial_y u_{\theta_n}^n}
- \bar{\eta}^n\frac{\partial_{yt} u^s}{\partial_y u_{\theta_n}^n}, \\[17pt]

\tilde{\zeta}_2^n = \frac{\partial_{yyt}(u_{\theta_n}^n -u^s)}{\partial_y u_{\theta_n}^n}
+ \frac{u_{\theta_n}^n\partial_{yyx}u_{\theta_n}^n}{\partial_y u_{\theta_n}^n}
- \eta^n\frac{\partial_{yt} u_{\theta_n}^n}{\partial_y u_{\theta_n}^n}
+ \bar{\eta}^n\frac{\partial_{yt} u^s}{\partial_y u_{\theta_n}^n}
- \eta^n \frac{u_{\theta_n}^n \partial_{yx} u_{\theta_n}^n}{\partial_y u_{\theta_n}^n},  \\[19pt]

\tilde{\zeta}^n := \tilde{\zeta}_1^n+\tilde{\zeta}_2^n = \frac{
\partial_{yyt}u_{\theta_n}^n + u_{\theta_n}^n\partial_{yyx}u_{\theta_n}^n
-\eta^n(\partial_{yt} u_{\theta_n}^n + u_{\theta_n}^n \partial_{yx} u_{\theta_n}^n)}
{\partial_y u_{\theta_n}^n},
\end{array}
\end{equation}

\begin{equation*}
\begin{array}{ll}
\lambda_{0,0}^n = \|u_{\theta_n}^n-u^s\|_{\mathcal{B}_{0,0}^{0,0}} + \|\partial_x u_{\theta_n}^n\|_{\mathcal{B}_{0,0}^{0,0}}
+ \|v_{\theta_n}^n\|_{L_y^{\infty}(L_{t,x}^2)} + \|\bar{\eta}^n\|_{L_y^2(L_{t,x}^{\infty})} \\[12pt]\hspace{1.1cm}
+ \|\eta^n-\bar{\eta}^n\|_{\mathcal{B}_{0,\ell}^{0,0}}
+ \|\zeta^n\|_{\mathcal{B}_{0,\ell}^{0,0}}\ ,
\\[18pt]

\lambda_{k_1,k_2}^n = \|u_{\theta_n}^n-u^s\|_{\mathcal{B}_{0,0}^{k_1,k_2}} + \|\partial_x u_{\theta_n}^n\|_{\mathcal{B}_{0,0}^{k_1,k_2}}
+ \sum\limits_{0\leq m\leq k_1,0\leq q\leq k_2}\|\partial_{t,x}^{m}\partial_y^{q} u^s\|_{L_{y}^2(L_t^{\infty})} \\[16pt]\hspace{1.4cm}
+ \sum\limits_{0\leq m\leq k_1,0\leq q\leq k_2}\|\partial_{t,x}^{m}\partial_y^{q} v_{\theta_n}^n\|_{L_y^{\infty}(L_{t,x}^2)}
+ \|\eta^n-\bar{\eta}^n\|_{\mathcal{B}_{0,\ell}^{k_1,k_2}} \\[16pt]\hspace{1.4cm}
+ \sum\limits_{0\leq m\leq k_1,0\leq q\leq k_2}\|\partial_{t,x}^{m}\partial_y^{q} \bar{\eta}^n\|_{L_y^2(L_{t,x}^{\infty})}
+ \|\zeta^n\|_{\mathcal{B}_{0,\ell}^{k_1,k_2}}\ ,\quad
k_1+[\frac{k_2+1}{2}]>0,
\\[19pt]

\lambda_{k_1}^n|_{\partial\Omega} = \big\|u_{\theta_n}^n|_{y=0}-u^s|_{y=0}\big\|_{A^{k_1}} + \big\|\partial_x u_{\theta_n}^n|_{y=0}\big\|_{A^{k_1}}
 + \sum\limits_{m=0}^{k_1}\big|\partial_t^{m}u^s|_{y=0}\big|_{\infty} \\[12pt]\hspace{1.5cm}
+ \sum\limits_{m=0}^{k_1}\big\|\partial_{t,x}^{m}\tilde{\zeta}_1^n|_{y=0}\big\|_{L_{t,x}^{\infty}}
+ \big\|\tilde{\zeta}_2^n|_{y=0}\big\|_{A^{k_1}}
+ \sum\limits_{m=0}^{k_1}\big\|\partial_{t,x}^{m}\bar{\eta}^n|_{y=0}\big\|_{L_{t,x}^{\infty}} \\[15pt]\hspace{1.5cm}
+ \big\|\eta^n|_{y=0}-\bar{\eta}^n|_{y=0}\big\|_{A^{k_1}}, \quad k_1\geq 0,
\\[19pt]

\lambda_k^n = \sum\limits_{0\leq k_1 + [\frac{k_2+1}{2}]\leq k} \lambda_{k_1,k_2}^n
+ \sum\limits_{0\leq k_1\leq k}\lambda_{k_1}^n|_{\partial\Omega}\ .
\end{array}
\end{equation*}

Then $w^n$ satisfies the following IBVP:
\begin{equation}\label{Sect5_VorticityEq}
\left\{\begin{array}{ll}
\partial_t w^n + (u_{\theta_n}^n w^n)_x + (v_{\theta_n}^n w^n)_y - 2(\eta^n w^n)_y
- \big(\zeta^n \int\limits_y^{+\infty} w^n(t,x,\tilde{y}) \,\mathrm{d}\tilde{y}\big)_y \\[9pt]\hspace{0.5cm}
- \partial_{yy}w^n = \partial_y\tilde{f}^n, \\[8pt]

\frac{\partial_t w^n}{\beta- \eta^n|_{y=0}} + \frac{\partial_x(u_{\theta_n}^n w^n)}{\beta- \eta^n|_{y=0}}
- \Big(\partial_y w^n + 2\eta^n|_{y=0} w^n + \tilde{f}^n + \zeta^n|_{y=0}\int\limits_0^{+\infty} w^n(t,x,\tilde{y}) \,\mathrm{d}\tilde{y}\Big) \\[10pt]\hspace{0.5cm}

+ \frac{w^n\tilde{\zeta}^n|_{y=0}}{(\beta- \eta^n|_{y=0})^2} =0,\quad y=0, \\[12pt]

w^n|_{t\leq 0} =0.
\end{array}\right.
\end{equation}

\subsection{Estimates of the Variables and Quantities for the Iteration}
In this subsection, we prove the estimates stated in the following theorem, $C_n(T)$ is written as $C_n$ for brevity.
\begin{theorem}\label{Sect5_Main_Thm}
Assume the reference index $\kk\geq 7$, $k_0\geq \kk+2$, $\dt\in (0,1)$ is small, if $u_0^s$ satisfies
$ \|u_0^s-1\|_{L^2} + |u_0^s|_{\infty} + \|u_0^s\|_{\dot{\mathcal{C}}_{\ell}^{2k_0+3}} + \|\frac{\partial_{yy}u_0^s}{\partial_y u_0^s}\|_{\mathcal{C}_{\ell}^{2k_0+2}}\leq C $
and $\tilde{u}_0 = u_0-u^s_0$ satisfies
$\|\tilde{u}_0\|_{\mathcal{A}_{\ell}^{2k_0+1}(\mathbb{R}_{+}^2,t=0)}
+ \|\frac{\partial_y \tilde{u}_0}{\partial_y u_0^s}\|_{\mathcal{A}_{\ell}^{2k_0+1}(\mathbb{R}_{+}^2,t=0)}
 \leq \e,$
then the following variables have estimates as follows:
\begin{equation}\label{Sect5_Variables_Estimate}
\begin{array}{ll}
\|f^n\|_{\mathcal{A}_{\ell}^k} \leq C_n\e \theta_n^{\max\{k-\kk,3-\kk\}} \triangle\theta_n,\hspace{3.46cm} 0\leq k\leq k_0,\\[9pt]
\|w^n\|_{\mathcal{A}_{\ell}^k} + \frac{1}{\sqrt{\beta + C_{\eta}}}\big\|w^n|_{y=0}\big\|_{A^k} \leq C_n\e \theta_n^{\max\{k-\kk,3-\kk\}} \triangle\theta_n,
  0\leq k\leq k_0, \\[10pt]

\|\delta u^n\|_{\mathcal{A}_{\ell}^k} \leq C_n\e \theta_n^{\max\{k-\kk,3-\kk\}} \triangle\theta_n,\hspace{3.32cm} 0\leq k\leq k_0, \\[9pt]

\big\|\partial_y\delta u^n\big\|_{\mathcal{A}_{\ell}^k}
+\big\|\frac{\partial_y\delta u^n}{\partial_y u^s}\big\|_{\mathcal{A}_{\ell}^k} \leq C_n\e \theta_n^{\max\{k-\kk,3-\kk\}}\triangle\theta_n,
\hspace{0.84cm}0\leq k\leq k_0, \\[12pt]

\sum\limits_{m=0}^{2k}\big\|\partial_y^m\delta u^n|_{y=0}\big\|_{A^{k-[\frac{m+1}{2}]}}
\leq C_n\e \theta_n^{\max\{k-\kk,3-\kk\}} \triangle\theta_n,\hspace{0.59cm} 0\leq k\leq k_0, \\[12pt]

\|\frac{\delta u^n}{\partial_y u^s}\|_{\mathcal{D}_0^k} \leq C_n\e \theta_n^{\max\{k-\kk,3-\kk\}}\triangle\theta_n,\hspace{3.2cm} 0\leq k\leq k_0, \\[11pt]

\|\delta v^n\|_{\mathcal{D}_{0}^{k}} \leq C_n\e \theta_n^{\max\{k+1-\kk,3-\kk\}} \triangle\theta_n,\hspace{2.98cm} 0\leq k\leq k_0-1, \\[11pt]

\|e_n^{(1)}\|_{\mathcal{A}_{\ell}^{k}} \leq C_n\e^2\theta_n^{\max\{k+3-2\kk,5-2\kk\}} \triangle\theta_n,\hspace{2.54cm} 0\leq k\leq k_0-1, \\[11pt]
\|e_n^{(2)}\|_{\mathcal{A}_{\ell}^{k}} \leq C_n\e^2\theta_n^{\max\{k+5-2\kk,7-2\kk\}} \triangle\theta_n,\hspace{2.54cm} 0\leq k\leq k_0-1.
\end{array}
\end{equation}

Moreover, the following quantities have estimates as follows:
\begin{equation}\label{Sect5_Quantities_Estimate}
\begin{array}{ll}
\|\frac{\partial_y u_{\theta_n}^{n}}{\partial_y u^s}\|_{L^{\infty}}
+ \|(\frac{\partial_y u_{\theta_n}^{n}}{\partial_y u^s})^{-1}\|_{L^{\infty}} \leq C, \\[9pt]

\|u_{\theta_n}^n -u^s\|_{\mathcal{A}_{\ell}^k}\leq C_n\e \theta_n^{\max\{k+1-\kk+\dt,0\}}, \hspace{3.55cm}0\leq k\leq 2k_0+2, \\[9pt]
\big\|u_{\theta_n}^n|_{y=0} -u^s|_{y=0}\big\|_{A^k}\leq C_n\e \theta_n^{\max\{k+1-\kk+\dt,0\}}, \hspace{2.28cm}0\leq k\leq 2k_0+2, \\[9pt]

\|v_{\theta_n}^n\|_{\mathcal{D}_0^{k}} \leq C_n\e \theta_n^{\max\{k+2-\kk+\dt,0\}},\hspace{4.35cm} 0\leq k\leq 2k_0+1, \\[9pt]

\|\frac{\partial_y (u_{\theta_n}^n -u^s)}{\partial_y u^s}\|_{\mathcal{A}_{\ell}^k}
 \leq C_n\e \theta_n^{\max\{k+1-\kk+\dt,0\}},\hspace{3.26cm} 0\leq k\leq 2k_0,\\[9pt]
\|\partial_y (u_{\theta_n}^n -u^s)\|_{\mathcal{A}_{\ell}^k}
 \leq C_n\e \theta_n^{\max\{k+1-\kk+\dt,0\}},\hspace{2.94cm} 0\leq k\leq 2k_0,\\[9pt]
 
\big\|\frac{\partial_y (u_{\theta_n}^n -u^s)}{\partial_y u^s}|_{y=0}\big\|_{A^k} \leq C_n\e \theta_n^{\max\{k+1-\kk+\dt,0\}},\hspace{2.6cm}
 0\leq k\leq 2k_0,\\[9pt]
\big\|\partial_y (u_{\theta_n}^n -u^s)|_{y=0}\big\|_{A^k} \leq C_n\e \theta_n^{\max\{k+1-\kk+\dt,0\}},\hspace{2.28cm}
 0\leq k\leq 2k_0,\\[9pt]

\|(\frac{\partial_y u_{\theta_n}^{n}}{\partial_y u^s})^{-1}\|_{\dot{\mathcal{A}}_{\ell}^k}
+ \big\|(\frac{\partial_y u_{\theta_n}^{n}}{\partial_y u^s})^{-1}|_{y=0}\big\|_{\dot{A}^k} \leq C_n\e \theta_n^{\max\{k+1-\kk+\dt,0\}},
1\leq k\leq 2k_0, \\[9pt]

\|\eta^n- \bar{\eta}^n\|_{\mathcal{A}_{\ell}^k} + \|(\eta^n- \bar{\eta}^n)|_{y=0}\|_{A^k}
\leq C_n\e\theta_n^{\max\{k+2-\kk+\dt,0\}},\hspace{0.41cm}  0\leq k\leq 2k_0-1, \\[9pt]

\|\bar{\eta}^n\|_{\mathcal{C}_{\ell}^k} \leq C_n + C_n\e \theta_n^{\max\{k+3-\kk+\dt,0\}},\hspace{3.66cm} 0\leq k\leq 2k_0-2, \\[9pt]
\|\partial_{t,x}^k\bar{\eta}^n|_{y=0}\|_{L_{t,x}^{\infty}} \leq C_n + C_n\e \theta_n^{\max\{k+3-\kk+\dt,0\}},\hspace{2.28cm} 0\leq k\leq 2k_0-2, \\[9pt]

\|\zeta^n\|_{\mathcal{A}_{\ell}^k} \leq C_n\e\theta_n^{\max\{k+3-\kk+\dt,0\}},\hspace{4.51cm} 0\leq k\leq 2k_0-2,
\end{array}
\end{equation}

\begin{equation*}
\begin{array}{ll}

\|\partial_{t,x}^{k}\tilde{\zeta}_1^n\|_{L_{t,x}^{\infty}} \leq C_n + C_n \e\theta_n^{\max\{k+3-\kk+\dt,0\}},\hspace{2.95cm} 0\leq k\leq 2k_0-2, \\[9pt]
\|\tilde{\zeta}_2^n\|_{A^k} \leq C_n\e\theta_n^{\max\{k+3-\kk+\dt,0\}},\hspace{4.57cm} 0\leq k\leq 2k_0-2, \\[9pt]

\lambda_k^n \leq C_n + C_n\e\theta_n^{\max\{k+3-\kk+\dt,0\}},\hspace{4.42cm} 0\leq k\leq 2k_0-2.
\end{array}
\end{equation*}
\end{theorem}

The main variables estimated in $(\ref{Sect5_Variables_Estimate})$ have the relationships indicated by the  following diagram, where $A\rto B$ means that we can estimate $B$ after we have estimated $A$.
\begin{equation*}
\begin{array}{ll}
f^a, f^0\hspace{0.56cm} \rto w^0, \hspace{0.51cm} w^0|_{y=0}, \hspace{0.5cm} \rto \delta u^0,\hspace{0.46cm} \delta u^0|_{y=0},\hspace{0.53cm} \delta v^0\hspace{0.5cm} \rto e_0^{(1)}, \hspace{0.33cm} e_0^{(2)}\hspace{0.12cm}\rto \\[8pt]
\rto f^1\hspace{0.64cm} \rto w^1, \hspace{0.51cm} w^1|_{y=0}, \hspace{0.51cm} \rto \delta u^1,\hspace{0.46cm} \delta u^1|_{y=0},\hspace{0.53cm} \delta v^1\hspace{0.5cm} \rto e_1^{(1)}, \hspace{0.33cm} e_1^{(2)}\hspace{0.12cm}\rto \\[6pt]
\rto\cdots \\[6pt]
\rto f^{n-1}\hspace{0.2cm} \rto w^{n-1}, \ w^{n-1}|_{y=0}, \ \rto \delta u^{n-1},\ \delta u^{n-1}|_{y=0},\ \delta v^{n-1}\ \rto e_{n-1}^{(1)},
\ e_{n-1}^{(2)}\rto \\[9pt]
\rto f^n\hspace{0.56cm} \rto w^n, \hspace{0.46cm} w^n|_{y=0}, \hspace{0.47cm} \rto \delta u^n,\hspace{0.48cm} \delta u^n|_{y=0},\hspace{0.47cm} \delta v^n\hspace{0.49cm} \rto e_n^{(1)}, \hspace{0.33cm} e_n^{(2)}\hspace{0.12cm}\rto \\[6pt]
\rto\cdots
\end{array}
\end{equation*}

For the zero-th order variables, it is easy to get the following estimates:
\begin{equation}\label{Sect5_Main_LowerOrder_Estimates_1}
\begin{array}{ll}
\|\tilde{u}^0\|_{\mathcal{A}_{\ell}^{k}} + \big\|\frac{\partial_y \tilde{u}^0}{\partial_y u^s}\big\|_{\mathcal{A}_{\ell}^{k}}
\leq C_0\e, \hspace{3.73cm} 0\leq k\leq k_0+1, \\[8pt]
\sum\limits_{m=0}^{2k}\big\|\partial_y^m\tilde{u}^0|_{y=0}\big\|_{A^{k-[\frac{m+1}{2}]}} \leq C_0\e, \hspace{2.94cm} 0\leq k\leq k_0+1, \\[15pt]
\|v^0\|_{\mathcal{D}_{0}^{k}} \leq C_0\e, \hspace{5.67cm} 0\leq k\leq k_0, \\[10pt]
\|f^a\|_{\mathcal{A}_{\ell}^{k}} + \|f^0\|_{\mathcal{A}_{\ell}^k}
+ \|\frac{f^a}{\partial_y u^s}\|_{\mathcal{A}_{\ell}^{k}}
+ \|\frac{f^0}{\partial_y u^s}\|_{\mathcal{A}_{\ell}^k} \leq C_0\e,\quad 0\leq k\leq k_0, \\[13pt]
\|w^0\|_{\mathcal{A}_{\ell}^k} +\frac{1}{\sqrt{\beta +C_{\eta}}}\big\|w^0|_{y=0}\big\|_{A^k} \leq C_0\e, \hspace{2.25cm}  0\leq k\leq k_0, \\[13pt]
\|\delta u^0\|_{\mathcal{A}_{\ell}^k} + \big\|\frac{\partial_y \delta u^0}{\partial_y u^s}\big\|_{\mathcal{A}_{\ell}^{k}} 
+ \|\frac{\delta u^0}{\partial_y u^s}\|_{\mathcal{D}_0^k}\leq C_0\e, \hspace{1.58cm} 0\leq k\leq k_0, \\[10pt]

\sum\limits_{m=0}^{2k}\big\|\partial_y^m\delta u^0|_{y=0}\big\|_{A^{k-[\frac{m+1}{2}]}}  \leq C_0\e,\hspace{2.79cm} 0\leq k\leq k_0, \\[13pt]
\|\delta v^0\|_{\mathcal{D}_{0}^{k}} \leq C_0\e,\hspace{5.55cm} 0\leq k\leq k_0-1, \\[11pt]
\|e_0^{(1)}\|_{\mathcal{A}_{\ell}^{k}} + \|e_0^{(2)}\|_{\mathcal{A}_{\ell}^{k}} \leq C_0\e^2,\hspace{3.63cm} 0\leq k\leq k_0-1, \\[11pt]
\|\frac{e_0^{(1)}}{\partial_y u^s}\|_{\mathcal{A}_{\ell}^{k}} + \|\frac{e_0^{(2)}}{\partial_y u^s}\|_{\mathcal{A}_{\ell}^{k}} \leq C_0\e^2,\hspace{3.35cm} 0\leq k\leq k_0-1.
\end{array}
\end{equation}
Then the estimates in $(\ref{Sect5_Variables_Estimate})$ and $(\ref{Sect5_Quantities_Estimate})$ hold for $n=0$ by adjusting the constant $C_0$, since $1\ll\theta_0<+\infty,\ 1\ll\theta_1<+\infty$.

\begin{remark}\label{Sect5_Sum_Power_Remark}
For any small $\dt\in (0,1)$, $\dt^{-\frac{1}{\dt}}\leq \theta_0<+\infty$ and $0\leq k\leq 2k_0+2$, 
\begin{equation*}
\begin{array}{ll}
\sum\limits_{m=0}^{j-1} \theta_m^{k-\kk}\triangle\theta_m \lem
\left\{\begin{array}{ll}
\theta_j^{k+1-\kk} + \theta_0^{k+1-\kk}, \hspace{1.82cm} k\neq \kk-1, \\[5pt]
\log\theta_j + \log\theta_0 \lem \theta_j^{\dt} + \theta_0^{\dt},  \hspace{0.5cm} k=\kk-1.
\end{array}\right.
\end{array}
\end{equation*}
Then we have a uniform expression
\begin{equation}\label{Sect5_Sum_Power_Formula}
\begin{array}{ll}
\sum\limits_{m=0}^{j-1} \theta_m^{k-\kk}\triangle\theta_m \lem \theta_j^{\max\{k+1-\kk+\dt,0\}},
\end{array}
\end{equation}
instead of $\sum\limits_{m=0}^{j-1} \theta_m^{k-\kk}\triangle\theta_m \lem \theta_j^{\max\{k+1-\kk,0\}}$ (see \cite{Alexandre_Yang_2014}), because when $k=\kk-1$, $\sum\limits_{m=0}^{j-1} \theta_m^{-1}\triangle\theta_m
\geq \frac{1}{3}\sum\limits_{m=0}^{j-1} (\frac{1}{\theta_m})^2
=\frac{1}{3}\sum\limits_{m=0}^{j-1}\frac{1}{\theta_0^2 + m}$ is not uniformly bounded for any $j$.

Thus, in this paper, we can not use $C_j\e \theta_j^{\max\{k+1-\kk,0\}}$ to bound
$\|u_{\theta_j}^j -u^s\|_{\mathcal{A}_{\ell}^k}$ and $\|v_{\theta_j}^j\|_{\mathcal{D}_0^{k-1}}$, etc, when $k= \kk-1$.
Moreover, due to $(\ref{Sect5_Sum_Power_Formula})$, the stability results in Theorem $\ref{Sect1_Main_Thm}$ hold for $p\leq k-2$ instead of $p\leq k-1$.
\end{remark}

Assume the estimates $(\ref{Sect5_Variables_Estimate})$ and $(\ref{Sect5_Quantities_Estimate})$ hold for $n=1,2,\cdots,j-1$, then we prove the main estimates for $n=j$ in the following lemmas, the rest are easy.

\begin{lemma}\label{Sect5_Lemmas_Velocity}
Assume the conditions are the same with the conditions in Theorem $\ref{Sect5_Main_Thm}$
and the estimates $(\ref{Sect5_Quantities_Estimate})$ hold for $n=1,2,\cdots,j-1$, then
\begin{equation}\label{Sect5_Lemmas_Velocity_1}
\begin{array}{ll}
\|u_{\theta_j}^j -u^s\|_{\mathcal{A}_{\ell}^k} \leq C_j\e \theta_j^{\max\{k+1-\kk+\dt,0\}},\hspace{1.78cm} 0\leq k\leq 2k_0+2, \\[10pt]
\big\|u_{\theta_j}^j|_{y=0} -u^s|_{y=0}\big\|_{A^k} \leq C_j\e \theta_j^{\max\{k+1-\kk+\dt,0\}},\hspace{0.51cm} 0\leq k\leq 2k_0+2, \\[12pt]
\|v_{\theta_j}^j\|_{\mathcal{D}_0^{k}} \leq C_j\e \theta_j^{\max\{k+2-\kk+\dt,0\}},\hspace{2.58cm} 0\leq k\leq 2k_0+1, \\[10pt]

\|\frac{\partial_y (u_{\theta_j}^j -u^s)}{\partial_y u^s}\|_{\mathcal{A}_{\ell}^k} \leq C_j\e \theta_j^{\max\{k+1-\kk+\dt,0\}},
\hspace{1.49cm} 0\leq k\leq 2k_0, \\[12pt]
\big\|\frac{\partial_y (u_{\theta_j}^j -u^s)}{\partial_y u^s}|_{y=0}\big\|_{A^k} \leq C_j\e \theta_j^{\max\{k+1-\kk+\dt,0\}},
\hspace{0.83cm} 0\leq k\leq 2k_0, \\[12pt]

\|(\frac{\partial_y u_{\theta_j}^{j}}{\partial_y u^s})^{-1}\|_{\dot{\mathcal{A}}_{\ell}^k} \leq C_j\e \theta_j^{\max\{k+1-\kk+\dt,0\}},\hspace{1.56cm}
0\leq k\leq 2k_0, \\[12pt]
\big\|(\frac{\partial_y u_{\theta_j}^{j}}{\partial_y u^s})^{-1}|_{y=0}\big\|_{\dot{A}^k} \leq C_j\e \theta_j^{\max\{k+1-\kk+\dt,0\}},\hspace{0.91cm}
0\leq k\leq 2k_0.
\end{array}
\end{equation}
\end{lemma}

\begin{proof}
Estimate of $(\ref{Sect5_Lemmas_Velocity_1})_1$:

When $k\leq \kk-2$, $\|u_{\theta_j}^j -u^s\|_{\mathcal{A}_{\ell}^k}\lem C_j\e$.
When $\kk-1\leq k\leq 2k_0+2$,
\begin{equation}\label{Sect5_Lemmas_Velocity_2}
\begin{array}{ll}
\|u_{\theta_j}^j -u^s\|_{\mathcal{A}_{\ell}^k} \leq \|S_{\theta_j}^u (\tilde{u}^0 + \sum\limits_{0\leq m\leq j-1}\delta u^m)\|_{\mathcal{A}_{\ell}^k}
\\[12pt]\hspace{2.1cm}
\lem \theta_j^{(k+1-\kk)_{+}}\|\tilde{u}^0 + \sum\limits_{0\leq m\leq j-1}\delta u^m\|_{\mathcal{A}_{\ell}^{\kk-1}} \\[12pt]\hspace{2.1cm}
\lem \theta_j^{k+1-\kk}(C_0\e + \e\theta_j^{\dt})
\lem C_j\e \theta_j^{k+1-\kk+\dt}.
\end{array}
\end{equation}
Thus, $\|u_{\theta_j}^j -u^s\|_{\mathcal{A}_{\ell}^k}\lem C_j\e \theta_j^{\max\{k+1-\kk+\dt,0\}}$.

\noindent
Estimate of $(\ref{Sect5_Lemmas_Velocity_1})_2$:

When $k\leq \kk-2$, $\big\|u_{\theta_j}^j|_{y=0} -u^s|_{y=0}\big\|_{A^k}\lem C_j\e$.
When $\kk-1\leq k\leq 2k_0+2$,
\begin{equation}\label{Sect5_Lemmas_Velocity_3}
\begin{array}{ll}
\big\|u_{\theta_j}^j|_{y=0} -u^s|_{y=0}\big\|_{A^k}
= \big\|S_{\theta_j}^u (\tilde{u}^0 + \sum\limits_{0\leq m\leq j-1}\delta u^m)|_{y=0}\big\|_{A^k}
\\[12pt]\hspace{3.34cm}
\lem \theta^{(k+1-\kk)_{+}}\big\|\tilde{u}^0 + \sum\limits_{0\leq m\leq j-1}\delta u^m
\big\|_{\mathcal{A}_{\ell}^{\kk-1}([0,T]\times\mathbb{R}\times[0,\theta_0^{-1}])}
\\[12pt]\hspace{3.34cm}
\lem \theta_j^{k+1-\kk}(C_0\e + \e\theta_j^{\dt})
\lem C_j\e \theta_j^{k+1-\kk+\dt}.
\end{array}
\end{equation}
Thus, $\big\|u_{\theta_j}^j|_{y=0} -u^s|_{y=0}\big\|_{A^k}\lem C_j\e \theta_j^{\max\{k+1-\kk+\dt,0\}}$.

\noindent
Estimate of $(\ref{Sect5_Lemmas_Velocity_1})_3$:

When $k\leq \kk-3$, $\|v_{\theta_j}^j\|_{\mathcal{D}_0^k}\lem C_j\e$.
When $\kk-2\leq k\leq 2k_0+1$,
\begin{equation}\label{Sect5_Lemmas_Velocity_4}
\begin{array}{ll}
\|v_{\theta_j}^j\|_{\mathcal{D}_0^k} \leq \|S_{\theta_j}^v (\tilde{v}^0 + \sum\limits_{0\leq m\leq j-1}\delta v^m)\|_{\mathcal{D}_0^k}
\\[12pt]\hspace{1.3cm}

\lem \theta_j^{(k+2-\kk)_{+}}\|\tilde{v}^0 + \sum\limits_{0\leq m\leq j-1}\delta v^m\|_{\mathcal{D}_{0}^{\kk-2}}\\[14pt]\hspace{1.3cm}
\lem \theta_j^{k+2-\kk}(C_0\e + \e\theta_j^{\dt})
\lem C_j\e \theta_j^{k+2-\kk+\dt}.
\end{array}
\end{equation}
Thus, $\|v_{\theta_j}^j\|_{\mathcal{D}_0^k}\lem C_j\e \theta_j^{\max\{k+2-\kk+\dt,0\}}$.

\noindent
Estimate of $(\ref{Sect5_Lemmas_Velocity_1})_4$:

When $k\leq \kk-2$, $\|\frac{\partial_y (u_{\theta_j}^j -u^s)}{\partial_y u^s}\|_{\mathcal{A}_{\ell}^k}\lem C_j\e$.
When $\kk-1\leq k\leq 2k_0$,
\begin{equation}\label{Sect5_Lemmas_Velocity_5}
\begin{array}{ll}
\|\frac{\partial_y (u_{\theta_j}^j -u^s)}{\partial_y u^s}\|_{\mathcal{A}_{\ell}^k}
\lem \Big\|S_{\theta_j}^u\frac{\partial_y\tilde{u}^0 + \sum\limits_{0\leq m\leq j-1}\partial_y\delta u^m}{\partial_y u^s}\Big\|_{\mathcal{A}_{\ell}^k}
\\[14pt]\hspace{2.4cm}
\lem \theta_j^{(k+1-\kk)_{+}}\Big\|\frac{\partial_y\tilde{u}^0 + \sum\limits_{0\leq m\leq j-1}\partial_y\delta u^m}{\partial_y u^s}\Big\|_{\mathcal{A}_{\ell}^{\kk-1}}\\[14pt]\hspace{2.4cm}

\lem \theta_j^{k+1-\kk}(C_0\e + \e\theta_j^{\dt})
\lem C_j\e \theta_j^{k+1-\kk+\dt}.
\end{array}
\end{equation}
Thus, $\|\frac{\partial_y (u_{\theta_j}^j -u^s)}{\partial_y u^s}\|_{\mathcal{A}_{\ell}^k}\lem C_j\e \theta_j^{\max\{k+1-\kk+\dt,0\}}$.

\noindent
Estimate of $(\ref{Sect5_Lemmas_Velocity_1})_5$:

When $k\leq \kk-2$, $\big\|\frac{\partial_y (u_{\theta_j}^j -u^s)}{\partial_y u^s}|_{y=0}\big\|_{A^k}\lem C_j\e$.
When $\kk-1\leq k\leq 2k_0$,
\begin{equation}\label{Sect5_Lemmas_Velocity_6}
\begin{array}{ll}
\big\|\frac{\partial_y (u_{\theta_j}^j -u^s)}{\partial_y u^s}|_{y=0}\big\|_{A^k}
\lem \Big\|S_{\theta_j}^u\frac{\partial_y\tilde{u}^0 + \sum\limits_{0\leq m\leq j-1}\partial_y\delta u^m}{\partial_y u^s}|_{y=0}\Big\|_{A^k}
\\[13pt]\hspace{3cm}

\lem \theta_j^{(k+1-\kk)_{+}}\Big\|\frac{\partial_y\tilde{u}^0 + \sum\limits_{0\leq m\leq j-1}\partial_y\delta u^m}{\partial_y u^s}\Big\|_{\mathcal{A}_{\ell}^{\kk-1}([0,T]\times\mathbb{R}\times[0,\theta_0^{-1}])}\\[13pt]\hspace{3cm}

\lem \theta_j^{k+1-\kk}(C_0\e + \e\theta_j^{\dt})
\lem C_j\e \theta_j^{k+1-\kk+\dt}.
\end{array}
\end{equation}
Thus, $\big\|\frac{\partial_y (u_{\theta_j}^j -u^s)}{\partial_y u^s}|_{y=0}\big\|_{A^k}\lem C_j\e \theta_j^{\max\{k+1-\kk+\dt,0\}}$.

\noindent
Estimate of $(\ref{Sect5_Lemmas_Velocity_1})_6$: by F$\grave{a}$ Di Bruno formula, we get
\begin{equation}\label{Sect5_Lemmas_Velocity_7}
\begin{array}{ll}
\big\|(\frac{\partial_y u_{\theta_j}^{j}}{\partial_y u^s})^{-1} \big\|_{\dot{\mathcal{A}}_{\ell}^k}
\lem \big\|\frac{\partial_y (u_{\theta_j}^{j} -u^s)}{\partial_y u^s} \big\|_{\dot{\mathcal{A}}_{\ell}^k}
\leq C_j\e \theta_j^{\max\{k+1-\kk+\dt,0\}}.
\end{array}
\end{equation}

\noindent
Estimate of $(\ref{Sect5_Lemmas_Velocity_1})_7$: by F$\grave{a}$ Di Bruno formula, we get
\begin{equation}\label{Sect5_Lemmas_Velocity_8}
\begin{array}{ll}
\big\|(\frac{\partial_y u_{\theta_j}^{j}}{\partial_y u^s})^{-1}|_{y=0} \big\|_{\dot{A}^k}
\lem \big\|\frac{\partial_y (u_{\theta_j}^{j} -u^s)}{\partial_y u^s} \big\|_{\dot{A}^k}
\leq C_j\e \theta_j^{\max\{k+1-\kk+\dt,0\}}.
\end{array}
\end{equation}
\end{proof}

\begin{lemma}\label{Sect5_Lemmas_U}
Assume the conditions are the same with the conditions in Theorem $\ref{Sect5_Main_Thm}$
and the estimates $(\ref{Sect5_Variables_Estimate})$ hold for $n=1,2,\cdots,j-1$, then
\begin{equation}\label{Sect5_Lemmas_U_1}
\begin{array}{ll}
\|\delta u^j\|_{\mathcal{A}_{\ell}^k} \leq C_j\e \theta_j^{\max\{k-\kk,3-\kk\}} \triangle\theta_j,\hspace{3.13cm} 0\leq k\leq k_0, \\[8pt]
\|\delta v^j\|_{\mathcal{D}_{0}^{k}} \leq C_j\e \theta_j^{\max\{k+1-\kk,3-\kk\}} \triangle\theta_j,\hspace{2.8cm} 0\leq k\leq k_0-1, \\[8pt]
\big\|\frac{\partial_y\delta u^j}{\partial_y u^s}\big\|_{\mathcal{A}_{\ell}^k} \leq C_j\e \theta_j^{\max\{k-\kk,3-\kk\}}\triangle\theta_j,
\hspace{2.8cm} 0\leq k\leq k_0, \\[9pt]

\|\frac{\delta u^j}{\partial_y u^s}\|_{\mathcal{D}_0^k} \leq C_j\e \theta_j^{\max\{k-\kk,3-\kk\}} \triangle\theta_j,\hspace{2.97cm} 0\leq k\leq k_0, \\[10pt]
\sum\limits_{m=0}^{2k}\big\|\partial_y^m \delta u^j|_{y=0}\big\|_{A^{k-[\frac{m+1}{2}]}} \leq C_j\e \theta_j^{\max\{k-\kk,3-\kk\}} \triangle\theta_j,\hspace{0.44cm} 0\leq k\leq k_0.
\end{array}
\end{equation}
\end{lemma}

\begin{proof}
Estimate of $(\ref{Sect5_Lemmas_U_1})_1$:
since $\delta u^j = - u_y^s\int\limits_y^{\infty}w\,\mathrm{d}\tilde{y} - \partial_y(u_{\theta_j}^j -u^s)\int\limits_y^{\infty}w\,\mathrm{d}\tilde{y}$,
\vspace{-0.1cm}
\begin{equation}\label{Sect5_Lemmas_U_3}
\begin{array}{ll}
\|\delta u^j\|_{\mathcal{A}_{\ell}^k} \lem |\partial_y u^s|_{\mathcal{C}_{0}^{k+1}}C_{\ell}\|w^j\|_{\mathcal{A}_{\ell}^k}
+ \|\partial_y(u_{\theta_j}^j -u^s)\|_{\mathcal{A}_{\ell}^{k}} C_{\ell}\|w^j\|_{\mathcal{C}_{\ell}^0} \\[7pt]\hspace{1.7cm}

+ \|\partial_y(u_{\theta_j}^j -u^s)\|_{L_{y,\ell}^2(L_{t,x}^{\infty})} C_{\ell}\|w^j\|_{\mathcal{A}_{\ell}^k} \\[9pt]\hspace{1.4cm}

\lem C_s\|w^j\|_{\mathcal{A}_{\ell}^k}
+ \|\partial_y(u_{\theta_j}^j -u^s)\|_{\mathcal{A}_{\ell}^{k}}
\|w^j\|_{\mathcal{A}_{\ell}^2} + \|\partial_y(u_{\theta_j}^j -u^s)\|_{\mathcal{A}_{\ell}^2} \|w^j\|_{\mathcal{A}_{\ell}^k}
\\[11pt]\hspace{1.4cm}

\lem C_j\e\theta_j^{\max\{k-\kk,3-\kk\}}\triangle\theta_j
+ C_j\e\theta_j^{\max\{k+1-\kk+\dt,0\}} C_j\e\theta_j^{3-\kk}\triangle\theta_j
 \\[11pt]\hspace{1.4cm}

 \lem \e\theta_j^{\max\{k-\kk,3-\kk\}}\triangle\theta_j 
+ \e\theta_j^{\max\{k+4-2\kk+\dt,3-\kk\}}\triangle\theta_j \\[8pt]\hspace{1.4cm}

\lem \e\theta_j^{\max\{k-\kk,3-\kk\}}\triangle\theta_j, \quad 0\leq k\leq k_0.
\end{array}
\end{equation}

\noindent
Estimate of $(\ref{Sect5_Lemmas_U_1})_2$:

$\|\delta v^j\|_{\mathcal{D}_0^{k}} = \|-\int\limits_0^y \partial_x(\delta u^j)(t,x,\tilde{y}) \,\mathrm{d}\tilde{y}\|_{\mathcal{D}_0^{k}} \leq C_{\ell}\|\delta u^j\|_{\mathcal{A}_{\ell}^{k+1}},\quad 0\leq k\leq k_0-1$.

\noindent
Estimate of $(\ref{Sect5_Lemmas_U_1})_3$: since
\begin{equation*}
\begin{array}{ll}
\frac{\partial_y \delta u^j}{\partial_y u^s} = w^j \frac{\partial_y (u_{\theta_j}^j- u^s)}{\partial_y u^s}
+ w^j - \frac{\partial_{yy} (u_{\theta_j}^j -u^s)}{\partial_y u^s} \int\limits_y^{\infty} w^j \,\mathrm{d}\tilde{y}
- \frac{\partial_{yy} u^s}{\partial_y u^s} \int\limits_y^{\infty} w^j \,\mathrm{d}\tilde{y},
\end{array}
\end{equation*}
\begin{equation}\label{Sect5_Lemmas_U_4}
\begin{array}{ll}
\big\|\frac{\partial_y \delta u^j}{\partial_y u^s}\big\|_{\mathcal{A}_{\ell}^k}
\lem (1 + C_{\ell}\|\frac{\partial_{yy} u^s}{\partial_y u^s}\|_{\mathcal{C}_{0}^{k+1}}
+\big\|\frac{\partial_y (u_{\theta_j}^j -u^s)}{\partial_y u^s}\big\|_{\infty} + \big\|\frac{\partial_{yy} (u_{\theta_j}^j-u^s)}{\partial_y u^s}\big\|_{\infty}) \|w^j\|_{\mathcal{A}_{\ell}^k} \\[9pt]\hspace{2.1cm}
+ \|w^j\|_{L_{0}^{\infty}}\big\|\frac{\partial_y (u_{\theta_j}^j- u^s)}{\partial_y u^s} \|_{\mathcal{A}_{\ell}^k}
+C_{\ell}\|w^j\|_{\mathcal{C}_{\ell}^0} \big\|\frac{\partial_{yy} (u_{\theta_j}^j -u^s)}{\partial_y u^s}\big\|_{\mathcal{A}_{\ell}^k}

\\[9pt]\hspace{1.7cm}
\lem C_j\e\theta_j^{\max\{k-\kk,3-\kk\}} \triangle\theta_j + C_j\e\theta_j^{3-\kk} \triangle\theta_j C_j\e\theta_j^{\max\{k+2-\kk+\dt,0\}}

\\[9pt]\hspace{1.7cm}
\lem C_j\e\theta_j^{\max\{k-\kk,3-\kk\}} \triangle\theta_j + C_j\e\theta_j^{\max\{k+5-2\kk+\dt,3-\kk\}}\triangle\theta_j

\\[9pt]\hspace{1.7cm}
\lem C_j\e\theta_j^{\max\{k-\kk,3-\kk\}} \triangle\theta_j, \quad 0\leq k\leq k_0.
\end{array}
\end{equation}

\noindent
Estimate of $(\ref{Sect5_Lemmas_U_1})_4$:

Since $\frac{\delta u^j}{\partial_y u^s} =
-\Big(1+ \frac{S_{\theta_j}^u(\partial_y\tilde{u}^0 +\sum\limits_{0\leq m\leq j-1}\partial_y\delta u^m)}{\partial_y u^s} \Big) \int\limits_y^{\infty} w^j(t,x,\tilde{y})\,\mathrm{d}\tilde{y}$,
\begin{equation}\label{Sect5_Lemmas_U_5}
\begin{array}{ll}
\|\frac{\delta u^j}{\partial_y u^s}\|_{\mathcal{D}_0^k}
\lem C_{\ell}\|w^j\|_{\mathcal{A}_{\ell}^k} (1+\big\|\frac{\partial_y (u_{\theta_j}^j -u^s)}{\partial_y u^s}\big\|_{L^{\infty}})
+ C_{\ell}\|w\|_{\mathcal{C}_{\ell}^0}\big\|\frac{\partial_y (u_{\theta_j}^s -u^s)}{\partial_y u^s}\big\|_{\mathcal{D}_0^k}
 \\[10pt]\hspace{1.52cm}

\lem C_{\ell}\|w^j\|_{\mathcal{A}_{\ell}^k}
+ C_{\ell}\|w\|_{\mathcal{C}_{\ell}^0}\|\frac{\partial_y (u_{\theta_j}^s -u^s)}{\partial_y u^s}\|_{\mathcal{A}_{\ell}^{k+1}}
\\[12pt]\hspace{1.52cm}
\lem C_j\e \theta_j^{\max\{k-\kk,3-\kk\}} \triangle\theta_j
+ C_j\e \theta_j^{3-\kk} \triangle\theta_j \e\theta_j^{\max\{k+2-\kk+\dt,0\}}
 \\[9pt]\hspace{1.52cm}
\lem C_j\e \theta_j^{\max\{k-\kk,3-\kk\}} \triangle\theta_j, \quad 0\leq k\leq k_0.
\end{array}
\end{equation}

\noindent
Estimate of $(\ref{Sect5_Lemmas_U_1})_5$:

Since $\partial_y^{m}\int\limits_y^{\infty} w^j(t,x,\tilde{y}) \,\mathrm{d}\tilde{y} =
- \partial_y^{m-1}w = \int\limits_y^{\infty} \partial_y^{m} w^j(t,x,\tilde{y}) \,\mathrm{d}\tilde{y}$, \vspace{-0.1cm}
\begin{equation}\label{Sect5_Lemmas_U_6}
\begin{array}{ll}
\partial_y^{m}\delta u^j|_{y=0}
= - \sum\limits_{m_1+m_2=m}\partial_y^{m_1+1}u^s|_{y=0}\int\limits_0^{\infty} \partial_y^{m_2}w^j(t,x,\tilde{y}) \,\mathrm{d}\tilde{y} \\[8pt]\hspace{2.1cm}
- \sum\limits_{m_1+m_2=m} \partial_y^{m_1+1}(u_{\theta_j}^j -u^s)|_{y=0} \int\limits_0^{\infty} \partial_y^{m_2}w^j(t,x,\tilde{y}) \,\mathrm{d}\tilde{y},
\end{array}
\end{equation}
then for $0\leq m\leq 2k$,
\begin{equation}\label{Sect5_Lemmas_U_7}
\begin{array}{ll}
\|\partial_y^{m}\delta u^j|_{y=0}\|_{A^{k-[\frac{m+1}{2}]}} \\[7pt]
\lem |\partial_y u^s|_{y=0}|_{\infty} C_{\ell}\|\partial_y^{m}w^j\|_{\mathcal{A}_{\ell}^{k-[\frac{m+1}{2}]}}
+ |\partial_{t,x}^{k-[\frac{m+1}{2}]}\partial_y^{m+1}u^s|_{y=0}|_{\infty} \|w^j\|_{\mathcal{A}_{\ell}^0} \\[10pt]\quad

+ C_{\ell}\|\partial_y(u_{\theta_j}^j -u^s)|_{y=0}\|_{L_{t,x}^{\infty}}
C_{\ell}\|\partial_y^{m}w^j\|_{\mathcal{A}_{\ell}^{k-[\frac{m+1}{2}]}} \\[11pt]\quad

+ C_{\ell}\|\partial_y^{m+1}(u_{\theta_j}^j -u^s)|_{y=0}\|_{A^{k-[\frac{m+1}{2}]}}
C_{\ell}\|w^j\|_{\mathcal{C}_{\ell}^0} \\[12pt]

\lem \|u^s\|_{\mathcal{C}^{k+2}}\|w^j\|_{\mathcal{A}_{\ell}^0}
+ \|\partial_y^{m}w^j\|_{\mathcal{A}_{\ell}^{k-[\frac{m+1}{2}]}}
+ \|w^j\|_{\mathcal{A}_{\ell}^2}\|(u_{\theta_j}^j -u^s)|_{y=0}\|_{A^{k+1}} \\[12pt]

\lem  C_j\e \theta_j^{\max\{k-\kk,3-\kk\}} \triangle\theta_j
+ C_j\e \theta_j^{3-\kk} \triangle\theta_j
\Big(C_s + \e \theta_j^{\max\{k+2-\kk+\dt,0\}} \Big) \hspace{0.6cm}
\\[7pt]

\lem C_j\e \theta_j^{\max\{k-\kk,3-\kk\}} \triangle\theta_j, \quad 0\leq k\leq k_0.
\end{array}
\end{equation}
\end{proof}

\vspace{0.1cm}
\begin{lemma}\label{Sect5_Lemmas_E1}
Assume the conditions are the same with the conditions in Theorem $\ref{Sect5_Main_Thm}$
and the estimates $(\ref{Sect5_Variables_Estimate})$ hold for $n=1,2,\cdots,j-1$, then
\begin{equation}\label{Sect5_Lemmas_E1_1}
\begin{array}{ll}
\|e_j^{(1)}\|_{\mathcal{A}_{\ell}^{k}} \leq C_j\e^2 \theta_j^{\max\{k+3-2\kk,5-2\kk\}} \triangle\theta_j, \hspace{0.5cm}
 0\leq k\leq k_0-1. \\[10pt]
\|\frac{e_j^{(1)}}{\partial_y u^s}\|_{\mathcal{A}_{\ell}^{k}} \leq C_j\e^2 \theta_j^{\max\{k+3-2\kk,5-2\kk\}} \triangle\theta_j,\quad 0\leq k\leq k_0-1.
\end{array}
\end{equation}
\end{lemma}

\begin{proof}
Estimate of $(\ref{Sect5_Lemmas_E1_1})_1$:
by $(\ref{Sect4_Increment_en1})$, $e_j^{(1)} = \delta u^{j}(\delta u^j)_x + \delta v^{j}(\delta u^j)_y$, we get
\begin{equation}\label{Sect5_Lemmas_E1_2}
\begin{array}{ll}
\|e_j^{(1)}\|_{\mathcal{A}_{\ell}^{k}}
\lem \|(\delta u^j)_x\|_{\mathcal{A}_{\ell}^k}\|\delta u^j\|_{L^{\infty}}
+ \|(\delta u^j)_x\|_{L^{\infty}}\|\delta u^j\|_{\mathcal{A}_{\ell}^k} \\[11pt]\quad
+ \|\delta v^j\|_{\mathcal{D}_0^2}\|(\delta u^j)_y\|_{\mathcal{A}_{\ell}^k}
+ \|\delta v^j\|_{\mathcal{D}_0^{k}}\|(\delta u^j)_y\|_{\mathcal{C}_{\ell}^0} \\[11pt]

\lem C_j\e \theta_j^{\max\{k+1-\kk,3-\kk\}} \triangle\theta_j C_j\e \theta_j^{3-\kk} \triangle\theta_j
+ C_j\e \theta_j^{3-\kk} \triangle\theta_j C_j\e \theta_j^{\max\{k-\kk,3-\kk\}} \triangle\theta_j \\[11pt]\quad
+ C_j\e \theta_j^{3-\kk} \triangle\theta_j C_j\e \theta_j^{\max\{k-\kk,3-\kk\}} \triangle\theta_j
+ C_j\e \theta_j^{\max\{k+1-\kk,3-\kk\}} \triangle\theta_j C_j\e \theta_j^{3-\kk} \triangle\theta_j \\[11pt]

\lem \e^2 \theta_j^{\max\{k+3-2\kk,5-2\kk\}} \triangle\theta_j, \quad 0\leq k\leq k_0-1.
\end{array}
\end{equation}

\noindent
Estimate of $(\ref{Sect5_Lemmas_E1_1})_2$:
since $\frac{e_j^{(1)}}{\partial_y u^s}  = \frac{\delta u^{j}}{\partial_y u^s}(\delta u^j)_x + \delta v^{j}\frac{\partial_y \delta u^j}{\partial_y u^s}$, we get
\begin{equation}\label{Sect5_Lemmas_E1_3}
\begin{array}{ll}
\|\frac{e_j^{(1)}}{\partial_y u^s}\|_{\mathcal{A}_{\ell}^{k}}
\lem \|\frac{\delta u^{j}}{\partial_y u^s}\|_{\mathcal{D}_0^2}\|(\delta u^j)_x\|_{\mathcal{A}_{\ell}^{k}}
+ \|\frac{\delta u^{j}}{\partial_y u^s}\|_{\mathcal{D}_{0}^{k}} \|(\delta u^j)_x\|_{C_{\ell}^0} \hspace{3.2cm}
\\[9pt]\quad

+ \|\delta v^{j}\|_{\mathcal{D}_{0}^{2}}\|\frac{\partial_y \delta u^j}{\partial_y u^s}\|_{\mathcal{A}_{\ell}^{k}}
+ \|\delta v^{j}\|_{\mathcal{D}_{0}^{k}}\|\frac{\partial_y \delta u^j}{\partial_y u^s}\|_{\mathcal{C}_{\ell}^{0}}
\\[9pt]\quad

\lem C_j\e \theta_j^{3-\kk} \triangle\theta_j C_j\e \theta_j^{\max\{k+1-\kk,3-\kk\}} \triangle\theta_j
+ C_j\e \theta_j^{\max\{k-\kk,3-\kk\}} \triangle\theta_j C_j\e \theta_j^{3-\kk} \triangle\theta_j \\[8pt]\quad
+ C_j\e \theta_j^{3-\kk} \triangle\theta_j C_j\e \theta_j^{\max\{k-\kk,3-\kk\}} \triangle\theta_j
+ C_j\e \theta_j^{\max\{k+1-\kk,3-\kk\}} \triangle\theta_j C_j\e \theta_j^{3-\kk} \triangle\theta_j \\[9pt]

\lem \e^2 \theta_j^{\max\{k+3-2\kk,5-2\kk\}} \triangle\theta_j, \quad 0\leq k\leq k_0-1.
\end{array}
\end{equation}
\end{proof}

\begin{lemma}\label{Sect5_Lemmas_E2}
Assume the conditions are the same with the conditions in Theorem $\ref{Sect5_Main_Thm}$
and the estimates $(\ref{Sect5_Variables_Estimate})$ hold for $n=1,2,\cdots,j-1$, then
\begin{equation}\label{Sect5_Lemmas_E2_1}
\begin{array}{ll}
\|e_j^{(2)}\|_{\mathcal{A}_{\ell}^{k}} \leq C_j\e^2\theta_j^{\max\{k+5-2\kk,7-2\kk\}} \triangle\theta_j, \hspace{0.54cm}
 0\leq k\leq k_0-1, \\[10pt]
\|\frac{e_j^{(2)}}{\partial_y u^s}\|_{\mathcal{A}_{\ell}^{k}} \leq C_j\e^2\theta_j^{\max\{k+5-2\kk,7-2\kk\}}\triangle\theta_j, \quad 0\leq k\leq k_0-1.
\end{array}
\end{equation}
\end{lemma}

\begin{proof}
By Remark $\ref{Sect5_Sum_Power_Remark}$, we have the estimates:
\begin{equation*}
\begin{array}{ll}
\|\tilde{u}^j\|_{\mathcal{A}_{\ell}^{\kk-1}}\lem \e\theta_j^{\dt},\quad
\|\frac{\partial_y\tilde{u}^j}{\partial_y u^s}\|_{\mathcal{A}_{\ell}^{\kk-1}}\lem \e\theta_j^{\dt},\quad
\|v^j\|_{\mathcal{D}_0^{\kk-2}}\lem \e\theta_j^{\dt}.
\end{array}
\end{equation*}
However, $k_0 \geq \kk+2 > \kk-1$, then we still have
\begin{equation*}
\begin{array}{ll}
\|\tilde{u}^j\|_{\mathcal{A}_{\ell}^{k_0}}\lem \e\theta_j^{k_0+1-\kk},\quad
\|\frac{\partial_y\tilde{u}^j}{\partial_y u^s}\|_{\mathcal{A}_{\ell}^{k_0}}\lem \e\theta_j^{k_0+1-\kk},\quad
\|v^j\|_{\mathcal{D}_0^{k_0-1}}\lem \e\theta_j^{k_0+1-\kk}.
\end{array}
\end{equation*}

\noindent
Estimate of $(\ref{Sect5_Lemmas_E2_1})_1$:

Since $e_j^{(2)} = [(1- S_{\theta_j}^u)\tilde{u}^j \delta u^j]_x + (1- S_{\theta_j}^v)v^j(\delta u^j)_y
+ \delta v^{j} [(1- S_{\theta_j}^u) \tilde{u}^{j}]_y$,
\begin{equation}\label{Sect5_Lemmas_E2_2}
\begin{array}{ll}
\|e_j^{(2)}\|_{\mathcal{A}_{\ell}^{k}} \\[10pt]
\lem \|(1- S_{\theta_j}^u)\tilde{u}^j\delta u^j\|_{\mathcal{A}_{\ell}^{k+1}}
+ \|\delta v^j [(1- S_{\theta_j}^u) \tilde{u}^j]_y\|_{\mathcal{A}_{\ell}^{k}}
+ \|(1- S_{\theta_j}^v)v^j(\delta u^j)_y\|_{\mathcal{A}_{\ell}^{k}} \\[10pt]

\lem \|(1- S_{\theta_j}^u)\tilde{u}^j\|_{\mathcal{A}_{\ell}^2}\|\delta u^j\|_{\mathcal{A}_{\ell}^{k+1}}
+ \|\delta u^j\|_{\mathcal{A}_{\ell}^2}\|(1- S_{\theta_j}^u)\tilde{u}^j\|_{\mathcal{A}_{\ell}^{k+1}} \\[12pt]\quad
+ \|\delta v^j \|_{\mathcal{D}_{0}^{2}} \|(1- S_{\theta_j}^u)\partial_y\tilde{u}^j\|_{\mathcal{A}_{\ell}^{k}}
+ \|\delta v^j \|_{\mathcal{D}_{0}^{k}} \|(1- S_{\theta_j}^u)\partial_y\tilde{u}^j\|_{\mathcal{C}_{\ell}^{0}} \\[12pt]\quad
+ \|(1- S_{\theta_j}^v)v^j\|_{\mathcal{D}_{0}^{2}}\|\partial_y\delta u^j\|_{\mathcal{A}_{\ell}^{k}}
+ \|(1- S_{\theta_j}^v)v^j\|_{\mathcal{D}_{0}^{k}}\|\partial_y\delta u^j\|_{\mathcal{C}_{\ell}^{0}} \\[12pt]

\lem \theta_j^{2-k_0}\big\|\tilde{u}^j\big\|_{\mathcal{A}_{\ell}^{k_0}}\|\delta u^j\|_{\mathcal{A}_{\ell}^{k+1}}
+ \|\delta u^j\|_{\mathcal{A}_{\ell}^{2}}\theta_j^{k+1-k_0}\big\|\tilde{u}^j\big\|_{\mathcal{A}_{\ell}^{k_0}} \\[12pt]\quad

+ \|\delta v^j \|_{\mathcal{D}_{0}^{2}} \theta_j^{k-k_0}\big\|\partial_y\tilde{u}^j\big\|_{\mathcal{A}_{\ell}^{k_0}}
+ \|\delta v^j \|_{\mathcal{D}_{0}^{k}} \theta_j^{2-k_0}\big\|\partial_y\tilde{u}^j\big\|_{\mathcal{A}_{\ell}^{k_0}} \\[12pt]\quad

+ \theta_j^{3-k_0}\|v^j\|_{\mathcal{D}_{0}^{k_0-1}}\|\partial_y\delta u^j\|_{\mathcal{A}_{\ell}^{k}}
+ \theta_j^{k+1-k_0}\|v^j\|_{\mathcal{D}_{0}^{k_0-1}}\|\partial_y\delta u^j\|_{\mathcal{A}_{\ell}^{2}} \\
\end{array}
\end{equation}

\noindent
Estimate of $(\ref{Sect5_Lemmas_E2_1})_2$:

Since $\frac{e_j^{(2)}}{\partial_y u^s} = [(1- S_{\theta_j}^u)\tilde{u}^j \frac{\delta u^j}{\partial_y u^s}]_x
+ (1- S_{\theta_j}^v)v^j\frac{(\delta u^j)_y}{\partial_y u^s} + \delta v^{j} \frac{[(1- S_{\theta_j}^u) \tilde{u}^{j}]_y}{\partial_y u^s}$,
\begin{equation}\label{Sect5_Lemmas_E2_3}
\begin{array}{ll}
\|\frac{e_j^{(2)}}{\partial_y u^s}\|_{\mathcal{A}_{\ell}^{k}} \\[11pt]

\lem \|(1- S_{\theta_j}^u)\tilde{u}^j \frac{\delta u^j}{\partial_y u^s}\|_{\mathcal{A}_{\ell}^{k+1}}
+ \|\delta v^{j} \frac{[(1- S_{\theta_j}^u) \tilde{u}^{j}]_y}{\partial_y u^s}\|_{\mathcal{A}_{\ell}^{k}}
+ \|(1- S_{\theta_j}^v)v^j\frac{(\delta u^j)_y}{\partial_y u^s}\|_{\mathcal{A}_{\ell}^{k}} \\[11pt]

\lem \|(1- S_{\theta_j}^u)\tilde{u}^j\|_{\mathcal{C}_{\ell}^0} \|\frac{\delta u^j}{\partial_y u^s} \|_{\mathcal{D}_{0}^{k+1}}
+ \|(1- S_{\theta_j}^u)\tilde{u}^j\|_{\mathcal{A}_{\ell}^{k+1}} \|\frac{\delta u^j}{\partial_y u^s} \|_{\mathcal{D}_{0}^{2}}

\\[11pt]\quad
+ \|\delta v^{j}\|_{\mathcal{D}_{0}^{2}} \|\frac{[(1- S_{\theta_j}^u) \tilde{u}^{j}]_y}{\partial_y u^s}\|_{\mathcal{A}_{\ell}^{k}}
+ \|\delta v^{j}\|_{\mathcal{D}_{0}^{k}} \|\frac{[(1- S_{\theta_j}^u) \tilde{u}^{j}]_y}{\partial_y u^s}\|_{\mathcal{C}_{\ell}^{0}}
\end{array}
\end{equation}

\begin{equation*}
\begin{array}{ll}
\quad + \|(1- S_{\theta_j}^v)v^j\|_{\mathcal{D}_0^2}\|\frac{(\delta u^j)_y}{\partial_y u^s}\|_{\mathcal{A}_{\ell}^{k}}
+ \|(1- S_{\theta_j}^v)v^j\|_{\mathcal{D}_{0}^{k}}\|\frac{(\delta u^j)_y}{\partial_y u^s}\|_{\mathcal{C}_{\ell}^{0}}
\\[11pt]

\lem \theta_j^{2-k_0}\|\tilde{u}^j\|_{\mathcal{A}_{\ell}^{k_0}} \|\frac{\delta u^j}{\partial_y u^s} \|_{\mathcal{D}_{0}^{k+1}}
+ \theta_j^{k+1-k_0}\|\tilde{u}^j\|_{\mathcal{A}_{\ell}^{k_0}} \|\frac{\delta u^j}{\partial_y u^s} \|_{\mathcal{D}_{0}^{2}} \hspace{2.6cm}
\\[11pt]\quad

+ \|\delta v^{j}\|_{\mathcal{D}_{0}^{2}} \theta_j^{k-k_0}\|\frac{\partial_y \tilde{u}^{j}}{\partial_y u^s}\|_{\mathcal{A}_{\ell}^{k_0}}
+ \|\delta v^{j}\|_{\mathcal{D}_{0}^{k}} \theta_j^{2-k_0}\|\frac{\partial_y\tilde{u}^{j}}{\partial_y u^s}\|_{\mathcal{A}_{\ell}^{k_0}} \\[11pt]\quad

+ \theta_j^{3-k_0}\|v^j\|_{\mathcal{D}_0^{k_0-1}}\|\frac{(\delta u^j)_y}{\partial_y u^s}\|_{\mathcal{A}_{\ell}^{k}}
+ \theta_j^{k+1-k_0}\|v^j\|_{\mathcal{D}_{0}^{k_0-1}}\|\frac{(\delta u^j)_y}{\partial_y u^s}\|_{\mathcal{C}_{\ell}^{0}} \\[11pt]

\lem \e\theta_j^{3-\kk}\|\frac{\delta u^j}{\partial_y u^s} \|_{\mathcal{D}_{0}^{k+1}}
+ \e\theta_j^{k+2-\kk} \|\frac{\delta u^j}{\partial_y u^s} \|_{\mathcal{D}_{0}^{2}}

+ \e\theta_j^{k+1-\kk}\|\delta v^{j}\|_{\mathcal{D}_{0}^{2}} \\[11pt]\quad
+ \e\theta_j^{3-\kk}\|\delta v^{j}\|_{\mathcal{D}_{0}^{k}}

+ \e\theta_j^{4-\kk}\|\frac{(\delta u^j)_y}{\partial_y u^s}\|_{\mathcal{A}_{\ell}^{k}}
+ \e\theta_j^{k+2-\kk}\|\frac{(\delta u^j)_y}{\partial_y u^s}\|_{\mathcal{A}_{\ell}^{2}} \hspace{1.6cm}
\\[11pt]

\lem \e^2\theta_j^{3-\kk}C_j\theta_j^{\max\{k+1-\kk,3-\kk\}}\triangle\theta_j
+ \e^2\theta_j^{k+2-\kk}C_j\theta_j^{3-\kk}\triangle\theta_j \\[11pt]\quad

+ \e^2\theta_j^{k+1-\kk}C_j\theta_j^{3-\kk}\triangle\theta_j
+ \e^2\theta_j^{3-\kk}C_j\theta_j^{\max\{k+1-\kk,3-\kk\}}\triangle\theta_j  \\
\end{array}
\end{equation*}
\end{proof}

\begin{lemma}\label{Sect5_Lemmas_F}
Assume the conditions are the same with the conditions in Theorem $\ref{Sect5_Main_Thm}$
and the estimates $(\ref{Sect5_Variables_Estimate})$ hold for $n=1,2,\cdots,j-1$, then
\begin{equation}\label{Sect5_Lemmas_F_1}
\begin{array}{ll}
\|\frac{f^j}{\partial_y u^s }\|_{\mathcal{A}_{\ell}^k} \leq C_j\e \theta_j^{\max\{k-\kk,3-\kk\}} \triangle\theta_j,\quad 0\leq k\leq k_0, \\[8pt]
\|\tilde{f}^j\|_{\mathcal{A}_{\ell}^k} \leq C_j\e \theta_j^{\max\{k-\kk,3-\kk\}} \triangle\theta_j,\hspace{0.7cm} 0\leq k\leq k_0.
\end{array}
\end{equation}
\end{lemma}

\begin{proof}
Estimate of $(\ref{Sect5_Lemmas_F_1})_1$:

Since $\frac{f^1}{\partial_y u^s} = (S_{\theta_0}-S_{\theta_1})\frac{f^a}{\partial_y u^s} + [\frac{1}{\partial_y u^s},S_{\theta_0}-S_{\theta_1}]f^a
- S_{\theta_1}\frac{e_0}{\partial_y u^s} - [\frac{1}{\partial_y u^s},S_{\theta_1}]e^0$, \\
we have
\begin{equation}\label{Sect5_Lemmas_F_4}
\begin{array}{ll}
\|\frac{f^1}{\partial_y u^s}\|_{\mathcal{A}_{\ell}^k} 
\lem \|(S_{\theta_0}-S_{\theta_1})\frac{f^a}{\partial_y u^s}\|_{\mathcal{A}_{\ell}^k}
+ \|S_{\theta_1}\frac{e_0}{\partial_y u^s} \|_{\mathcal{A}_{\ell}^k} \\[10pt]\hspace{1.5cm}

\lem \theta_1^{k-k_0}\triangle\theta_1\|\frac{f^a}{\partial_y u^s}\|_{\mathcal{A}_{\ell}^{k_0}}
+ \theta_1^{(k+1-k_0)_{+}}\|\frac{e_0}{\partial_y u^s} \|_{\mathcal{A}_{\ell}^{k_0-1}} \\[10pt]\hspace{1.5cm}

\lem C_0\e\theta_1^{k-\kk}\triangle\theta_1 \cdot\theta_1^{\kk-k_0}
+ C_0\e^2(1+\theta_1^{k-\kk}\triangle\theta_1 \cdot\theta_1^{\kk+2-k_0})  \\[10pt]\hspace{1.5cm}

\lem C_0\e\theta_1^{\max\{k-\kk,3-\kk\}}\triangle\theta_1.
\end{array}
\end{equation}
where $\e\lem \theta_1^{3-\kk}\triangle\theta_1$ by adjusting the constants, $k_0\geq \kk+2$.

By Lemma $\ref{Sect5_Lemmas_E1}$ and Lemma $\ref{Sect5_Lemmas_E2}$, we have
\begin{equation}\label{Sect5_Lemmas_F_2}
\begin{array}{ll}
\|\frac{e_j}{\partial_y u^s}\|_{\mathcal{A}_{\ell}^{k}}
\leq C_j\e^2\theta_j^{\max\{k+5-2\kk,7-2\kk\}}\triangle\theta_j,\quad 0\leq k\leq k_0-1.
\end{array}%
\end{equation}

When $\kk+1\leq k\leq k_0$, \\[3pt]
$\|S_{\theta_j} \frac{e_{j-1}}{\partial_y u^s} \|_{\mathcal{A}_{\ell}^k}
\lem \theta_j^{k-\kk-1} \|\frac{e_{j-1}}{\partial_y u^s} \|_{\mathcal{A}_{\ell}^{\kk+1}}
\lem \theta_j^{k-\kk}\triangle\theta_j \big(3\|\frac{e_{j-1}}{\partial_y u^s} \|_{\mathcal{A}_{\ell}^{\kk+1}} \big)
\lem \e^2\theta_j^{k-\kk}\triangle\theta_j.$

When $0\leq k\leq \kk$, $\|S_{\theta_j} \frac{e_{j-1}}{\partial_y u^s} \|_{\mathcal{A}_{\ell}^k}
\leq \|\frac{e_{j-1}}{\partial_y u^s} \|_{\mathcal{A}_{\ell}^k}
\lem \e^2\theta_{j-1}^{\max\{k+5-2\kk,7-2\kk\}-1} \\[5pt]
\lem \e^2(\frac{\theta_{j}}{2})^{\max\{k+4-2\kk,6-2\kk\}}
\lem \e^2\theta_{j}^{\max\{k-\kk,3-\kk\}}\triangle\theta_j\big[ 3\theta_j^{5-\kk}2^{\min\{2\kk-k-4,2\kk-6\}}\big] \\[4pt]
\lem \e^2\theta_{j}^{\max\{k-\kk,3-\kk\}}\triangle\theta_j$,
where $\theta_j^{\kk-5}\geq 3\cdot 2^{\min\{2\kk-k-4,2\kk-6\}}$ when $\theta_j\gg 1$.

\vspace{0.2cm}
\noindent
Then for $0\leq k\leq k_0$, \vspace{-0.1cm}
\begin{equation}\label{Sect5_Lemmas_F_6}

\end{equation}
\end{proof}

\begin{lemma}\label{Sect5_Lemmas_Eta}
Assume the conditions are the same with the conditions in Theorem $\ref{Sect5_Main_Thm}$
and the estimates $(\ref{Sect5_Quantities_Estimate})$ hold for $n=1,2,\cdots,j-1$, then
\begin{equation}\label{Sect5_Lemmas_Eta_1}
%
\end{equation}
\end{proof}

\vspace{-0.2cm}
\begin{lemma}\label{Sect5_Lemmas_Zeta}
Assume the conditions are the same with the conditions in Theorem $\ref{Sect5_Main_Thm}$
and the estimates $(\ref{Sect5_Quantities_Estimate})$ hold for $n=1,2,\cdots,j-1$, then
\begin{equation}\label{Sect5_Lemmas_Zeta_1}
\begin{array}{ll}
\|\zeta^j\|_{\mathcal{A}_{\ell}^k} \leq C_j\e\theta_j^{\max\{k+3-\kk+\dt,0\}},\quad 0\leq k\leq 2k_0-2.
\end{array}
\end{equation}
\end{lemma}

\begin{proof}
In order to prove $(\ref{Sect5_Lemmas_Zeta_1})$, we need to prove the following estimates: \vspace{-0.1cm}
\begin{equation}\label{Sect5_Lemmas_Zeta_2}
%
\end{equation}
\end{proof}

\begin{lemma}\label{Sect5_Lemmas_Zeta2}
Assume the conditions are the same with the conditions in Theorem $\ref{Sect5_Main_Thm}$
and the estimates $(\ref{Sect5_Quantities_Estimate})$ hold for $n=1,2,\cdots,j-1$, then
\begin{equation}\label{Sect5_Lemmas_Zeta2_1}
%
\end{equation}

\vspace{0.2cm}
In order to prove $(\ref{Sect5_Lemmas_Zeta2_1})_2$, we need to prove the following estimates:
\begin{equation}\label{Sect5_Lemmas_Zeta2_2}
%
\end{equation}
\end{proof}

\vspace{-0.5cm}
\begin{lemma}\label{Sect5_Lemmas_Lambda}
Assume the conditions are the same with the conditions in Theorem $\ref{Sect5_Main_Thm}$
and the estimates $(\ref{Sect5_Quantities_Estimate})$ hold for $n=1,2,\cdots,j-1$, then
\begin{equation}\label{Sect5_Lemmas_Lambda_1}
\begin{array}{ll}
\lambda_k^j\leq C_j + C_j\e\theta_j^{\max\{k+3-\kk+\dt,0\}},\quad 0\leq k\leq 2k_0-2.
\end{array}
\end{equation}
\end{lemma}

\begin{proof}
By the estimates in $(\ref{Sect5_Quantities_Estimate})$, we have
\begin{equation}\label{Sect5_Lemmas_Lambda_2}
%
\end{equation*}

Thus, when $k\leq \kk-4$, $\lambda_k^j$ is bounded.
$\kk\geq 7$, then $\lambda_3^j$ is bounded.
\end{proof}

\vspace{0.4cm}
Since estimate of $\lambda_{k}$ is obtained, we have the following lemma:
\begin{lemma}\label{Sect5_Lemmas_W}
Assume the conditions are the same with the conditions in Theorem $\ref{Sect5_Main_Thm}$
and the estimates $(\ref{Sect5_Variables_Estimate})$ hold for $n=1,2,\cdots,j-1$, then
\begin{equation}\label{Sect5_Lemmas_W_1}
%
\end{equation}
where $\kk\geq 7$.
\end{proof}

\subsection{Convergence of the Nash-Moser-H$\ddot{o}$rmander Iteration}
In this subsection, we prove the convergence of the Nash-Moser-H$\ddot{o}$rmander Iteration, which implies the existence of classical solutions to the Prandtl system $(\ref{Sect1_PrandtlEq})$. The existence theorem is stated as follows:
\begin{theorem}\label{Sect5_Convergence_Thm}
Assume for any $0<\delta_{\beta}\leq \beta< +\infty$, the conditions are the same with the conditions in Theorem $\ref{Sect5_Main_Thm}$ and $\e$ is suitably small, then there exists $T\in (0,+\infty)$ such that the Prandtl system $(\ref{Sect1_PrandtlEq})$ admits a classical solution $(u,v)$ satisfying the monotonicity conditions $(\ref{Sect1_Solution_Monotonicity})$ and
\begin{equation}\label{Sect5_Solution_Regularity}
\begin{array}{ll}
u-u^s \in \mathcal{A}_{\ell}^k([0,T]\times\mathbb{R}_{+}^2), \hspace{1.2cm}
\frac{\partial_y(u-u^s)}{\partial_y u^s} \in \mathcal{A}_{\ell}^k([0,T]\times\mathbb{R}_{+}^2), \\[12pt]

v \in \mathcal{D}_0^{k-1}([0,T]\times\mathbb{R}_{+}^2), \hspace{1.64cm}
\partial_y v,\ \partial_{yy} v \in \mathcal{A}_{\ell}^{k-1}([0,T]\times\mathbb{R}_{+}^2), \\[12pt]

\partial_y^{j} u|_{y=0}- \partial_y^{j}u^s|_{y=0} \in A^{k-[\frac{j+1}{2}]}([0,T]\times\mathbb{R}), \quad 0\leq j\leq 2k, \\[12pt]
\partial_y^{j+1} v|_{y=0} \in A^{k-1-[\frac{j+1}{2}]}([0,T]\times\mathbb{R}), \hspace{1.4cm} 0\leq j\leq 2k-2.
\end{array}
\end{equation}
where $k\leq \kk-2$.

As $\beta\rto +\infty$, $(u,v)$ satisfy $(\ref{Sect5_Solution_Regularity})$ uniformly.
When $\beta=+\infty$, $(\ref{Sect5_Solution_Regularity})$ holds.
\end{theorem}

\begin{proof}
By $(\ref{Sect5_Variables_Estimate})$ and $(\ref{Sect5_Main_LowerOrder_Estimates_1})$, there exist two positive constants $c_5,c_6$ which depend on $T$,
such that
$\big\|\frac{\partial_y\tilde{u}_0 + \sum\limits_{m=0}^{\infty}\partial_y \delta u^m}{\partial_y u^s}\big\|_{L^{\infty}}\leq c_5\e,\
\big\|\frac{\partial_{yy}\tilde{u}_0 + \sum\limits_{m=0}^{\infty}\partial_{yy} \delta u^m}{\partial_y u^s}\big\|_{L^{\infty}}\leq c_6\e$.

There exist two constant $c_7,c_8$, such that $c_7>c_5,\ c_8>c_6$ and
\begin{equation*}
\begin{array}{ll}
\big\|\frac{\partial_y (u_{\theta_n}^n -u^s)}{\partial_y u^s}\big\|_{L^{\infty}}
\leq \big(1+C(\|\frac{u_{yy}^s}{u_y^s}\|_{\mathcal{C}_0^2})\big)\|\varrho\|_{L^1}\big\|\frac{\partial_y (u^n -u^s)}{\partial_y u^s}\big\|_{L^{\infty}} \leq c_7\e,\\[5pt]
\big\|\frac{\partial_{yy} (u_{\theta_n}^n -u^s)}{\partial_y u^s}\big\|_{L^{\infty}}
\leq \big(1+C(\|\frac{u_{yy}^s}{u_y^s}\|_{\mathcal{C}_0^2})\big)\|\varrho\|_{L^1}\big\|\frac{\partial_{yy} (u^n -u^s)}{\partial_y u^s}\big\|_{L^{\infty}} \leq c_8\e
\end{array}
\end{equation*}

For $\forall\delta\in (0,\delta_s)$, there is $T\in (0,\infty)$ such that $\frac{1}{2c_7}\geq\e, \frac{\delta_s-\delta}{c_7(C_{\eta}+\delta) + c_8 + c_7\delta}\geq \e$. Then for any $y\geq 0,\ n\geq 0,\ t\in [0,T]$, $\frac{\partial_y u_{\theta_n}^n}{\partial_y u^s} =  1 + \frac{\partial_y (u_{\theta_n}^n -u^s)}{\partial_y u^s}
\geq 1 - c_7\e \geq \frac{1}{2}>0$.

$u_{\theta_n}^n>0$ is due to $u_{\theta_n}^n|_{y=0} =\frac{1}{\beta}\partial_y u_{\theta_n}^n|_{y=0} >0$
and $\partial_y u_{\theta_n}^n>0$.

Since $\eta^n = (\frac{\partial_y u_{\theta_n}^n}{\partial_y u^s})^{-1} \big[\frac{u_{yy}^s}{u_y^s} + \frac{\partial_{yy} (u_{\theta_n}^n -u^s)}{u_y^s}\big]$,
\begin{equation}\label{Sect5_Apriori_Condition_2}
\begin{array}{ll}
|\eta^n|_{\infty} \leq
(1-c_7\e)^{-1} \big(\frac{u_{yy}^s}{u_y^s} + c_8\e\big) \leq 2\big(\frac{u_{yy}^s}{u_y^s} + c_8\e\big)<+\infty,
\end{array}
\end{equation}
uniformly for any $n\geq 0$, then we let $C_{\eta} =\max\limits_{n\geq 0}\{|\eta^n|_{\infty}\}$.

When $\beta\geq C_{\eta}+\delta$, $\beta - \eta^n\geq\delta>0$. When $\beta< C_{\eta}+\delta$,
\begin{equation}\label{Sect5_Apriori_Condition_1}
\begin{array}{ll}
\beta - \eta^n
= (\frac{\partial_y u_{\theta_n}^n}{\partial_y u^s})^{-1} \bigg[
(\beta -\frac{u_{yy}^s}{u_y^s}) +\beta\frac{\partial_y (u_{\theta_n}^n -u^s)}{u_y^s}
- \frac{\partial_{yy} (u_{\theta_n}^n -u^s)}{u_y^s} \bigg] \\[13pt]\hspace{1.1cm}

\geq (1+c_7\e)^{-1}(\delta_s - (C_{\eta}+\delta) c_7\e - c_8\e)
\geq\delta>0.
\end{array}
\end{equation}
Thus, the monotonicity conditions $(\ref{Sect5_Iteration_Monotonicity})$ for the Nash-Moser-H$\ddot{o}$rmander iteration are satisfied in the time interval $[0,T]$.

The approximate solution is constructed as
\begin{equation}\label{Sect5_Convergence_1}
\left\{\begin{array}{ll}
u^{n+1} = u^s + \tilde{u}^0 + \sum\limits_{j=0}^n \delta u^j, \\[8pt]
v^{n+1} = v^0 + \sum\limits_{j=0}^n \delta v^j,
\end{array}\right.
\end{equation}
then $u^{n+1}|_{y=0} = u^s|_{y=0} + \tilde{u}^0|_{y=0} + \sum\limits_{j=0}^n \delta u^j|_{y=0}$.

\vspace{-0.2cm}
When $k\leq \kk -2$, we have $\sum\limits_{j=0}^{+\infty} C_j\e\theta_j^{\max\{k-\kk,3-\kk\}} \triangle\theta_j \leq C$, then for $\forall n\geq 1$,\vspace{-0.2cm}
\begin{equation*}
\begin{array}{ll}
\|u^n - u^s\|_{\mathcal{A}_{\ell}^k}
\leq \|\tilde{u}^0\|_{\mathcal{A}_{\ell}^k} + \sum\limits_{j=0}^{\infty} \|\delta u^j\|_{\mathcal{A}_{\ell}^k} \leq C, \\[9pt]

\|\frac{\partial_y (u^n -u^s)}{\partial_y u^s}\|_{\mathcal{A}_{\ell}^k} \leq \|\frac{\partial_y \tilde{u}^0}{\partial_y u^s}\|_{\mathcal{A}_{\ell}^k} + \sum\limits_{j=0}^{\infty} \|\frac{\partial_y \delta u^j}{\partial_y u^s}\|_{\mathcal{A}_{\ell}^k} \leq C, \\[14pt]

\big\|\partial_y^m u^n|_{y=0} - \partial_y^m u^s|_{y=0}\big\|_{A^{k-[\frac{m+1}{2}]}}  \\[4pt]
\leq \big\|\partial_y^m\tilde{u}^0|_{y=0}\big\|_{A^{k-[\frac{m+1}{2}]}}
+ \sum\limits_{j=0}^{\infty} \big\|\partial_y^m\delta u^j|_{y=0}\big\|_{A^{k-[\frac{m+1}{2}]}} \leq C.
\end{array}
\end{equation*}

\vspace{-0.3cm}
When $k\leq \kk -3$, $\|v^n\|_{\mathcal{D}_0^k}\leq \|\tilde{v}^0\|_{\mathcal{D}_0^k} + \sum\limits_{j=0}^{\infty} \|\delta v^j\|_{\mathcal{D}_0^k} \leq C.$

By estimating the sums of remainder terms, it is easy to check that
$\{u^n -u^s\}$ and $\{\frac{\partial_y (u^{n} -u^s)}{\partial_y u^s}\}$ are Cauchy sequences in $\mathcal{A}_{\ell}^{\kk-2}$,
$\{v^n\}$ is a Cauchy sequence in $\mathcal{D}_{0}^{\kk-3}$,
$\{\partial_y^m u^n|_{y=0} - \partial_y^m u^s|_{y=0}\}$ is a Cauchy sequence in $A^{\kk-2-[\frac{m+1}{2}]}$.

Due to the completeness of functional spaces $\mathcal{A}_{\ell}^k$, $A^k$ and $\mathcal{D}_0^k$, there exist $u,v$ such that
\begin{equation}\label{Sect5_Convergence_4}
\begin{array}{ll}
\lim\limits_{n\rto\infty} u^n =u\in u^s+ \mathcal{A}_{\ell}^{\kk-2}, \\[8pt]
\lim\limits_{n\rto\infty} \frac{\partial_y (u^{n} -u^s)}{\partial_y u^s} =\frac{\partial_y (u -u^s)}{\partial_y u^s}\in\mathcal{A}_{\ell}^{\kk-2}, \\[8pt]
\lim\limits_{n\rto\infty} v^n =v\in\mathcal{D}_{0}^{\kk-3}, \\[8pt]
\lim\limits_{n\rto\infty} \partial_y^m u^n|_{y=0} =\partial_y^m u|_{y=0} \in \partial_y^m u^s|_{y=0}+  A^{\kk-2-[\frac{m+1}{2}]}.
\end{array}
\end{equation}

The approximate solution $(u^{n+1},v^{n+1})$ satisfies
\begin{equation}\label{Sect5_Convergence_8}
\left\{\begin{array}{ll}
\mathcal{P}(u^{n+1},v^{n+1}) = e_n + (1 -S_{\theta_n})\big(\sum\limits_{j=0}^{n-1}e_j + f^a\big), \\[10pt]
\partial_x u^{n+1} + \partial_y v^{n+1} =0, \\[9pt]
(\partial_y u^{n+1} -\beta u^{n+1})|_{y=0} =0,\quad v^{n+1}|_{y=0} =0, \\[9pt]
u^{n+1}|_{t=0} = u_0(x,y).
\end{array}\right.
\end{equation}

When $k\leq \kk-1$,\hspace{0.12cm} 
$\sum\limits_{j=0}^{+\infty}\|e_j\|_{\mathcal{A}_{\ell}^{k+1}} \leq \sum\limits_{j=0}^{+\infty} C_j\e^2\theta_j^{\max\{k+6-2\kk,7-2\kk\}} \triangle\theta_j <+\infty,$ \\
$\|f^a\|_{\mathcal{A}_{\ell}^{k+1}} \leq \|f^a\|_{\mathcal{A}_{\ell}^{k_0}}\leq C_0\e$, then as $n\rto +\infty$, \\
\begin{equation*}
\begin{array}{ll}
\|(1-S_{\theta})\big(\sum\limits_{j=0}^{n-1} e_j + f^a \big)\|_{\mathcal{A}_{\ell}^k} \leq \theta_n^{-1}
(\|f^a\|_{\mathcal{A}_{\ell}^{k+1}} + \sum\limits_{j=0}^{\infty}\|e_j\|_{\mathcal{A}_{\ell}^{k+1}}) \rto 0.
\end{array}
\end{equation*}
Thus, $\lim\limits_{n\rto \infty}\|\mathcal{P}(u^{n},v^{n})\|_{\mathcal{A}_{\ell}^{\kk-1}}=\|\mathcal{P}(u,v)\|_{\mathcal{A}_{\ell}^{\kk-1}}=0$, then the limits $u,v$ satisfy $\mathcal{P}(u,v)=0$ pointwisely.

Due to the uniform convergence of the variables and their derivatives, the limits $u$ and $v$ satisfy
$u_x + v_y =0, \ (\partial_y u -\beta u)|_{y=0} =0,\ v|_{y=0}=0$.
For any $n\geq 0$, $u^n|_{t=0} \equiv u_0^s +\tilde{u}_0 =u_0$, then $u|_{t=0} =u_0^s +\tilde{u}_0 =u_0$.
Thus, $(u,v)$ is a classical solution to the Prandtl system $(\ref{Sect1_PrandtlEq})$.

As $\beta\rto +\infty$, $(\ref{Sect5_Solution_Regularity})$ holds uniformly, because the bounds of the estimates in $(\ref{Sect5_Variables_Estimate})$, $(\ref{Sect5_Quantities_Estimate})$ and $(\ref{Sect5_Main_LowerOrder_Estimates_1})$ are independent of $\beta$.
Thus, $(\ref{Sect5_Solution_Regularity})$ holds when $\beta =+\infty$.

The estimates of $v_y,v_{yy},\partial_y^{j+1}v|_{y=0}, 0\leq j\leq 2k-2$ come from $u_x + v_y=0$.

Finally, we verify that $u$ satisfies the monotonicity conditions $(\ref{Sect1_Solution_Monotonicity})$ in $[0,T]$.

$\frac{\partial_y u}{\partial_y u^s} =  1 + \frac{\partial_y\tilde{u}_0 + \sum\limits_{m=0}^{\infty}\partial_y \delta u^m}{\partial_y u^s}
\geq 1 - c_5\e > 1-c_7\e \geq\frac{1}{2}>0$.

$u>0$ is due to $u|_{y=0} =\frac{1}{\beta}\partial_y u|_{y=0} >0$
and $\partial_y u>0$.

When $\beta\geq C_{\eta}+\delta$, $\beta - \frac{\partial_{yy}u}{\partial_y u}\geq\delta >0$. When $\beta< C_{\eta}+\delta$,
\vspace{-0.3cm}
\begin{equation*}
\begin{array}{ll}
\beta - \frac{\partial_{yy}u}{\partial_y u}
= (\frac{\partial_y u}{\partial_y u^s})^{-1} \bigg[
(\beta -\frac{u_{yy}^s}{u_y^s}) +\beta\frac{\partial_y\tilde{u}_0 + \sum\limits_{m=0}^{\infty}\partial_y \delta u^m}{u_y^s}
- \frac{\partial_{yy}\tilde{u}_0 + \sum\limits_{m=0}^{\infty}\partial_{yy} \delta u^m}{u_y^s} \bigg] \\[10pt]\hspace{1.42cm}

\geq (1+c_5\e)^{-1}(\delta_s - ((C_{\eta}+\delta)) c_5\e - c_6\e) \\[6pt]\hspace{1.42cm}
\geq (1+c_7\e)^{-1}(\delta_s - ((C_{\eta}+\delta)) c_7\e - c_8\e)
\geq\delta.
\end{array}
\end{equation*}
Thus, Theorem $\ref{Sect5_Convergence_Thm}$ is proved.
\end{proof}

Let $k_0=\kk+2$, the Prandtl system $(\ref{Sect1_PrandtlEq})$ loses $k+9$ orders of regularity.
If the monotonicity conditions $(\ref{Sect1_Solution_Monotonicity})$ are violated at $t=T$, the monotonicity conditions $(\ref{Sect5_Iteration_Monotonicity})$ for the Nash-Moser-H$\ddot{o}$rmander iteration must have been violated at $t\leq T$ and for some $n$, thus classical solutions to $(\ref{Sect1_PrandtlEq})$ can not be extended beyond $T$.

\subsection{The Derivatives on the Boundary}
In this subsection, we study the behavior of the derivatives on the boundary as $\beta\rto +\infty$, and discuss the regularity of the derivatives on the boundary.

\begin{lemma}\label{Sect5_Lemmas_Boundary}
Assume the conditions are the same with the conditions in Theorem $\ref{Sect5_Main_Thm}$, then
\begin{equation}\label{Sect5_Lemmas_Boundary_1}
\begin{array}{ll}
\sqrt{\beta}\big\|\delta u^n|_{y=0}\big\|_{A^k}
\leq C_n\e \theta_n^{\max\{k-\kk,3-\kk\}} \triangle\theta_n,\hspace{2.16cm} 0\leq k\leq k_0,\\[7pt]

\frac{1}{\sqrt{\beta}}\big\|\partial_y\delta u^n|_{y=0}\big\|_{A^k}
\leq C_n\e \theta_n^{\max\{k-\kk,3-\kk\}} \triangle\theta_n,\hspace{1.82cm} 0\leq k\leq k_0.
\end{array}
\end{equation}
\end{lemma}

\begin{proof}
Assume the estimates $(\ref{Sect5_Lemmas_Boundary_1})$ hold for $n=0,2,\cdots, j-1$, then we prove $(\ref{Sect5_Lemmas_Boundary_1})$ hold for $n=j$.

\noindent
Estimate of $(\ref{Sect5_Lemmas_Boundary_1})_1$: since $\delta u^j|_{y=0} = \frac{\partial_y u_{\theta_j}^j}{\beta - \eta^j}|_{y=0} w^j|_{y=0}$,
\vspace{-0.2cm}
\begin{equation}\label{Sect5_Lemmas_Boundary_2}
\begin{array}{ll}
\sqrt{\beta}\|\delta u^j|_{y=0}\|_{A^{k}} = \sqrt{\beta}\big\|\frac{\partial_y u_{\theta_j}^j}{\beta - \eta^j}|_{y=0} w^j|_{y=0}\big\|_{A^{k}}\\[8pt]

\lem \frac{\beta + C_{\eta}}{\max\{\delta,\beta-C_{\eta}\}}\|\partial_y u_{\theta_j}^j\|_{L^{\infty}}
\frac{\big\|w^j|_{y=0}\big\|_{A^{k}}}{\sqrt{\beta + C_{\eta}}} \\[8pt]\quad
+ \frac{\beta + C_{\eta}}{\max\{\delta,\beta-C_{\eta}\}}\frac{\big\|w^j|_{y=0}\big\|_{L_{t,x}^{\infty}}}{\sqrt{\beta + C_{\eta}}}
(\big\|\partial_y u_{\theta_j}^j|_{y=0} -\partial_y u^s|_{y=0}\big\|_{A^k}
+ \big\|\partial_{t,x}^k \partial_y u^s|_{y=0}\big\|_{L_{t,x}^{\infty}}) \\[10pt]\quad

+ \frac{\beta + C_{\eta}}{\max\{\delta,\beta-C_{\eta}\}}\frac{\big\|w^j|_{y=0}\big\|_{L_{t,x}^{\infty}}}{\sqrt{\beta + C_{\eta}}}
(\big\|\eta^j|_{y=0} -\bar{\eta}^j|_{y=0}\big\|_{A^k} + \big\|\partial_{t,x}^k\bar{\eta}^j|_{y=0}\big\|_{L_{t,x}^{\infty}}) \\[12pt]

\lem C_j\e \theta_j^{\max\{k-\kk,3-\kk\}} \triangle\theta_j
+ C_j\e \theta_j^{3-\kk} \triangle\theta_j(C_s + C_j\e\theta_j^{\max\{k+1-\kk+\dt,0\}}) \\[8pt]\quad
+ C_j\e \theta_j^{3-\kk} \triangle\theta_j(C_j\e\theta_j^{\max\{k+2-\kk+\dt,0\}} + C_j + C_j\e\theta_j^{\max\{k+3-\kk+\dt,0\}})
\\[8pt]

\lem C_j\e \theta_j^{\max\{k-\kk,3-\kk\}} \triangle\theta_j, \quad 0\leq k\leq k_0.
\end{array}
\end{equation}
Note that $\frac{\beta + C_{\eta}}{\max\{\delta,\beta-C_{\eta}\}}<+\infty$ for $0<\delta_{\beta}\leq \beta\leq +\infty$.

\vspace{0.2cm}
\noindent
Estimate of $(\ref{Sect5_Lemmas_Boundary_1})_2$: since $\partial_{t,x}^{k}(\partial_y\delta u^j-\beta\delta u^j)|_{y=0}=0$, \\[3pt]
$\frac{1}{\sqrt{\beta}}\big\|\partial_y\delta u^j|_{y=0}\big\|_{A^{k}} = \sqrt{\beta}\big\|\delta u^j|_{y=0}\big\|_{A^{k}} \leq C_j\e \theta_j^{\max\{k-\kk,3-\kk\}} \triangle\theta_j,
\quad 0\leq k\leq k_0.$
\end{proof}

\vspace{-0.3cm}
Based on the estimates $(\ref{Sect5_Lemmas_Boundary_1})$, we have the following results:
\begin{lemma}\label{Sect5_Lemmas_Rate}
For $0\leq k\leq \kk-2$, as $\beta\rto +\infty$,
\begin{equation}\label{Sect5_Lemmas_Rate_1}
\begin{array}{ll}
\partial_t^j u^s|_{y=0}=O(\frac{1}{\beta}), \quad  \partial_t^j\partial_y u^s|_{y=0}=O(1), \quad 0\leq j\leq k,  \\[7pt]
\big\|u|_{y=0} - u^s|_{y=0}\big\|_{A^k}=O(\frac{1}{\sqrt{\beta}}),\quad \big\|\partial_y u|_{y=0} - \partial_y u^s|_{y=0}\big\|_{A^k}=O(\sqrt{\beta}),
\end{array}
\end{equation}
\end{lemma}
\begin{proof}
When $0\leq k\leq \kk-2$,
\begin{equation*}
\begin{array}{ll}
\sqrt{\beta}\big\|u|_{y=0} - u^s|_{y=0}\big\|_{A^k}
\leq \frac{1}{\sqrt{\delta_{\beta}}}\beta\big\|\tilde{u}^0|_{y=0}\big\|_{A^k}
+ \sum\limits_{j=0}^{\infty}\sqrt{\beta}\big\|\delta u^j|_{y=0}\big\|_{A^k} \\[7pt]

\lem \frac{1}{\sqrt{\delta_{\beta}}}\big\|\partial_y\tilde{u}^0|_{y=0}\big\|_{A^k}
+ \sum\limits_{j=0}^{\infty} C_j\e\theta_j^{\max\{k-\kk,3-\kk\}}
\leq C(T),
\end{array}
\end{equation*}
\vspace{-0.3cm}
and
\begin{equation*}
\begin{array}{ll}
\frac{1}{\sqrt{\beta}}\big\|\partial_y u|_{y=0} - \partial_y u^s|_{y=0}\big\|_{A^k}
\leq \frac{1}{\sqrt{\beta}}\big\|\partial_y\tilde{u}^0|_{y=0}\big\|_{A^k}
+ \sum\limits_{j=0}^{\infty}\frac{1}{\sqrt{\beta}}\big\|\partial_y\delta u^j|_{y=0}\big\|_{A^k} \\[7pt]

\lem \frac{1}{\sqrt{\delta_{\beta}}}\big\|\partial_y\tilde{u}^0|_{y=0}\big\|_{A^k}
+ \sum\limits_{j=0}^{\infty} C_j\e\theta_j^{\max\{k-\kk,3-\kk\}}
\leq C(T).
\end{array}
\end{equation*}

Obviously, $u|_{y=0} =O(\frac{1}{\sqrt{\beta}})\rto 0$, pointwisely, as $\beta\rto +\infty$.
\end{proof}

\begin{remark}\label{Sect5_Example_Remark}
In the $y$ direction, there is a similarity between the Prandtl system $(\ref{Sect1_PrandtlEq})$ and the heat equation $(\ref{Sect2_HeatEq})$: the solutions have lower regularities on the boundary than in the interior.

(i) In Theorem $\ref{Sect2_Theorem}$ for the heat equation $(\ref{Sect2_HeatEq})$, the initial data $u_0^s$ satisfies $\|u_0^s-1\|_{L^2} + \|u_0^s\|_{\dot{\mathcal{C}}_{\ell}^k} \leq C$, then \vspace{-0.2cm}
\begin{equation*}
\begin{array}{ll}
\|u^s-1\|_{L^2}^2 +
\sum\limits_{0\leq k_1+[\frac{k_2+1}{2}] \leq k}
 \int\limits_0^t \int\limits_0^{\infty} <y>^{2\ell}\big|\partial_t^{k_1}\partial_y^{k_2+1} u^s|_{y=0} \big|^2 \,\mathrm{d}y\,\mathrm{d}\tilde{t} \\[12pt]

+ \sum\limits_{j=1}^k\int\limits_0^t \big|\partial_t^{j} u^s|_{y=0} \big|^2 \,\mathrm{d}\tilde{t}
+ \sum\limits_{j=0}^{k-1}\int\limits_0^t\big|\partial_t^{j}\partial_y u^s|_{y=0}\big|^2 \,\mathrm{d}\tilde{t}
\leq C(T),
\end{array}
\end{equation*}
$\sum\limits_{j=0}^{k}\int\limits_0^t\big|\partial_t^{j}\partial_y u^s|_{y=0}\big|^2 \,\mathrm{d}\tilde{t}
\leq \beta\cdot C(T)<+\infty$ is due to the Robin boundary condition, it is not satisfied when $\beta\rto +\infty$.

If the initial data $u_0^s$ satisfies $\|u_0^s-1\|_{L^2} + \|u_0^s\|_{\dot{\mathcal{C}}_{\ell}^k}
+ \|\partial_y u_0^s\|_{\dot{\mathcal{C}}_{\ell}^k}\leq C$, then \vspace{-0.2cm}
\begin{equation*}
\begin{array}{ll}
\|u^s-1\|_{L^2}^2 +
\sum\limits_{1\leq k_1+[\frac{k_2+1}{2}] \leq k+1}
 \int\limits_0^t \int\limits_0^{\infty} <y>^{2\ell}\big|\partial_t^{k_1}\partial_y^{k_2} u^s|_{y=0} \big|^2 \,\mathrm{d}y\,\mathrm{d}\tilde{t} \\[12pt]

+ \sum\limits_{j=1}^k\int\limits_0^t \big|\partial_t^{j} u^s|_{y=0} \big|^2 \,\mathrm{d}\tilde{t}
+ \sum\limits_{j=0}^{k}\int\limits_0^t\big|\partial_t^{j}\partial_y u^s|_{y=0}\big|^2 \,\mathrm{d}\tilde{t}
\leq C(T).
\end{array}
\end{equation*}

(ii) In Theorem $\ref{Sect1_Main_Thm}$ for the Prandtl system $(\ref{Sect1_PrandtlEq})$, the regularities in the y-direction can be improved very slightly, though it can not in the t,x-directions.
The estimate $\|\partial_y w^n\|_{\mathcal{A}_{\ell}^k} \leq C_n\e\theta_n^{\max\{k-\kk,3-\kk\}}\triangle\theta_n$ is attainable from the refined estimates of mixed derivatives in Section 3, then we have the extra regularities:
\begin{equation*}
\begin{array}{ll}
\partial_{yy}(u-u^s) \in \mathcal{A}_{\ell}^k([0,T]\times\mathbb{R}_{+}^2), \qquad
\partial_{yyy} v \in \mathcal{A}_{\ell}^{k-1}([0,T]\times\mathbb{R}_{+}^2), \\[7pt]

\partial_y^{j+1} u|_{y=0}- \partial_y^{j+1}u^s|_{y=0} \in A^{k-[\frac{j+1}{2}]}([0,T]\times\mathbb{R}), \quad 0\leq j\leq 2k, \\[7pt]

\partial_y^{j+2} v|_{y=0} \in A^{k-1-[\frac{j+1}{2}]}([0,T]\times\mathbb{R}), \quad 0\leq j\leq 2k-2.
\end{array}
\end{equation*}
However, in the y-direction, the solutions still have lower regularities on the boundary than in the interior.
\end{remark}

\section{Uniqueness and Stability of Classical Solutions to the Nonlinear Prandtl Equations with Robin Boundary Condition}
In this section, we prove the uniqueness and stability of classical solutions to the Prandtl system  $(\ref{Sect1_PrandtlEq})$.

\begin{theorem}\label{Sect6_Unique_Stable_Thm}
The classical solution to $(\ref{Sect1_PrandtlEq})$ is unique and stable with respect to the initial data in the following sense: assume $(u^1,v^1)$ and $(u^2,v^2)$ be two classical solutions to the problem $(\ref{Sect1_PrandtlEq})$ with the initial data $u_0^1$ and $u_0^2$ respectively, where $u_0^1(x,y)$ and $u_0^2(x,y)$ satisfy the conditions in Theorem $\ref{Sect5_Main_Thm}$, then for any $0\leq\delta_{\beta}\leq \beta<+\infty$, there is a constant $C(T,\e,u_0^s)\in (0,+\infty)$ such that
\begin{equation}\label{Sect6_Unique_Stable}
\begin{array}{ll}
\|u^1-u^2\|_{\mathcal{A}_{\ell}^p([0,T]\times\mathbb{R}_{+}^2)}
+ \|\frac{\partial_y (u^1-u^2)}{\partial_y u^s}\|_{\mathcal{A}_{\ell}^p([0,T]\times\mathbb{R}_{+}^2)} \\[8pt]\quad
+\|v^1-v^2\|_{\mathcal{D}_0^{p-1}([0,T]\times\mathbb{R}_{+}^2)}
+ \sum\limits_{j=0}^{2p}\big\|\partial_y^j u^1|_{y=0} - \partial_y^j u^2|_{y=0} \big\|_{A^{p-[\frac{j+1}{2}]}([0,T]\times\mathbb{R})}
\\[14pt]

\leq C(T,\e,u_0^s)\Big\|\partial_y (\frac{u_0^1-u_0^2}{\partial_y(u_0^1 + u_0^2)})\Big\|_{\mathcal{A}_{\ell}^{p}(\mathbb{R}_{+}^2,t=0)} \\[13pt]\quad
+ \frac{C(T,\e,u_0^s)}{\max\{\sqrt{\beta -C_{\eta}},\sqrt{\delta}\}} \Big\|\partial_y (\frac{u_0^1-u_0^2}{\partial_y(u_0^1 + u_0^2)})|_{y=0}\Big\|_{A^p(\mathbb{R},t=0)} \ ,
\end{array}
\end{equation}
for all $p\leq \kk-4,\ \kk\geq 7$.
\end{theorem}

\begin{proof}
Denote
\begin{equation}\label{Sect6_Define_Variables}
\left\{\begin{array}{ll}
u=u^1-u^2,\hspace{0.5cm} \tilde{u}=\frac{u^1+u^2}{2}, \\[7pt]
v=v^1-v^2,\hspace{0.6cm} \tilde{v}=\frac{v^1+v^2}{2},
\end{array}\right.
\end{equation}

Since $u^1,u^2$ satisfy the monotonicity conditions $(\ref{Sect1_Solution_Monotonicity})$,
then $\tilde{u}$ satisfies $(\ref{Sect1_Solution_Monotonicity})$. Thus,
$(\tilde{u},\tilde{v})$ satisfies the following conditions:
\begin{equation}\label{Sect6_Background_Condition}
\left\{\begin{array}{ll}
\tilde{u}>0,\quad \tilde{u}_y>0,\quad \beta -\frac{\tilde{u}_{yy}}{\tilde{u}_y}\geq\delta >0,\quad \forall y\in [0,+\infty), \\[8pt]
\tilde{u}_x + \tilde{v}_y =0, \\[8pt]
(\tilde{u}_y -\beta\tilde{u})|_{y=0} =0,\quad \tilde{v}|_{y=0}=0, \\[8pt]
\lim\limits_{y\rto +\infty} \tilde{u}(t,x,y) =1.
\end{array}\right.
\end{equation}
Note that in $(\ref{Sect6_Background_Condition})$, $(\tilde{u}_y -\beta\tilde{u})|_{y=0} =0$ will not be used because $(\ref{Sect3_Background_Condition_2})$ does not need this condition.

$(u,v)$ satisfies the following IBVP:
\begin{equation}\label{Sect6_Difference_Eq}
\left\{\begin{array}{ll}
u_t + \tilde{u} u_x + \tilde{v} u_y + u \tilde{u}_x +  v \tilde{u}_y - u_{yy} = 0, \\[8pt]
u_x + v_y =0, \\[8pt]
(\partial_y u - \beta u)|_{y=0} =0,\quad  v|_{y=0} =0,\\[8pt]
 \lim\limits_{y\rto +\infty} u(t,x,y) =0, \\[8pt]
u|_{t\leq 0} = u_0^1 - u_0^2.
\end{array}\right.
\end{equation}

Set $w = (\frac{u}{\tilde{u}_y})_y $, then $w$ satisfies the following IBVP:
\begin{equation}\label{Sect6_Difference_VorticityEq}
\left\{\begin{array}{ll}
w_t + (\tilde{u}w)_x + (\tilde{v}w)_y - 2(\eta w)_y
- \big(\zeta \int\limits_y^{+\infty} w(t,x,\tilde{y}) \,\mathrm{d}\tilde{y} \big)_y - w_{yy} = 0, \\[12pt]

\frac{w_t}{\beta- \eta|_{y=0}} + \frac{(\tilde{u} w)_x}{\beta- \eta|_{y=0}}
- \Big(w_y + 2\eta|_{y=0} w + \tilde{f} + \zeta|_{y=0}\int\limits_0^{+\infty} w(t,x,\tilde{y}) \,\mathrm{d}\tilde{y} \Big) \\[10pt]\hspace{0.5cm}

+ \frac{w\tilde{\zeta}|_{y=0}}{(\beta- \eta|_{y=0})^2} =0,\quad y=0, \\[12pt]

w|_{t\leq 0} = w_0 := 2\partial_y(\frac{u_0^1-u_0^2}{\partial_y (u_0^1 + u_0^2)}) \ ,
\end{array}\right.
\end{equation}
where $\eta,\ \bar{\eta},\ \zeta,\ \tilde{\zeta}$ are defined as $(\ref{Sect3_New_Variables})$.

\vspace{0.2cm}
When $p\leq \kk -4$, we use $\lambda_{p}(\tilde{u})$ to denote $\lambda_{p}$ defined with respect to $\tilde{u}=\frac{u^1+u^2}{2}$, we can calculate directly that $\lambda_{p}(\tilde{u}) <+\infty$ by the regularities of $u^1,u^1|_{y=0},u^2,u^2|_{y=0},v^1,v^2$ and the definition of $\lambda_p$, namely $(\ref{Sect3_Notions})$. Then by Theorem $\ref{Sect3_Main_Estimate_ZeroForce_Thm}$, there exists a constant $C(T,\lambda_{p}(\tilde{u}))>0$, such that
\begin{equation}\label{Sect6_Apply_Thm}
\begin{array}{ll}
\|w\|_{\mathcal{A}_{\ell}^{p}}
\leq C(T,\lambda_{p}(\tilde{u}))
\Big(\|w|_{t=0}\|_{\mathcal{A}_{\ell}^{p}(\mathbb{R}_{+}^2,t=0)} \\[8pt]\hspace{3.5cm}
+ \frac{1}{\max\{\sqrt{\beta -C_{\eta}},\sqrt{\delta}\}} \big\|w|_{y=0,t=0}\big\|_{A^p(\mathbb{R},t=0)}\Big)\ .
\end{array}
\end{equation}

Similarly, by the estimates for
\begin{equation}\label{Sect6_Solve_UW}
\begin{array}{ll}
u^1 -u^2 =  -\partial_y\tilde{u}\int\limits_y^{\infty} w(t,x,\tilde{y}) \,\mathrm{d}\tilde{y}, \\[8pt]
\partial_y^m u^1|_{y=0} - \partial_y^m u^2|_{y=0}  
= -\sum\limits_{m_1+m_2=m}\partial_y^{m_1+1}\tilde{u}|_{y=0}\int\limits_0^{\infty} \partial_y^{m_2}w(t,x,\tilde{y}) \,\mathrm{d}\tilde{y},
\end{array}
\end{equation}
 we get
\begin{equation}\label{Sect6_Difference_VorticityEq_Estimate}
\begin{array}{ll}
\|u^1-u^2\|_{\mathcal{A}_{\ell}^p([0,T]\times\mathbb{R}_{+}^2)}
+\|v^1-v^2\|_{\mathcal{D}_0^{p-1}([0,T]\times\mathbb{R}_{+}^2)}
+ \|\frac{\partial_y (u^1-u^2)}{\partial_y u^s}\|_{\mathcal{A}_{\ell}^p([0,T]\times\mathbb{R}_{+}^2)} \\[6pt]\quad
+ \sum\limits_{j=0}^{2p}\big\|(\partial_y^j u^1- \partial_y^j u^2)|_{y=0}\|_{A^{p-[\frac{j+1}{2}]}([0,T]\times\mathbb{R})} \\[12pt]
\lem C(T,\lambda_{p}(\tilde{u}))\|w\|_{\mathcal{A}_{\ell}^{p}([0,T]\times\mathbb{R}_{+}^2)},
\end{array}
\end{equation}

Combining $(\ref{Sect6_Apply_Thm})$ and $(\ref{Sect6_Difference_VorticityEq_Estimate})$, we get $(\ref{Sect6_Unique_Stable})$, which implies the stability of
the nonlinear Prandtl equations with Robin boundary condition. Note that $C(T,\lambda_p(\tilde{u}))\leq C(T,\e,u^s)\leq C(T,\e,u_0^s)$.

\vspace{0.2cm}
Next, we prove the uniqueness of the nonlinear Prandtl equations with Robin boundary condition, assume that $u_0^1 = u_0^2$.

In $\mathbb{R}_{+}^2$, $u_0^1-u_0^2\equiv 0$ implies that $\partial_y(u_0^1-u_0^2)\equiv 0$, then
\begin{equation}\label{Sect6_Uniqueness_1}
\begin{array}{ll}
\partial_y (\frac{u_0^1-u_0^2}{\partial_y(u_0^1 + u_0^2)}) = (u_0^1-u_0^2)[\frac{1}{(u_0^1 + u_0^2)_y}]_y + \partial_y(u_0^1-u_0^2)\frac{1}{(u_0^1 + u_0^2)_y} \equiv 0.
\end{array}
\end{equation}

On the boundary $\{y=0\}$, $\partial_y(u_0^1-u_0^2) = \beta(u_0^1-u_0^2) \equiv 0$, then
\begin{equation}\label{Sect6_Uniqueness_2}
\begin{array}{ll}
\partial_y (\frac{u_0^1-u_0^2}{\partial_y(u_0^1 + u_0^2)})|_{y=0} = (u_0^1-u_0^2)|_{y=0}\Big([\frac{1}{(u_0^1 + u_0^2)_y}]_y + \beta\frac{1}{(u_0^1 + u_0^2)_y}\Big) =0.
\end{array}
\end{equation}

By $(\ref{Sect6_Uniqueness_1})$ and $(\ref{Sect6_Uniqueness_2})$, the right hand side of $(\ref{Sect6_Unique_Stable})$ equals zero if $u_0^1 = u_0^2$.

Thus, Theorem $\ref{Sect6_Unique_Stable_Thm}$ is proved.
\end{proof}

\begin{remark}\label{Sect6_Extra_Stability_Remark}
Due to the estimate of $\|\partial_y w^n\|_{\mathcal{A}_{\ell}^k}$, the stability results in the y-direction can be improved very slightly, 
though those in the t,x-directions can not be improved. Namely, we have the stability of $\|\partial_{yy}(u-u^s)\|_{\mathcal{A}_{\ell}^p([0,T]\times\mathbb{R}_{+}^2)}$ and
$\sum\limits_{j=0}^{2p}\|\partial_y^{j+1} u|_{y=0}- \partial_y^{j+1}u^s|_{y=0}\|_{A^{p-[\frac{j+1}{2}]}([0,T]\times\mathbb{R})}$.
\end{remark}

\bibliographystyle{siam}
\addcontentsline{toc}{section}{References}
\bibliography{FuzhouWu_PrandtlEq_RobinBC}

\end{document}